\pdfoutput=1
\documentclass[12pt,centertags]{amsart}

\usepackage{fullpage}
\usepackage{latexsym}

\usepackage{amsmath}
\usepackage{amsthm}
\usepackage{amssymb}
\usepackage{graphicx}
\usepackage{amsfonts}
\usepackage{hyperref}
\usepackage{bbm}
\usepackage{mathrsfs}
\usepackage{accents}
\usepackage{mathtools}	

\usepackage{color}
\usepackage[usenames,dvipsnames]{xcolor}

\usepackage{mparhack}	
\usepackage{comment}	
\usepackage{marvosym}	
\usepackage{enumitem}	
\usepackage{array}		    
\usepackage{todonotes}

\numberwithin{equation}{section}
\theoremstyle{plain}
\newtheorem{Prop}{Proposition}[section]
\newtheorem{Thm}[Prop]{Theorem}
\newtheorem*{Thm*}{Theorem}
\newtheorem{Lem}[Prop]{Lemma}
\newtheorem{Cor}[Prop]{Corollary}

\theoremstyle{definition}
\newtheorem{Def}[Prop]{Definition}

\theoremstyle{remark}
\newtheorem{Rem}[Prop]{Remark}

\def\vint_#1{\mathchoice%
          {\mathop{\kern 0.2em\vrule width 0.6em height 0.69678ex
depth -0.58065ex
                  \kern -0.8em \intop}\nolimits_{\kern -0.4em#1}}%
          {\mathop{\kern 0.1em\vrule width 0.5em height 0.69678ex
depth -0.60387ex
                  \kern -0.6em \intop}\nolimits_{#1}}%
          {\mathop{\kern 0.1em\vrule width 0.5em height 0.69678ex
              depth -0.60387ex
                  \kern -0.6em \intop}\nolimits_{#1}}%
          {\mathop{\kern 0.1em\vrule width 0.5em height 0.69678ex
depth -0.60387ex
                  \kern -0.6em \intop}\nolimits_{#1}}}
\def\vintslides_#1{\mathchoice%
          {\mathop{\kern 0.1em\vrule width 0.5em height 0.697ex depth -0.581ex
                  \kern -0.6em \intop}\nolimits_{\kern -0.4em#1}}%
          {\mathop{\kern 0.1em\vrule width 0.3em height 0.697ex depth -0.604ex
                  \kern -0.4em \intop}\nolimits_{#1}}%
          {\mathop{\kern 0.1em\vrule width 0.3em height 0.697ex depth -0.604ex
                  \kern -0.4em \intop}\nolimits_{#1}}%
          {\mathop{\kern 0.1em\vrule width 0.3em height 0.697ex depth -0.604ex
                  \kern -0.4em \intop}\nolimits_{#1}}}

\newcommand{\intav}{\vint}
\newcommand{\aveint}[2]{\mathchoice
          {\mathop{\kern 0.2em\vrule width 0.6em height 0.69678ex
depth -0.58065ex
                  \kern -0.8em \intop}\nolimits_{\kern -0.45em#1}^{#2}}%
          {\mathop{\kern 0.1em\vrule width 0.5em height 0.69678ex
depth -0.60387ex
                  \kern -0.6em \intop}\nolimits_{#1}^{#2}}%
          {\mathop{\kern 0.1em\vrule width 0.5em height 0.69678ex
depth -0.60387ex
                  \kern -0.6em \intop}\nolimits_{#1}^{#2}}%
          {\mathop{\kern 0.1em\vrule width 0.5em height 0.69678ex
depth -0.60387ex
                  \kern -0.6em \intop}\nolimits_{#1}^{#2}}}

\DeclareMathOperator{\supp}{supp}
\DeclareMathOperator{\diam}{diam}
\DeclareMathOperator{\dv}{div}

\DeclareMathOperator{\osc}{osc}
\DeclareMathOperator{\sign}{sign}
\DeclareMathOperator{\dist}{dist}

\newcommand{\set}[2]{\left\{#1 : #2\right\}}
\newcommand{\mns}{\setminus}
\newcommand{\N}{\mathbb{N}}
\newcommand{\R}{\mathbb{R}}

\newcommand{\del}{\partial}

\newcommand{\eps}{\varepsilon}
\newcommand{\inv}[1]{{#1}^{-1}}
\newcommand{\dx}{\, dx}
\newcommand{\loc}{\text{\rm loc}}
\newcommand{\Om}{\Omega}
\newcommand{\om}{\omega}
\newcommand{\inp}[2]{\big\langle #1,#2\big\rangle}
\newcommand{\gr}{\nabla}

\newcommand{\hh}{\mathbb{H}}
\newcommand{\X}{\mathfrak{X} }
\newcommand{\Xu}{\X u}
\newcommand{\XX}{\X\X}
\newcommand{\dvh}{\dv_{H}}
\newcommand{\A}{\mathcal{A}}
\newcommand{\Aeps}{\mathcal{A\,_\eps}}
\newcommand{\F}{\textsc{F}}
\newcommand{\weight}{\F\left(|\X u|\right)}
\newcommand{\hw}[2]{HW^{#1,#2}}
\newcommand{\0}{0_{\,\!_\hh} }
\newcommand{\dhh}{d_{\hh^n}}
\newcommand{\normh}[1]{\|#1\|_{\hh^n}}
\newcommand{\Q}{\textsc{Q}}
\newcommand{\tu}{\tilde{u}}

\title[$C^{1,\alpha}$-Regularity of Quasilinear equations on the Heisenberg Group]
{$C^{1,\alpha}$-Regularity of Quasilinear equations on the Heisenberg Group}

\author[Shirsho Mukherjee]{Shirsho Mukherjee}
\address[S.\ Mukherjee]{Department of Mathematics, 
University of Bologna, Piazza di Porta San Donato 5,
  40126 - Bologna (BO), Italy.}
\email{shirsho.mukherjee2@unibo.it}
\thanks{
The author was supported by the European Union’s Seventh Framework Programme “Metric Analysis For Emergent Technologies (MAnET)”, Marie Curie Actions-Initial Training Network, under Grant Agreement No. 607643.}

\begin{document}

\begin{abstract}
In this article, we reproduce results of classical regularity theory of  quasilinear elliptic equations in the divergence form, in the setting of 
Heisenberg Group. The considered cases encompass a very wide class of equations with isotropic growth conditions that are generalizations of the $p$-Laplacian type equation and also include equations with polynomial or exponential type growth. Some more general conditions have also been explored. 
\end{abstract}

\date{\today}
\maketitle
\setcounter{tocdepth}{1}
\phantomsection

\section{Introduction}\label{sec:Introduction}

Regularity theory for weak solutions of second order quasilinear elliptic equations in the Euclidean spaces, has been well-developed over a long 
period of time since the pioneering work of De Giorgi \cite{DeG} and has   
involved significant contributions of many authors. For more details on this topic, we refer to 
\cite{Simon,Dib,Tolk,Gia-Giu--min,Gia-Giu--div,Uhlen,Evans,Lewis}, etc. and references therein. A comprehensive study of the subject can be found in
the nowadays classical 
books by Gilbarg-Trudinger \cite{Gil-Tru}, Ladyzhenskaya-Ural'tseva \cite{Lady-Ural} and Morrey \cite{Mor}.

The goal of this paper is to obtain regularity results in the setting of Heisenberg Group $\hh^n$, that are previously known in the Euclidean setting.  
We consider the equation 
\begin{equation}\label{eq:eq}
\Q u = \dvh A(x,u,\X u)+ B(x,u,\X u) =  0
\end{equation}
in a domain $\Om\subset \hh^n$ for any $n\geq 1$, where 
$\X u =(X_1u,\ldots,X_{2n}u)$ is the horizontal gradient of a function 
$u:\Om\to \R$ and $\dvh $ is the horizontal divergence of a vector field  
(see Section \ref{sec:Preliminaries} for details). Here 
$A:\Om\times\R\times \R^{2n}\to \R^{2n}$ and 
$B:\Om\times\R\times \R^{2n}\to \R$ are given locally integrable functions.  
We also assume that $A$ is differentiable and 
the $(2n\times 2n)$ matrix $D_p \,A(x,z,p) = (\del A_i(x,z,p)/\del p_j)_{ij}$ is symmetric for every $x\in\Om, z\in\R$ and 
$p=(p_1,\ldots,p_{2n})\in\R^{2n}$. Thus, the results of  
this setting can also be applied to minimizers of a variational integral 
$$ I(u) = \int_\Om f(x,u,\X u)\dx $$
for a smooth scalar function $f:\Om\times\R\times\R^{2n}\to\R$; the 
Euler-Lagrange equation corresponding to the functional $I$, would be an equation of 
the form \eqref{eq:eq}. The equations in settings similar to ours, are often  
referred as sub-elliptic equations.

In addition to $A$ and $B$, 
we consider a $C^1$-function $g:[0,\infty)\to [0,\infty)$ 
also as given data, which satisfies $g(0)=0$ and 
there exists constants $ g_0 \geq \delta \geq 0$ such that the 
following holds, 
\begin{equation}\label{eq:g prop}
 \delta   \leq \frac{tg'(t)}{g(t)} \leq g_0 \quad\text{for all}\ t>0.
\end{equation}
The function $g$ shall be used in the hypothesis of 
growth and ellipticity conditions satisfied by $A$ and $B$, as given below. The condition 
\eqref{eq:g prop} appears in the work of Lieberman \cite{Lieb--gen}, in the 
Euclidean setting. In the case of Heisenberg Groups, a special class of quasilinear equations with growth conditions involving 
\eqref{eq:g prop}, has 
been recently studied in \cite{Muk0}. We remark that the special case 
$g(t) = t^{p-1}$ for $1<p<\infty$, would correspond to equations with $p$-laplacian type growth. For a more detailed discussion on the relevance of the condition 
\eqref{eq:g prop} and more examples of such function $g$, we refer to 
\cite{Lieb--gen,Mar2,Bar}, etc.   

The study of regularity theory for sub-elliptic equations goes back to the 
fundamental work of H\"ormander \cite{Hor}. We refer to 
\cite{Cap--reg,C-D-G,Cap-Garo,Foglein,Dom-Man--reg,Dom-Man--cordes,Marchi,Man-Min,Dom} and references therein, for earlier results on regularity 
of weak solutions of quasilinear equations. 

The structure conditions for the equation \eqref{eq:eq} used in this paper, 
have been introduced in \cite{Lieb--gen}, which are generalizations of the 
so called natural conditions for elliptic equations in divergence form; these have been 
extensively studied by Ladyzhenskaya-Ural'tseva in \cite{Lady-Ural} for equations in the Euclidean setting. 
The first structure condition is as follows. 

Given some non-negative constants $a_1, a_2, a_3,b_0,b_1$ and $\chi$, 
we assume that $A$ and $B$ satisfies  
\begin{equation}\label{eq:str1}
\begin{aligned}
 &\inp{A(x,z,p)}{p}\geq |p|g(|p|)
 -a_1\,g\bigg(\frac{|z|}{R}\bigg)\frac{|z|}{R}-g(\chi)\chi;\\
   &|A(x,z,p)| \leq a_2\, g(|p|) + a_3 \,g\bigg(\frac{|z|}{R}\bigg)
   +g(\chi) ;\\
    &|B(x,z,p)|\leq \frac{1}{R}
  \bigg[b_0\,g(|p|) + b_1\,g\bigg(\frac{|z|}{R}\bigg) + g(\chi) \bigg],
\end{aligned}
\end{equation}
where $(x,z,p)\in\Om\times\R\times\R^{2n}$ and 
$0<R<\frac{1}{2}\diam(\Om)$. Similar growth conditions have been 
considered previously in \cite{Gil-Tru},\cite{Lady-Ural} and \cite{Tru--harnack}  for the special case $g(t) = t^{\alpha-1}$ for $\alpha>1$. 

For weak solutions of equation \eqref{eq:eq} with the above structure conditions, the appropriate domain is the 
Horizontal Orlicz-Sobolev space $HW^{1,G}(\Om)$ (see Section \ref{sec:Preliminaries} for the definition), where $G(t)= \int^t_0 g(s)ds$. 
The following is the first result of this paper. 
\begin{Thm}\label{thm:c0alpha}
 Let $u\in HW^{1,G}(\Om)\cap L^\infty(\Om)$ be a  
 weak solution of the equation \eqref{eq:eq}, with $G(t)= \int^t_0 g(s)ds$ 
 and $|u|\leq M$ in $\Om$. Suppose the structure condition 
 \eqref{eq:str1} holds for some $\chi\geq 0,\,0<R\leq R_0$ and a function $g$ satisfying \eqref{eq:g prop} with 
 $\delta>0$, then there exists $c>0$ and $\alpha\in(0,1)$ dependent on 
 $n,\delta,g_0,a_1, a_2, a_3,b_0M,b_1$ such that 
 $u\in C^{\,0,\alpha}_\loc(\Om)$ and 
 \begin{equation}\label{eq:oscest}
  \osc_{B_r}u \leq c\left(\frac{r}{R}\right)^\alpha
\Big(\osc_{B_R}u + \chi R\Big),
\end{equation}
whenver $B_{R_0}\subset\subset\Om$ and $B_r,B_R$ are concentric to $B_{R_0}$ with 
$0<r<R\leq R_0$.
\end{Thm}
The above theorem follows as a consequence of Harnack inequalities,  
Theorem \ref{thm:supharnack} and Theorem \ref{thm:infharnack} in 
Section \ref{sec:Holder}. Similar Harnack inequalities in the sub-elliptic 
setting, has also been shown in \cite{C-D-G} for the special case of polynomial 
type growth. 
The proof of these are standard imitations of the 
corresponding classical results due to Serrin \cite{Serrin}, see also 
\cite{Tru--harnack, Lieb--gen}. 

Theorem \ref{thm:c0alpha} is necessary for our second result,
the $C^{1,\alpha}$-regularity of weak solutions. This is new and 
relies on some recent development in \cite{Muk0}. The structure conditions 
considered for this, are as follows. 

Given the constants $L,L'\geq 1$ and $\alpha\in(0,1]$, we assume that 
the following holds,  
\begin{equation}\label{eq:str2}
\begin{aligned}
  &\frac{g(|p|)}{|p|}\,|\xi|^2
 \leq \,\inp{D_p\,A(x,z,p)\,\xi}{\xi}\,\leq L\,\frac{g(|p|)}{|p|}\,|\xi|^2;\\
  &|A(x,z,p)-A(y,w,p)|\,\leq
 \,L'\big(1+g(|p|)\big)\Big( |x-y|^\alpha + |z-w|^\alpha\Big);\\
  &|B(x,z,p)|\,\leq\, L'\big(1+g(|p|)\big)|p|, 
\end{aligned}
\end{equation}
for every $x,y\in\Om,\, z,w\in [-M_0,M_0]$ and $p,\xi\in\R^{2n}$, where 
$M_0>0$ is another given constant. The following theorem is the second 
result of this paper. 
\begin{Thm}\label{thm:c1alpha}
 Let $u\in HW^{1,G}(\Om)\cap L^\infty(\Om)$ be a  
 weak solution of the equation \eqref{eq:eq}, with $G(t)= \int^t_0 g(s)ds$ 
 and $\|u\|\leq M_0$ in $\Om$. Suppose the structure condition 
 \eqref{eq:str2} holds for some $L,L'\geq 1, \alpha\in (0,1]$ and a function $g$ satisfying \eqref{eq:g prop} with $\delta>0$, then there exists a constant 
 $\beta = \beta(n,\delta,g_0,\alpha, L)\in (0,1)$ such that 
 $u\in C^{1,\beta}_\loc(\Om)$ and for any open $\Om'\subset\subset\Om$, 
 we have 
 \begin{equation}\label{eq:gradest}
 |\X u|_{C^{\,0,\beta}(\Om',\R^{2n})} \leq 
 C\big(n,\delta,g_0,\alpha, L,L', M_0,g(1),\dist(\Om',\del\Om)\big). 
\end{equation}
\end{Thm}

Pertaining to the growth conditions involving \eqref{eq:g prop}, 
local Lipshcitz continuity for the class of 
equations of the form $\dvh\A(\X u)=0$, 
has been shown in \cite{Muk0}. As a follow up, here we show 
the $C^{1,\alpha}$-regularity for this case as well, with a robust gradient 
estimate unlike \eqref{eq:gradest}. 
\begin{Thm}\label{thm:holder1}
Let $u\in HW^{1,G}(\Om)$ be a weak solution of the equation 
$\dvh\A(\X u)=0$, where $\A:\R^{2n}\to\R^{2n}$, the matrix $D\A$ is symmetric and the following structure condition holds, 
\begin{equation}\label{eq:str3}
 \begin{aligned}
 \frac{g(|p|)}{|p|}\,|\xi|^2 
 \leq \,&\inp{D\A(p)\,\xi}{\xi}\leq L\,\frac{g(|p|)}{|p|}\,|\xi|^2;\\
 &|\A(p)|\leq L\,g(|p|). 
\end{aligned} 
\end{equation}
for every $p,\xi\in\R^{2n}$, $L\geq1$ is a given constant and 
$g$ satisfies \eqref{eq:g prop} with $\delta>0$. Then 
$\X u$ is locally H\"older continuous and there exists 
$\sigma=\sigma(n,g_0,L)\in (0,1)$ and $c=c(n,\delta,g_0,L)>0$ such that 
for any $B_{r_0}\subset\Om$ and $0<r<r_0/2$, we have 
\begin{equation}\label{eq:holder1}
\max_{1\leq l\leq 2n}\intav_{B_r}G(|X_l u - \{X_l u\}_{B_r}|)\dx \,\leq\, c\Big(\frac{r}{r_0}\Big)^\sigma
\intav_{B_{r_0}} G(|\X u|)\dx .
\end{equation} 
\end{Thm}
The proof of the above theorem, follows similarly along the line of that in \cite{Muk-Zhong}. 
It involves Caccioppoli type estimates 
of the horizontal and vertical vector fields along with the use of an integrability 
estimate of \cite{Zhong} and   
a double truncation of \cite{Tolk} and \cite{Lieb--bound}.

We remark that the spaces $C^{\,0,\alpha}$ and $C^{1,\alpha}$ considered in 
this paper, are in the sense of Folland-Stein \cite{Folland-Stein--book}. In other words, the spaces are defined with respect to the 
homogeneous metric of the Heisenberg Group, see Section \ref{sec:Preliminaries} for details. No assertions are made concerning 
the regularity of the vertical derivative.

This paper is organised as follows. In Section \ref{sec:Preliminaries}, we 
provide a brief review on Heisenberg Group and Orlicz spaces. Then in 
Section \ref{sec:Holder}, first we prove a global maximum principle exploring 
some generalised growth conditions along the lines of \cite{Lieb--gen}; then 
we prove the Harnack inequalities, thereby leading to the proof of Theorem 
\ref{thm:c0alpha}. The whole of Section \ref{sec:LocalH} is devoted to the 
proof of Theorem \ref{thm:holder1}. Finally in Section \ref{sec:c1alpha}, the proof of Theorem \ref{thm:c1alpha} is provided and some possible 
extensions of the structure conditions are discussed.

\section{Preliminaries}\label{sec:Preliminaries}
In this section, we fix the notations used and provide a brief introduction of the  
Heisenberg Group $\hh^n$. Also, we provide some essential facts on 
Orlicz spaces and the Horizontal Sobolev spaces, which are required for 
the purpose of this setting.  
\subsection{Heisenberg Group}\label{subsec:Heisenberg Group}\noindent
\\
Here we provide the definition and properties of Heisenberg group  
that would be useful in this paper.  
For more details, we 
refer the reader to \cite{Bonfig-Lanco-Ugu},\cite{C-D-S-T}, etc. 
\begin{Def}\label{def:H group}
 For $n\geq 1$, the \textit{Heisenberg Group} denoted by $\hh^n$, is identified to  the Euclidean space 
$\R^{2n+1}$ with the group operation 
\begin{equation}\label{eq:group op}
 x\cdot y\, := \Big(x_1+y_1,\ \dots,\ x_{2n}+y_{2n},\ t+s+\frac{1}{2}
\sum_{i=1}^n (x_iy_{n+i}-x_{n+i}y_i)\Big)
\end{equation}
for every $x=(x_1,\ldots,x_{2n},t),\, y=(y_1,\ldots,y_{2n},s)\in {\mathbb H}^n$.
\end{Def}
Thus, $\hh^n$ with the group operation \eqref{eq:group op} forms a 
non-Abelian Lie group, whose left invariant vector fields corresponding to the
canonical basis of the Lie algebra, are
\[ X_i=  \partial_{x_i}-\frac{x_{n+i}}{2}\partial_t, \quad
X_{n+i}=  \partial_{x_{n+i}}+\frac{x_i}{2}\partial_t,\] 
for every $1\leq i\leq n$ and the only
non zero commutator $ T= \partial_t$. 
We have 
\begin{equation}\label{eq:comm}
  [X_i\,,X_{n+i}]=  T\quad 
  \text{and}\quad [X_i\,,X_{j}] = 0\ \ \forall\ j\neq n+i.
\end{equation}
We call $X_1,\ldots, X_{2n}$ as \textit{horizontal
vector fields} and $T$ as the \textit{vertical vector field}. 
For a scalar function $ f: \hh^n \to \R$, we denote
$\X f  = (X_1f,\ldots, X_{2n}f)$ and $
\XX f =  (X_i(X_j f))_{i,j} $
as the \textit{Horizontal gradient} and \textit{Horizontal Hessian}, 
respectively. 
From \eqref{eq:comm}, we have the following
trivial but nevertheless, an important inequality 
$
|Tf|\leq 2|\XX f|$. 
For a vector valued function 
$F = (f_1,\ldots,f_{2n}) : \hh^n\to \R^{2n}$, the 
\textit{Horizontal divergence} is defined as 
$$ \dvh (F)  =  \sum_{i=1}^{2n} X_i f_i .$$
The Euclidean gradient of a 
function $g: \R^{k} \to \R$, shall be denoted by
$\gr g=(D_1g,\ldots,D_{k} g)$ and the Hessian matrix by $D^2g$.

The \textit{Carnot-Carath\`eodory metric} (CC-metric) is defined 
as the length 
of the shortest 
horizontal curves, connecting two points.  
This is equivalent to the \textit{Kor\`anyi metric}, denoted as  
$ \dhh(x,y)= \normh{y^{-1}\cdot x}$,  
where the Kor\`anyi norm for $x=(x_1,\ldots,x_{2n}, t)\in \hh^n$ is 
\begin{equation}\label{eq:norm}
 \normh{x} :=  \Big(\ \sum_{i=1}^{2n} x_i^2+ |t|\ \Big)^\frac{1}{2}.
\end{equation}
Throughout this article we use CC-metric balls denoted by $B_r(x) = \set{y\in\hh^n}{d(x,y)<r}$ for $r>0$ and $ x \in \hh^n $. However, by virtue of the equivalence 
of the metrics, all assertions for CC-balls can be restated to Kor\`anyi balls. 

The Haar measure of $\hh^n$ is just the Lebesgue 
measure of $\R^{2n+1}$. For a measurable set $E\subset \hh^n$, we denote 
the Lebesgue measure as $|E|$. For an integrable function $f$, we denote 
$$ \{f\}_E = \intav_E f\dx = \frac{1}{|E|} \int_E f\dx .$$
The Hausdorff dimension with respect to the metric $d$ is also the homogeneous dimension of the group $\mathbb H^n$, which shall be denoted as 
$Q=2n+2$, 
throughout this paper. Thus, for any CC-metric ball $B_r$, we have that 
$|B_r| = c(n)r^Q$. 

For $ 1\leq p <\infty$, the \textit{Horizontal Sobolev space} $HW^{1,p}(\Omega)$ consists
of functions $u\in L^p(\Omega)$ such that the distributional horizontal gradient $\X u$ is in $L^p(\Omega\,,\R^{2n})$.
$HW^{1,p}(\Omega)$ is a Banach space with respect to the norm
\begin{equation}\label{eq:sob norm}
  \| u\|_{HW^{1,p}(\Omega)}= \ \| u\|_{L^p(\Omega)}+\| \X u\|_{L^p(\Omega,\R^{2n})}.
\end{equation}
We define $HW^{1,p}_{\loc}(\Omega)$ as its local variant and 
$HW^{1,p}_0(\Omega)$ as the closure of $C^\infty_0(\Omega)$ in 
$HW^{1,p}(\Omega)$ with respect to the norm in \eqref{eq:sob norm}. 
The Sobolev Embedding theorem has the following version in the setting of 
Heisenberg group (see \cite{C-D-G},\cite{C-D-S-T}).
\begin{Thm}[Sobolev Embedding]\label{thm:sob emb}
Let $B_r\subset {\mathbb H}^n$ and $1<q<Q$. 
For all $u \in HW^{1,q}_0(B_r)$, there exists constant $c=c(n,q)>0$ such that
\begin{equation}\label{eq:sob emb}
\left(\intav_{B_r}| u|^{\frac{Q q}{Q-q}}\, dx\right)^{\frac{Q-q}{Q q}}
\leq\, c\,r \left(\intav_{B_r}| \X u|^q\, dx\right)^{\frac 1 q}.
\end{equation}
\end{Thm}

H\"older spaces with respect to homogeneous metrics have appeared in Folland-Stein \cite{Folland-Stein--book} and therefore, are sometimes called 
are known as Folland-Stein classes and denoted by $\Gamma^{\alpha}$ or 
$\Gamma^{\,0,\alpha}$ in some literature. However, here we maintain the classical notation and define 
\begin{equation}\label{def:holderspace}
C^{\,0,\alpha}(\Om) =
\set{u\in L^\infty(\Om)}{|u(x)-u(y)|\leq c\,d(x,y)^\alpha\ \forall\ x,y\in \Om}
\end{equation} 
for $0<\alpha \leq 1$, 
which are Banach spaces with the norm 
\begin{equation}\label{eq:holder norm}
\|u\|_{C^{\,0,\alpha}(\Om)}
=  \|u\|_{L^\infty(\Om)}+ \sup_{x,y\in\Om} 
\frac{|u(x)-u(y)|}{d(x,y)^\alpha}.
\end{equation}
These have standard extensions to classes $C^{k,\alpha}(\Om)$ 
for $k\in \N$, which consists of functions having horizontal 
derivatives up to order $k$ in $C^{\,0,\alpha}(\Om)$. The local 
counterparts are denoted as $C^{k,\alpha}_\loc(\Om)$. Now, 
the definition of Morrey and Campanato spaces in sub-elliptic setting 
differs in different texts. Here, we adopt the definition similar to 
the classical one.  

For any domain $\Om\subset \hh^n$ and $\lambda>0$, we define the {\it Morrey space} as 
\begin{equation}\label{def:morrey}
\mathcal M^{1,\lambda}(\Om) = 
\set{u\in L^1_\loc(\Om)}{\int_{B_r} |u|\dx <c\,r^\lambda\ \forall\ B_r\subset\Om, r>0}
\end{equation}
and the  {\it Campanato space} as 
\begin{equation}\label{def:camp}
\mathcal L^{1,\lambda}(\Om) = 
\set{u\in L^1_\loc(\Om)}{\int_{B_r} \big|u-\{u\}_{B_r}\big|\dx <c\,r^{\lambda}\ \forall\ B_r\subset\Om,r>0},
\end{equation}
where in both definitions $B_r$ represents balls with metric $d$. These spaces are Banach spaces and have properties similar to the classical spaces in the Euclidean setting. We shall use the fact that for every $0<\alpha<1$ and 
$Q=2n+2$, we have  
\begin{equation}\label{eq:holdcamp}
\mathcal L^{1,Q+\alpha}(\Om) \subset C^{\,0,\alpha}(\Om),
\end{equation}
where the inclusion is to be understood as taking continuous representatives. 
For details on classical Morrey and Campanato spaces, we refer to 
\cite{Kuf-O-F} and for the sub-elliptic setting we refer to 
\cite{C-D-S-T}. 

\subsection{Orlicz-Sobolev Spaces}\label{subsec:Orlicz-Sobolev Spaces}\noindent
\\
In this subsection, 
we recall some basic facts on Orlicz-Sobolev functions, which shall be necessary later. Further details can be found in textbooks e.g. \cite{Kuf-O-F},\cite{Rao-Ren}. 
\begin{Def}[Young function]\label{def:young}
 If $ \psi :[0,\infty) \to [0,\infty) $ is an non-decreasing, left continuous function with 
$\psi(0) = 0 $ and $ \psi(s)>0 $ for all $ s >0 $, then any function 
$\Psi : [0,\infty) \to [0,\infty] $ of the form 
\begin{equation}\label{eq:young}
  \Psi(t) = \int^t_0 \psi(s)\, ds  
\end{equation}
is called a \textit{Young function}. 
A continuous Young function $\Psi : [0,\infty) \to [0,\infty) $ 
satisfying  
$\Psi(t) = 0$ iff $t = 0,\ \lim_{t\to \infty}\Psi(t)/t  =  \infty$ and $ 
\lim_{t \to 0}\Psi(t)/t=  0 $, is called 
\textit{N-function}. 
\end{Def}
There are several different definitions available in various references.  
However, within a slightly restricted range of functions (as in our case), 
all of them are equivalent. We refer to the book of Rao-Ren \cite{Rao-Ren}, 
for a more general discussion.
\begin{Def}[Conjugate]\label{def:conj}
The \textit{generalised inverse} of a montone function $\psi$ is defined as 
$ \psi^{-1}(t)  :=  \inf\{s \geq 0\ |\ \psi(s) >t \} $.  
Given any Young function $\Psi(t)  =  \int^t_0 \psi(s) ds $, 
its 
\textit{conjugate} function $ \Psi^* : [0,\infty) \to [0,\infty]$ is defined as
\begin{equation}\label{eq:young comp}
 \Psi^* (s)  :=  \int^s_0 \psi^{-1}(t)\, dt 
\end{equation}
and $(\Psi,\Psi^*) $ is called a \textit{complementary pair}, which 
is \textit{normalised} if 
$ \Psi(1) +\Psi^*(1) = 1$. 
\end{Def}
A Young function $\Psi$ is convex, increasing, left continuous and  
satisfies $ \Psi(0) = 0 $ and $ \lim_{t\to\infty}\Psi(t) = \infty $. 
The generalised inverse of $\Psi$ is right continuous, increasing and coincides with the usual inverse 
when $\Psi$ is continuous and strictly increasing.
In general, the  
inequality 
\begin{equation}\label{eq:inv ineq}
\Psi(\Psi^{-1}(t))\leq t \leq \Psi^{-1}(\Psi(t))
\end{equation}
is satisfied for all $ t \geq 0 $ and equality holds when 
$\Psi(t)$ and $\Psi^{-1}(t) \in (0,\infty)$. 
 It is also evident that that the conjugate function $\Psi^*$ is also a Young 
 function, $\Psi^{**} = \Psi$
and for any constant $c>0$, we have $ (c\,\Psi)^*(t) = c\,\Psi^*(t/c)$. 

Here are two standard examples of 
complementary pair of Young functions.  
\begin{enumerate}
 \item $ \Psi(t) = t^p/p$ and $ \Psi^*(t) = t^{p^*}/p^*$ when 
 $ 1< p, p^*< \infty$ and $ 1/p +1/p^* = 1 $.
 \item $\Psi(t) = (1+t)\log(1+t)-t$ and $ \Psi^*(t) = e^t-t-1 $.
 \end{enumerate}
The following Young's inequality is well known. We refer to \cite{Rao-Ren} for a proof. 
\begin{Thm}[Young's Inequality]\label{thm:YI}
 Given a Young function $\Psi(t) =  \int^t_0 \psi(s) ds $, we have 
 \begin{equation}\label{eq:YI}
 st \,\leq\,   \Psi(s)  + \Psi^*(t)
\end{equation}
for all $s,t >0$ and equality holds if and only if $ t = \psi(s) $ or $ s = \inv{\psi}(t)$.
\end{Thm}
A Young function $\Psi$ is called \textit{doubling} if there exists 
a constant $C_2>0 $ such that for all $t \geq 0$, we have 
$\Psi(2t)\leq C_2\, \Psi(t)$. By virtue of \eqref{eq:g prop}, the structure function $g$ is doubling with the  doubling constant $C_2 = 2^{g_0}$ and hence, we restrict to Orlicz spaces of doubling functions. 
\begin{Def}\label{def:Orl space}
Let $ \Omega\subset\R^m $ be Borel and $\nu$ be a $\sigma$-finite measure on $\Om$. For a doubling Young function $\Psi$, the 
\textit{Orlicz space} $L^\Psi(\Omega,\nu)$ is defined as the vector space generated by the set
$ \{u: \Omega \to \R\ |\ u\ \text{measurable},\ \int_\Omega 
\Psi(|u|)\,d\nu < \infty\} $. The space is equipped with the following 
\textit{Luxemburg norm} 
\begin{equation}\label{eq:lux norm}
 \|u\|_{L^\Psi(\Omega,\nu)}  :=  \inf\Big\{ k >0 : \int_\Omega 
\Psi\left(\frac{|u|}{k}\right) \,d\nu \leq 1\Big\} 
\end{equation}
If $\nu $ is the Lebesgue measure, the space 
is denoted by $L^\Psi(\Omega)$ and any $u \in L^\Psi(\Omega)$ is called 
a $\Psi$-integrable function.
\end{Def}
The function 
$u \mapsto \|u\|_{L^\Psi(\Omega,\nu)}$ is lower semi continuous and $L^\Psi(\Omega,\nu)$ is a
Banach space with the norm in \eqref{eq:lux norm}. 
The following theorem is a generalised version of H\"older's inequality, which 
follows easily from   
the Young's inequality \eqref{eq:YI}, see \cite{Rao-Ren} or \cite{Tuo}. 
\begin{Thm}[H\"older's Inequality]\label{thm:gen holder}
 For every $ u \in L^\Psi(\Omega,\nu)$ and $ v \in L^{\Psi^*}(\Omega,\nu)$, we have 
 \begin{equation}\label{eq:gen holder}
 \int_\Om |uv|\,d\nu\leq 2\, \|u\|_{L^\Psi(\Omega,\nu)}\|v\|_{L^{\Psi^*}(\Omega,\nu)}
\end{equation}
\end{Thm}
\begin{Rem}
The factor $2$ on the right hand side of the above, can be dropped if $(\Psi, \Psi^*)$ is normalised and one is replaced by 
$\Psi(1)$ in the definition \eqref{eq:lux norm} of Luxemburg norm. 
\end{Rem}
The \textit{Orlicz-Sobolev 
space} $W^{1,\Psi}(\Om)$ 
can be defined similarly by $ L^\Psi$ norms of the function and 
its gradient, see \cite{Rao-Ren}, that resembles $W^{1,p}(\Om)$. 
But here for $\Om\subset \hh^n$, we require the notion of  
\textit{Horizontal Orlicz-Sobolev spaces}, analoguous to the horizontal Sobolev spaces defined in the previous subsection.
\begin{Def}\label{def:HOS space}
 We define the space 
 $HW^{1,\Psi}(\Omega) = \{u\in L^\Psi(\Omega)\ |\ \Xu\in L^\Psi(\Omega,\R^{2n})\}$ for an open set $\Om\subset \hh^n$ and a  
 doubling Young function $\Psi$, along with the norm 
 $$ \|u\|_{HW^{1,\Psi}(\Omega)} :=  
\|u\|_{L^\Psi(\Omega)}+ \|\Xu\|_{L^\Psi(\Omega, \R^{2n})}.$$
The spaces $HW^{1,\Psi}_\loc(\Om),\ HW^{1,\Psi}_0(\Om)$ are  defined, 
similarly as earlier.
\end{Def}
We remark that, all these notions can be defined for a general metric space, equipped with a doubling measure. We refer to \cite{Tuo} for the details.

The following theorem, so called $(\Psi,\Psi)$-Poincar\'{e} inequality, has been 
proved (see Proposition 6.23 in \cite{Tuo}) in the setting of a general metric space with a doubling 
measure and metric upper gradient. We provide the statement in the setting 
of Heisenberg Group.
\begin{Thm}\label{thm:psipsi}
Given any doubling N-function $\Psi$ with doubling constant $c_2>0$, every 
$u\in HW^{1,\Psi}(\Om)$ satisfies the following inequality for every $B_r\subset\Om$ and some $c= c(n,c_2)>0$,
\begin{equation}\label{eq:psipsi}
\intav_{B_r} \Psi\bigg( \frac{|u-\{u\}_{B_r}|}{r}\bigg)\dx 
\leq c \intav_{B_r}\Psi(|\X u|)\dx.
\end{equation}
\end{Thm}
In case of $\Psi(t)= t^p$, the inequality is referred as $(p,p)$-Poincar\'{e} inequality. 
The following corrollary follows easily from \eqref{eq:psipsi} and the $(1,1)$-Poincar\'{e} inequality on $\hh^n$. 
\begin{Cor}\label{cor:psisob}
Given a convex doubling N-function $\Psi$ with doubling constant $c_2>0$, there exists 
$c=c(n,c_2)$ such that for every $B_r\subset\Om$ and 
$u\in HW^{1,\Psi}(\Om)\cap HW^{1,1}_0(\Om)$, we have 
\begin{equation}\label{eq:psisob}
\intav_{B_r} \Psi\bigg( \frac{|u|}{r}\bigg)\dx 
\leq c \intav_{B_r}\Psi(|\X u|)\dx.
\end{equation}
\end{Cor}
Given a domain $\Om\subset\hh^n$, using \eqref{eq:psisob} and arguments 
with chaining method (see \cite{Haj-Kos}), it is also possible to show that for 
$u,\Psi$ and $c=c(n,c_2)>0$ as in Corrollary \ref{cor:psisob}, we have
\begin{equation}\label{eq:psidiam}
\intav_{\Om} \Psi\bigg( \frac{|u|}{\diam(\Om)}\bigg)\dx 
\leq c \intav_{\Om}\Psi(|\X u|)\dx.
\end{equation}

Now we enlist some important properties of the  
function $g$ that satisfies \eqref{eq:g prop}. 
\begin{Lem}\label{lem:gandG prop} Let $ g\in C^1([0,\infty)) $ be a function that satisfies \eqref{eq:g prop} for some constant $g_0>0$ and 
$g(0) = 0$. If $ G(t) = \int^t_0 g(s)ds $, then the following holds. 
\begin{align}
\label{eq:gG1}&(1) \  \ G \in C^2([0,\infty))\ \text{is convex}\,;\\
\label{eq:gG2}&(2)\ \ tg(t)/(1+g_0)\leq G(t)\leq tg(t)\ \ \forall\ \ t\geq 0;\\
\label{eq:gG3}&(3)\  \ g(s)\leq g(t)\leq (t/s)^{g_0}g(s) \ \ \forall\ \ 0\leq s<t;\\
\label{eq:gG4}&(4) \ \ G(t)/t \ \text{is an increasing function}\ \ \forall\ \ t>0;\\
\label{eq:gG5}&(5) \ \ tg(s)\leq tg(t) + sg(s)\ \ \forall\ \ t,s \geq 0.
\end{align}
\end{Lem}
The proof is trivial (see Lemma 1.1 of \cite{Lieb--gen}), so we omit it.
Notice that \eqref{eq:gG3} implies that $g$ is increasing and 
doubling, with $g(2t) \leq 2^{g_0} g(t)$ and 
\begin{equation}\label{eq:gdoub}
 \min\{1,\alpha^{g_0}\}g(t)\leq g(\alpha t)\leq 
 \max\{1,\alpha^{g_0}\}g(t)\quad\text{for all}\ \alpha, t\geq 0. 
\end{equation}
Since $G$ is convex, an easy application of Jensen's inequality yields 
\begin{equation}\label{eq:jenap}
\intav_\Om G(|w-\{w\}_\Om|)\dx \,\leq\, c(g_0)\min_{k\in\R}\intav_\Om G(|w-k|)\dx\quad \forall\, w\in L^G(\Om)  
\end{equation}
All the above properties hold even if $\delta=0$ in \eqref{eq:g prop} and 
they are purposefully kept that way. However, the properties corresponding to 
$\delta>0$, shall be required in some situations. For this case, 
\eqref{eq:gG2} and \eqref{eq:gG3} becomes 
\begin{align}
\label{eq:gG2'}& tg(t)/(1+g_0)\leq G(t)\leq tg(t)/(1+\delta)
\ \ \forall\ \ t\geq 0;\\
\label{eq:gG3'}& (t/s)^\delta g(s)\leq g(t)\leq (t/s)^{g_0}g(s) \ \ \forall\ \ 0\leq s<t,
\end{align}
and hence $t\mapsto g(t)/t^{g_0}$ is decreasing and 
$t\mapsto g(t)/t^\delta$ is increasing. 

\section{H\"older continuity of weak solutions}\label{sec:Holder}
In this section, we show that weak solutions of quasilinear equations in the 
Heisenberg Group satisfy the Harnack inequalities, which leads to the 
H\"older continuity, thereby proving Theorem \ref{thm:c0alpha}. The techniques are standard, based on 
appropriate modifications of similar results in the Euclidean setting, by 
Trudinger \cite{Tru--harnack} and Lieberman \cite{Lieb--gen}.
 
On a domain $\Om\subset \hh^n$, we consider the prototype 
quasilinear operator in divergence form
\begin{equation}\label{eq:qop}
\Q u = \dvh A(x,u,\X u)+ B(x,u,\X u) 
\end{equation}
throughout this paper, where $A:\Om\times\R\times \R^{2n}\to \R^{2n}$ and 
$B:\Om\times\R\times \R^{2n}\to \R$ are given functions. 
Appropriate additional hypothesis on structure conditions satisfied by $A$ 
and $B$, shall be assumed in the following subsections, accordingly as required. 

Here onwards, throughout this paper, we fix the notations 
\begin{equation}\label{eq:GandF}
\F(t)  :=  g(t)/t \quad \text{and}\quad G(t)  :=  \int^t_0 g(s)\,ds. 
\end{equation} 

We remark that the conditions chosen for $A$, always ensure some 
sort of ellipticity for the operator \eqref{eq:qop} and the existence of 
weak solutions $u\in HW^{1,G}(\Om)$ for $\Q u=0$ is always assured. Any pathological situation, where this does not hold, is avoided. 
\subsection{Global Maximum principle}\noindent
\\
Given weak solution $u\in HW^{1,G}(\Om)$ for $\Q u=0$, here we show global 
$L^\infty$ estimates of $u$ under appropriate boundary conditions. The method and techniques are adaptations of similar classical results in \cite{Lieb--gen} for quasilinear equations in the Euclidean setting.

Here, we 
assume that $u$ satisfies the boundary condition $u-u_0\in HW^{1,G}_0(\Om)$ for some $u_0\in L^\infty(\bar\Om)$. 
In addition, we assume that there exists $b_0>0$ and $M\geq \|u_0\|_{L^\infty}$ 
such that  
\begin{align}
 \label{f1}\inp{A(x,z,p)}{p}&\geq |p|g(|p|)-f_1(|z|);\\
   \label{f2}zB(x,z,p) &\leq b_0\inp{A(x,z,p)}{p} + f_2(|z|),
\end{align}
holds for all $x\in\Om,\ |z|\geq M$ and $p\in\R^{2n}$, where $f_1,f_2$ and $g$ 
are non-negative increasing functions. 
Also, 
we require 
$\inp{A(x,u,\Xu)}{\X u} \in L^1(\Om)$ and $u\in L^\infty(\Om)$. The first 
condition \eqref{f1}, can be viewed as a weak ellipticity condition.

Additional conditions on $f_1$ and $f_2$, yields apriori integral estimates 
as in the following lemma. Similar results in Euclidean setting, can be found in  \cite{Gil-Tru} and \cite{Lady-Ural}. 
\begin{Lem}\label{lem:Gbound1}
Let $ u \in HW^{1,G}(\Om) $ be a weak solution of $\Q u=0$ in $\Om$ along with 
the conditions \eqref{f1} and \eqref{f2} and $u-u_0\in HW^{1,G}_0(\Om)$. If 
the functions $f_1,f_2$ and $g$ satisfy
\begin{align}
\label{Gcond1} (1)&\ tg(t) \leq a_1G(t);\\
\label{Gcond2} (2)&\ tg(t) f_1(Rt)+G(t)f_2(Rt) \leq a_1G(t)^2,
\end{align}
for some $a_1\geq 1,R>0$ and every $t>M/R$, then there exists $c(n)>0$ such that 
for $Q=2n+2$ and 
$c= c(n)[(1+a_1)(1+2b_0)]^Q$, we have 
\begin{equation}\label{eq:Gbound1}
\sup_\Om \,G(|u|/R) \leq \max\Big\{\frac{c}{R^Q}\int_\Om G(|u|/R)\dx\,,\, 
(1+a_1)G(M/R)\Big\}.
\end{equation}
\end{Lem}
\begin{proof}
The proof is similar to that of Lemma 2.1 in \cite{Lieb--gen} (see also Lemma 10.8 in \cite{Gil-Tru}) and follows from standard Moser's iteration. We provide a brief outline. 

Note that, we can assume $|u|\geq M$ without loss of generality, as otherwise we are done; we provide the proof for $u\geq M$, the proof for $u\leq -M$ is similar. 
The test function $\varphi = h(u)$ is used for the equation $\Q u=0$, where 
letting $G=G(|u|/R)$ and $\tau=G(M/R)$, we choose 
$$h(u)= u G^\beta \big|\big(1-\tau/G\big)^+\big|^{Q\beta+1},$$ for $\beta\geq 2b_0$ and 
$Q=2n+2$. Thus $\varphi/u\geq 0$ and $\varphi =0$ on $\del\Om$,  
since $M\geq \|u_0\|_{L^\infty}$.
Hence, applying $\varphi$ as a test function 
and using \eqref{f2}, we get 
\begin{equation}\label{GB1}
\begin{aligned}
\int_\Om \inp{A(x,u,\Xu)}{\X \varphi} \dx &= \int_\Om B(x,u,\Xu)\varphi \dx\\
&\leq\int_\Om \Big[b_0\inp{A(x,u,\Xu)}{\Xu} + f_2(|u|)\Big]\frac{\varphi}{u}\dx.
\end{aligned}
\end{equation}
Note that $\X \varphi = h'(u)\X u$ and we have 
$$ h'(u) = \frac{\varphi}{u}
+\bigg[ \beta\Big(1-\frac{\tau}{G}\Big)+(Q\beta+1)\frac{\tau}{G} \bigg] 
G^{\beta-1} \big|(1-\tau/G)^+\big|^{Q\beta}
g\bigg(\frac{|u|}{R}\bigg)\frac{u}{R},$$
which implies $h'(u) \geq (\beta+1)\varphi/u$ and 
$h'(u) \leq a_1(Q+2)(\beta+1) |(1-\tau/G)^+|^{Q\beta}G^\beta$ from 
\eqref{Gcond1}.
For every $\beta\geq 2b_0$, we obtain that
\begin{equation}\label{GB2}
  \begin{aligned}
   \frac{1}{2}\int_\Om h'(u)g(|\X u|)|\X u|\dx
   &\leq \int_\Om \Big(h'(u)- b_0\varphi/u\Big)\Big[\inp{A(x,u,\Xu)}{\Xu} + f_1(|u|)\Big]\dx\\
    &\leq \int_\Om \Big[f_2(|u|)\varphi/u +\Big(h'(u)\ - b_0\varphi/u\Big)f_1(|u|) \Big]\dx, 
  \end{aligned}
 \end{equation}
where we have used $h'(u) \geq 2b_0\, \varphi/u $ and \eqref{f1} for the first inequality and \eqref{GB1} for the second inequality of the above. 
From \eqref{GB2} and \eqref{Gcond2}, we obtain
\begin{equation}\label{GB3}
\frac{1}{2}\int_\Om h'(u)g(|\X u|)|\X u|\dx
\leq a_1(\beta+1)(2n+4)\int_\Om \big|(1-\tau/G)^+\big|^{Q\beta} G^{\beta+1}\dx.
\end{equation}
Now, leting $w = \psi(G)= \frac{1}{2}G^{\beta+1} |(1-\tau/G)^+|^{Q\beta+1}$, 
note that $|\psi'(G)|\leq h'(u)g(|u|/R)|\X u|/R$. Then, we use \eqref{eq:gG5} of Lemma \ref{lem:gandG prop} with $t= |\X u|$ and $s=|u|/R$, to obtain
\begin{equation}\label{GB4}
\begin{aligned}
\int_\Om |\X w|\dx \leq \int_\Om h'(u)g\bigg(\frac{|u|}{R}\bigg)\frac{|\X u|}{R}\dx&\leq \int_\Om h'(u)\bigg[ g\bigg(\frac{|u|}{R}\bigg)\frac{|u|}{R^2}
+g(|\X u|)\frac{|\X u|}{R}\bigg]\dx\\
& \leq \frac{c(n)}{R}a_1(\beta+1)\int_\Om \big|(1-\tau/G)^+\big|^{Q\beta} G^{\beta+1}\dx
\end{aligned}
\end{equation}
for some $c(n)>0$, 
where for the last inequality of the above, we have used \eqref{GB3} and 
\eqref{Gcond1}. Recalling Sobolev's inequality \eqref{eq:sob emb} with 
$q=1$, we have 
$$ \Big(\int_\Om w^\kappa \dx\Big)^{1/\kappa}\leq c(n)\int_\Om |\X w|\dx $$
for $\kappa= Q/(Q-1)= (2n+2)/(2n+1)$. Combining this with \eqref{GB4}, we 
obtain
\begin{equation}\label{GB5}
\Big(\int_\Om \big|(1-\tau/G)^+\big|^{\kappa(Q\beta+1)}G^{\kappa(\beta+1)}\dx\Big)^{1/\kappa}
\leq \frac{c(n)}{R}a_1(\beta+1)\int_\Om \big|(1-\tau/G)^+\big|^{Q\beta} G^{\beta+1}\dx
\end{equation}
which can be reduced to $\|v\|_{L^{\kappa\gamma}(\Om,\mu)}
\leq (\gamma/\gamma_0)^{1/\gamma}\|v\|_{L^{\gamma}(\Om,\mu)}$, where 
$v = G|(1-\tau/G)^+|^Q$, $\gamma=\beta+1, \gamma_0=2b_0+1$ and the measure $\mu$ satisfying  
$d\mu= (\frac{c(n)}{R}a_1\gamma_0)^Q(1-\tau/G)^{-Q}dx$. 
Iterating with $\gamma_m =\kappa^m\gamma_0$ for 
$m=0,1,2,\ldots$ and taking $m\to \infty$, we finally obtain
$$ \sup_\Om\, G|(1-\tau/G)^+|^Q \leq c(n) \bigg(\frac{a_1(2b_0+1)}{R}\bigg)^Q 
\int_\Om G\dx $$ for some $c(n)>0$.  
It is easy to see that this yields \eqref{eq:Gbound1}, since $\sup_\Om G >(1+a_1)\tau$ implies 
$\sup_\Om G|(1-\tau/G)^+|^Q\geq \big(\frac{a_1}{1+a_1}\big)^Q\sup_\Om G$. 
Thus, the proof is finished. 
\end{proof}
Now, we are ready to prove the global maximum principle. For the Euclidean 
setting, similar theorems have been proved before, see e.g. Theorem 10.10 in 
\cite{Gil-Tru}. 
\begin{Thm}\label{thm:maxp}
Let $ u \in HW^{1,G}(\Om) $ be a weak solution of $\Q u=0$ in $\Om$ with 
$\sup_{\del\Om} |u|<\infty$. We assume that there exists non-negative increasing 
functions $f_1,f_2$ and $g$ such that  
the conditions \eqref{f1} and \eqref{f2} hold for $R=\diam(\Om)$ and $0<b_0<1$; furthermore we assume $\Psi(t) =tg(t)$ is convex and  
$g$ satisfies \eqref{Gcond1} for some $a_1\geq 1$. Then there exists 
$c_0=c_0(n,a_1)$ sufficiently small such that, if 
$f_1$ and $f_2$ satisfy 
\begin{equation}\label{Gcond3}
f_1(|z|)+\frac{f_2(|z|)}{1-b_0}\leq c_0\Psi\bigg(\frac{|z|}{R}\bigg)
\end{equation}
for all $|z|\geq \sup_{\del\Om} |u| $, then for some $c(n,b_0,a_1)>0$, we have 
\begin{equation}\label{eq:maxp}
\sup_\Om \, G(|u|/R) \leq c(n,b_0,a_1)\sup_{\del\Om} G(|u|/R)
\end{equation}
\end{Thm}
\begin{proof}
First notice that, since $\Psi(t) =tg(t)$ and $g$ is increasing, we 
have $G(t) \leq \Psi(t)$ and from \eqref{Gcond1}, we have $\Psi(t)\leq a_1G(t)$. 
These together imply that $G$ is convex and doubling and so is $\Psi$, with $2^{a_1}$ as their doubling constant. 

Let us denote $M=\sup_{\del\Om} |u|$ and $\Om^+ =\{u>M\}$. We choose 
$\varphi= (u-M)^+$ as a test function for $\Q u=0$ and use \eqref{f2} to get  
\begin{equation}\label{max1}
\begin{aligned}
\int_{\Om^+} \inp{A(x,u,\Xu)}{\X u} \dx &= \int_{\Om^+} (u-M)B(x,u,\Xu) \dx\\
&\leq\int_{\Om^+} \big(1-M/u\big)\Big[b_0\inp{A(x,u,\Xu)}{\Xu} + f_2(|u|)\Big]\dx\\
&\leq\int_{\Om^+}b_0\inp{A(x,u,\Xu)}{\Xu}\dx + \int_{\Om^+}f_2(|u|)\dx,
\end{aligned}
\end{equation}
and then we use \eqref{max1} together with \eqref{f1} and \eqref{Gcond3} to obtain 
\begin{equation}\label{max2}
\begin{aligned}
   \int_{\Om^+}\Psi(|\Xu|)\dx \leq \int_{\Om^+}\left[f_1(|u|) +\frac{f_2(|u|)}{1-b_0}\right]\dx  \leq c_0\int_{\Om^+}\Psi\bigg(\frac{|u|}{R}\bigg)\dx.
  \end{aligned}
\end{equation}
Now, from the Poincar\'{e} inequality \eqref{eq:psidiam}, we have 
\begin{equation}\label{max3}
\int_{\Om} \Psi\bigg(\frac{\varphi}{R}\bigg) \dx
\leq c(n,a_1) \int_\Om \Psi(|\X \varphi|)\dx = c(n,a_1)\int_{\Om^+}\Psi(|\Xu|)\dx. 
\end{equation}
We have $\Psi(2\varphi/R)\leq 2^{a_1}\Psi(\varphi/R)$ from the doubling 
condition and letting $\Om^*=\{u>2M\}$, notice that $\Psi(u/R)\leq \Psi(2\varphi/R)$ on $\Om^*$. Using these together with \eqref{max3} and \eqref{max2}, we get 
\begin{equation}\label{max4}
\int_{\Om^*}\Psi\bigg(\frac{|u|}{R}\bigg) \dx \leq \tau_0\int_{\Om^+}\Psi\bigg(\frac{|u|}{R}\bigg)\dx = \tau_0\bigg[\int_{\Om^*}\Psi\bigg(\frac{|u|}{R}\bigg)\dx +\int_{\Om^+\mns\Om^*}\Psi\bigg(\frac{|u|}{R}\bigg)\dx\bigg]
\end{equation}
where $\tau_0 = 2^{a_1}c(n,a_1)c_0 <1$ for small enough $c_0$. Hence, from 
\eqref{max4}, we arrive at 
$$ (1-\tau_0)\int_{\Om^*}\Psi\bigg(\frac{|u|}{R}\bigg) \dx  
\leq \tau_0\int_{\Om^+\mns\Om^*}\Psi\bigg(\frac{|u|}{R}\bigg)\dx,$$ 
which, after adding $(1-\tau_0)\int_{\Om^+\mns\Om^*}\Psi(|u|/R)dx$ on both 
sides, imply 
\begin{equation}\label{om+}
(1-\tau_0)\int_{\Om^+}\Psi\bigg(\frac{|u|}{R}\bigg) \dx  
\leq \int_{\Om^+\mns\Om^*}\Psi\bigg(\frac{|u|}{R}\bigg)\dx 
\leq |\Om^+|\Psi(2M/R).
\end{equation}
From a similar argument with $\Om^-=\{u<-M\}$, we can obtain 
\begin{equation}\label{om-}
(1-\tau_0)\int_{\Om^-}\Psi\bigg(\frac{|u|}{R}\bigg) \dx  
\leq |\Om^-|\Psi(2M/R).
\end{equation}
Now for $\Om^0=\{|u|\leq M\}$, we directly have 
\begin{equation}\label{om0}
(1-\tau_0)\int_{\Om^0}\Psi\bigg(\frac{|u|}{R}\bigg) \dx  
\leq |\Om^0|\Psi(2M/R)
\end{equation}
since $\Psi$ is increasing. Thus, adding \eqref{om+},\eqref{om-} and \eqref{om0}, 
we obtain
\begin{equation}\label{om}
(1-\tau_0)\int_{\Om}\Psi\bigg(\frac{|u|}{R}\bigg) \dx  
\leq |\Om|\Psi(2M/R).
\end{equation}
Now, if $c_0<1/a_1$, notice that multiplying $\Psi(|z|/R)$ on both sides of 
\eqref{Gcond3} and using inequality $G(t)\leq \Psi(t)\leq a_1G(t)$, we can obtain 
$$ \Psi(|z|/R)f_1(|z|)+G(|z|/R)\frac{f_2(|z|)}{1-b_0}\leq a_1 G(|z|/R)^2 $$
which is similar to \eqref{Gcond2}. Hence, we can combine \eqref{eq:Gbound1} 
of Lemma \ref{lem:Gbound1} with \eqref{om} and conclude 
$\sup_\Om G(|u|/R)\leq c(n,b_0,a_1) G(M/R)$, which completes the proof.
\end{proof}

\begin{Rem}
 With minor modifications of the above arguments, the global bound can also be shown
 corresponding to $u^+$ for weak supersolutions $u$ i.e. for $\Q u\geq 0$. 
\end{Rem}

\subsection{Harnack Inequality}\noindent
\\
Here we show that weak solutions of $\Q u=0$, satisfy Harnack inequality. The 
proofs are standard modifications of those in \cite{Tru--harnack} and 
\cite{Lieb--gen} for the Euclidean setting. We also refer to \cite{C-D-G} for 
the Harnack inequalities on special cases, in the sub-elliptic setting.

In this subsection, we consider 
\begin{align}
 \label{A1}\inp{A(x,z,p)}{p}&\geq |p|g(|p|)
 -a_1\,g\bigg(\frac{|z|}{R}\bigg)\frac{|z|}{R}-g(\chi)\chi\\
   \label{A2}|A(x,z,p)| &\leq a_2\, g(|p|) + a_3 \,g\bigg(\frac{|z|}{R}\bigg)
   +g(\chi)
\end{align}
for given non-negative constants $a_1, a_2, a_3,$ and $\chi,R>0$.  

\begin{Thm}\label{thm:supharnack}
 In $B_R\subset\Om$, let $ u \in HW^{1,G}(B_R)\cap L^\infty(B_R) $ 
 be a weak supersolution, $ \Q u \geq 0 $ with $ |u|\leq M$ in $B_R$ 
 and with the structure conditions 
 \eqref{A1},\eqref{A2} and 
 \begin{equation}\label{B1}
 \sign(z)B(x,z,p)\leq\frac{1}{R}
 \bigg[b_0\,g(|p|) + b_1\,g\bigg(\frac{|z|}{R}\bigg) + g(\chi) \bigg]
\end{equation}
for given non-negative constants $a_1, a_2, a_3,b_0,b_1$ and 
$g\in C^1([0,\infty))$ that satisfies \eqref{eq:g prop} with 
$\delta\geq 0$. Then for any 
$q>0$ and $0<\sigma<1$, there exists 
$c=c(n,g_0, a_1, a_2, a_3,b_0M,b_1,q)>0$ such that, letting $Q=2n+2$, 
we have
\begin{equation}\label{supharnack}
 \sup_{B_{\sigma R}}\, u^+ \leq \frac{c}{(1-\sigma)^{(1+g_0)Q/q}}
\bigg[ \bigg(\intav_{B_R}|u^+|^q\dx\bigg)^\frac{1}{q} +\chi R .\bigg]
\end{equation}
\end{Thm}

\begin{proof}
 The proof is based on Moser's iteration, similar to that of 
 Theorem 1.2 in \cite{Lieb--gen}. We provide an outline. First notice that, 
 using $\bar{z} = z + \chi R$, the structure conditions 
 \eqref{A1},\eqref{A2} and \eqref{B1} can be reduced to 
 \begin{align}
  \label{A1'}\inp{A(x,z,p)}{p}&\geq |p|g(|p|)
 -(1+a_1)g\big(|\bar z|/R\big)|\bar z|/R;\\
   \label{A2'}|A(x,z,p)| &\leq a_2\, g(|p|) + 
   (1+a_3)g\big(|\bar z|/R\big);\\
   \label{B1'}\bar zB(x,z,p)&\leq b_0|p|g(|p|) 
   + (1+b_0+b_1)g\big(|\bar z|/R\big)|\bar z|/R.
 \end{align}
To obtain \eqref{B1'}, we multiply $\bar z$ on \eqref{B1} and use 
\eqref{eq:gG5} of Lemma \ref{lem:gandG prop} with 
$t=|\bar z|/R$ and $s= |p|$.

Hence, we use $\bar u = u^+ +\chi R$ for the proof. Given any 
$\sigma\in(0,1)$, we choose a standard cutoff function 
$ \eta \in C^\infty_0(B_R) $ such that 
$ 0\leq \eta\leq 1,\ \eta= 1$ in $B_{\sigma R}$ and 
$ |\X\eta| \leq 2/(1-\sigma)R$. Then, for some $\gamma\in\R$ and 
$\beta\geq 1+|\gamma|$ which are chosen later, we use 
$$\varphi= \eta^\gamma\bar u\, G(\eta\bar u/R)^{\beta-1}e^{b_0\bar u}$$ as 
a test function for $\Q u\geq 0$, to get
\begin{equation}\label{suph1}
 \begin{aligned}
  (1+b_0)\int_{B_R}\eta^\gamma G(\eta\bar u/R&)^{\beta-1}e^{b_0\bar u}
  \inp{A(x,u,\Xu)}{\X\bar u}\dx \\
  +\ \frac{\beta-1}{R}
  \int_{B_R}\eta^\gamma\bar u\, G(\eta&\bar u/R)^{\beta-2}g(\eta\bar u/R)e^{b_0\bar u}\inp{A(x,u,\Xu)}{\X\bar u}\dx\\
  \leq\ -\frac{\beta-1}{R}
   \int_{B_R}\eta^\gamma&|\bar u|^2 G(\eta\bar u/R)^{\beta-2}g(\eta\bar u/R)e^{b_0\bar u}\inp{A(x,u,\Xu)}{\X\eta}\dx\\
   -\ \gamma\int_{B_R}&\eta^{\gamma-1}\bar u\, G(\eta\bar u/R)^{\beta-1}
   e^{b_0\bar u}\inp{A(x,u,\Xu)}{\X\eta}\dx\\
   +\ \int_{B_R}&\eta^\gamma G(\eta\bar u/R)^{\beta-1}e^{b_0\bar u}
   \bar u\, B(x,u,\Xu)\dx.
 \end{aligned}
\end{equation}
Now we use the structure condition \eqref{A1'} for the left hand side and 
\eqref{A2'},\eqref{B1'} for the right hand side of the above inequality.
Then, we use \eqref{eq:gG2} and \eqref{eq:gG3} of Lemma \ref{lem:gandG prop} and also 
the fact that $ e^{b_0\chi R}\leq e^{b_0\bar u}\leq e^{b_0(M+\chi R)}$, 
since $|u|\leq M$ in $B_R$. We obtain
\begin{equation}\label{suph2}
 \begin{aligned}
  \beta\int_{B_R}\eta^\gamma &G(\eta\bar u/R)^{\beta-1}g(|\X\bar u|)|\X\bar u|\dx\\
  &\leq \frac{a_2\beta\,e^{b_0M}}{(1-\sigma)}
 \int_{B_R}\eta^{\gamma-1}G(\eta\bar u/R)^{\beta-1}\frac{\bar u}{R}\,g(|\X\bar u|)\dx\\
&\quad + \beta(1+g_0)C_1 e^{b_0M}
  \int_{B_R}\eta^{\gamma-1}G(\eta\bar u/R)^{\beta-1}
  g\bigg(\frac{\bar u}{R}\bigg)\frac{\bar u}{R}\dx\\
&=  I_1 + I_2
 \end{aligned}
\end{equation}
where $C_1 = (1+a_1)(1+b_0)+(1+b_0+b_1)+(1+a_3)/(1-\sigma)$. 
Here onwards, we use $c=c(n,g_0,a_1,a_2,a_3,b_0M, b_1)>0$ as a large enough 
constant, throughout 
the rest of the proof. 
Now we estimate both $I_1$ and $I_2$ as follows.

For $I_1$, we use \eqref{eq:gG5} with 
$t=\frac{2}{(1-\sigma)}a_2e^{b_0M}\bar u/\eta R $ and $s=|\X \bar u|$, 
to obtain
\begin{equation}\label{suphI1}
 \begin{aligned}
  I_1 \leq\ &\frac{\beta}{2}\int_{B_R}\eta^\gamma G(\eta\bar u/R)^{\beta-1}
  g(|\X\bar u|)|\X\bar u|\dx\\ 
  &+ \frac{c\beta}{(1-\sigma)}
 \int_{B_R}\eta^\gamma G(\eta\bar u/R)^{\beta-1}\frac{\bar u}{\eta R}
 \,g\bigg(\frac{\bar u}{(1-\sigma)\eta R}\bigg)\dx\\
 \leq\ &\frac{\beta}{2}\int_{B_R}\eta^\gamma G(\eta\bar u/R)^{\beta-1}
  g(|\X\bar u|)|\X\bar u|\dx\\ 
  &+ \frac{c\beta}{(1-\sigma)^{1+g_0}}
  \int_{B_R}\eta^{\gamma-(2+2g_0)} G(\eta\bar u/R)^{\beta}\dx,
 \end{aligned}
\end{equation}
where we have used $ g(\bar u/\eta R) \leq \eta^{-2g_0}g(\eta\bar u/R) $ for 
the latter inequality of the above. 

For $I_2$, we trivially have 
\begin{equation}\label{suphI2}
 I_2 \leq \frac{c\beta}{(1-\sigma)} \int_{B_R}
 \eta^{\gamma-1}G(\eta\bar u/R)^\beta\dx.
\end{equation}
Letting $\theta = 2+2g_0$ and combining \eqref{suph2} 
with \eqref{suphI1} and \eqref{suphI2}, we obtain
\begin{equation}\label{suph3}
 \frac{\beta}{2}\int_{B_R}\eta^\gamma G(\eta\bar u/R)^{\beta-1}
  g(|\X\bar u|)|\X\bar u|\dx\leq \frac{c\beta}{(1-\sigma)^{\theta/2}}
  \int_{B_R}\eta^{\gamma-\theta}G(\eta\bar u/R)^\beta\dx.
\end{equation}
Now, we use Sobolev inequality 
$$ \Big(\int_{B_R} |w|^\kappa\dx\Big)^\frac{1}{\kappa} 
\leq c(n)\int_{B_R} |\X w|\dx $$ 
for $\kappa = Q/(Q-1)=(2n+2)/(2n+1)$ and 
$w= \eta^\gamma G(\eta\bar u/R)^\beta$ with the choice of 
$\gamma = -(Q-1)\theta$, so that $\kappa\gamma = -Q\theta = \gamma-\theta$. 
Combining with \eqref{suph3}, we obtain 
$$ \Big(\int_{B_R} \eta^{-Q\theta}
G(\eta\bar u/R)^{\kappa\beta}\dx\Big)^\frac{1}{\kappa} 
\leq  \frac{c\beta}{(1-\sigma)^{\theta/2}}\int_{B_R} 
\eta^{-Q\theta}G(\eta\bar u/R)^\beta\dx.$$
Iterating the above with $\beta_0 = q\geq Q\theta$ and 
$\beta_m=\kappa^m\beta_0$ and letting $m\to \infty$, we get 
\begin{equation}\label{suphG}
 \sup_{B_R}\,G(\eta\bar u/R) \leq 
\frac{c(q)}{(1-\sigma)^{Q\theta/2q}}\Big(\intav_{B_R}\eta^{-Q\theta}
G(\eta\bar u/R)^q\dx\Big)^\frac{1}{q}.
\end{equation}
Hence, using \eqref{eq:gG2}, we get 
$$ \sup_{B_{\sigma R}} \bar u\leq 
\frac{c(q)}{(1-\sigma)^{Q\theta/2q}}\Big(\intav_{B_R} 
|\bar u|^q\dx\Big)^\frac{1}{q}$$
for all $q\geq Q\theta$ and 
$c(q) = c(n,g_0,a_1,a_2,a_3,b_0M, b_1,q)>0$. Then from the interpolation 
argument in \cite{Dib-Tru}, we get the above for all $q>0$. This 
concludes the proof.
\end{proof}

\begin{Thm}\label{thm:infharnack}
  In $B_R\subset\Om$, let $ u \in HW^{1,G}(B_R)\cap L^\infty(B_R) $ 
  be a weak subsolution, $ \Q u \leq 0 $ with 
  $ 0\leq u\leq M$ in $B_R$ 
  and with the structure conditions 
  \eqref{A1},\eqref{A2} and 
  \begin{equation}\label{B2}
  \sign(z)B(x,z,p)\geq -\frac{1}{R}
  \bigg[b_0\,g(|p|) + b_1\,g\bigg(\frac{|z|}{R}\bigg) + g(\chi) \bigg]
  \end{equation}
for given non-negative constants $a_1, a_2, a_3,b_0,b_1$ and 
$g\in C^1([0,\infty))$ that satisfies \eqref{eq:g prop} with 
$\delta> 0$. Then there exists positive constants $q_0$ and $c$ depending on 
$n,\delta,g_0, a_1, a_2, a_3,b_0M,b_1$ such that, letting $Q=2n+2$, 
we have
\begin{equation}\label{infharnack}
\bigg(\intav_{B_{R/2}}u^{q_0}\dx\bigg)^\frac{1}{q_0} \leq\, 
c\,\Big(\inf_{B_{R/4}}\, u +\chi R\Big)
\end{equation}
\end{Thm}
\begin{proof}
 Taking $\bar u = u+\chi R$ and $\eta\in C^\infty_0(B_{R/2})$ similarly as in the proof of Theorem \ref{thm:supharnack}, we can use the test function $\varphi = \eta^\gamma \bar u 
 G(\bar u/\eta R)e^{-b_0\bar u}$ on $\Q u\leq 0$ and obtain 
 \begin{equation}\label{infh1}
\bigg(\intav_{B_{R/2}}\bar{u}^{-q}\dx\bigg)^{-\frac{1}{q}} \leq \, 
c(q)\inf_{B_{R/4}}\, \bar u 
 \end{equation}
for any $q>0$. Now for any $0<r\leq R$, we choose 
$\eta \in C^\infty_0(B_r)$ such that $\0\leq \eta\leq 1, \eta =1$ in 
$B_{r/2}$ and $|\X \eta|\leq 2/r$. Then we choose test function 
$\varphi = \eta^{g_0}\bar u G(\bar u/r)^{-1}$ in $\Q u\leq 0$. Here we use 
the fact that $g$ satisfies \eqref{eq:g prop} with $\delta>0$, so that 
from \eqref{eq:gG2'} and \eqref{eq:gG3'}, we have 
$$ G(\bar u/r)^{-1} -G(\bar u/r)^{-2}g(\bar u/r)\bar u/r 
\leq -G(\bar u/r)^{-1}\delta /(1+\delta).$$ 
Thus, using test function $\varphi$ and structure conditions 
\eqref{A1'},\eqref{A2'} and \eqref{B2}, we obtain
\begin{equation*}
 \int_{B_r}\eta^{g_0}\,\frac{g(|\X\bar u|)|\X\bar u|}{G(\bar u/r)}\dx \,
 \leq c\int_{B_r}\left[(a_1+a_3+b_0+b_1)\,
 \frac{g(\bar u/r)\bar u/r}{G(\bar u/r)}\right]\dx\, \leq cr^Q
\end{equation*}
where we suppress the dependence of $a_i,b_j,g_0,\delta$ and denote 
constant as $c$. Now, recalling \eqref{eq:gG5}, we use 
$t\leq tg(t)/g(s) + s$, with $t=|\X u|$ and $s= \bar u /r$, to obtain
\begin{equation}\label{infh2}
 \begin{aligned}
  \int_{B_{r/2}} \frac{|\X\bar u|}{\bar u}\dx &\leq
  \int_{B_{r/2}} \Big[\frac{g(|\X\bar u|)|\X\bar u|}{\bar u g(\bar u/r)} 
  + \frac{1}{r}\Big] \dx\\
  &\leq \frac{c}{r} \int_{B_r}\Big[\eta^{g_0}\, 
  \frac{g(|\X\bar u|)|\X\bar u|}{G(\bar u/r)} + 1\Big]\dx \,
   \leq c\,r^{Q-1}
 \end{aligned}
\end{equation}
Taking $ w = \log(\bar u) $, we use Poincar\'{e} inequality and \eqref{infh2} 
to get 
\begin{equation*}
 \begin{aligned}
  \intav_{B_{r/2}}|w -\{w\}_{B_{r/2}}|\dx\,
  \leq c\,r\intav_{B_{r/2}}|\X w|\dx
   = \frac{c}{r^{Q-1}}\int_{B_{r/2}} \frac{|\X\bar u|}{\bar u}\dx\,
   \leq\, c,
 \end{aligned}
\end{equation*}
which shows that $w\in \textnormal{BMO}(B_{r/2})$. 
John-Nirenberg type inequalities in the setting 
of metric spaces with doubling measures, is known; we refer to 
\cite{Buck}. This is applicable in our setting and the above 
inequality imples exponential inetegrability for $w= \log(\bar u)$. Thus 
there exists $q_0>0$ and $c_0>0$ such that 
\begin{equation}\label{infh3}
 \bigg(\intav_{B_{r/2}}\bar {u}^{-q_0}\dx\bigg)\bigg(\intav_{B_{r/2}}
 \bar {u}^{q_0}\dx\bigg)
  \leq \bigg(\intav_{B_{r/2}}e^{q_0|w -\{w\}_{B_{r/2}}|}\dx\bigg)^2
   \leq\, c_0^2. 
\end{equation}
for any $r\leq R$. 
Thus, \eqref{infh1} with $q=q_0$ and \eqref{infh3}, concludes the proof.
\end{proof}
From Theorem \ref{thm:supharnack} and Theorem \ref{thm:infharnack}, 
the following corrollary is immediate. 
\begin{Cor}\label{cor:harnack}
  In $B_R\subset\Om$, let $ u \in HW^{1,G}(B_R)\cap L^\infty(B_R) $ 
  be a weak solution of $ \Q u = 0 $ with 
  $ 0\leq u\leq M$ in $B_R$ 
  and with the structure conditions 
  \eqref{A1},\eqref{A2} and 
  \begin{equation}\label{B}
  |B(x,z,p)|\leq \frac{1}{R}
  \bigg[b_0\,g(|p|) + b_1\,g\bigg(\frac{|z|}{R}\bigg) + g(\chi) \bigg]
  \end{equation}
for given non-negative constants $a_1, a_2, a_3,b_0,b_1$ and 
$g\in C^1([0,\infty))$ that satisfies \eqref{eq:g prop} with 
$\delta> 0$. Then there exists 
$c=c(n,\delta,g_0, a_1, a_2, a_3,b_0M,b_1)>0$ such that  
we have
\begin{equation}\label{eq:harnack}
\sup_{B_{R/4}}\, u \leq\, 
c\,\Big(\inf_{B_{R/4}}\, u +\chi R\Big)
\end{equation}
\end{Cor}
Thus, bounded weak solutions satisfy the Harnack inequality 
\eqref{eq:harnack}, which implies the H\"older continuity of weak 
solutions. By standard arguments, it is possible to show that there 
exists $\alpha=\alpha(n,\delta,g_0, a_1, a_2, a_3,b_0M,b_1)\in (0,1)$ 
and $c=c(n,\delta,g_0, a_1, a_2, a_3,b_0M,b_1)>0$ such that, we have  
\begin{equation}\label{eq:osc u}
\osc_{B_r}u \leq c\left(\frac{r}{R}\right)^\alpha
\Big(\osc_{B_R}u + \chi R\Big).
\end{equation}
for every $0<r<R$ and $B_R\subset\Om$. This is enough to prove 
Theorem \ref{thm:c0alpha}.

\begin{Rem}\label{rem:maxinholder}
The growth and ellipticity conditions \eqref{A1},\eqref{A2} and \eqref{B} are special cases of the more general conditions in \eqref{f1} and \eqref{f2}. When 
$g$ satisfies \eqref{eq:g prop}, it is easy to see that \eqref{Gcond1} holds with $a_1=1+g_0$ and 
\eqref{f1}, \eqref{f2} and 
\eqref{Gcond2} holds if $f_1(|z|),f_2(|z|)\sim  g(|z|/R)|z|/R + g(\chi)\chi$. 
Therefore, it is not restrictive to assume $|u|\leq M$ since
we have Theorem \ref{thm:maxp} for the above cases. Furthermore, 
\eqref{B} can be relaxed to 
\begin{equation}\label{B'}
|zB(x,z,p)|\leq b_0|p|g(|p|) 
   + b_1 g\bigg(\frac{|z|}{R}\bigg)\frac{|z|}{R} + g(\chi)\chi
\end{equation}
so that, in this case \eqref{B1'} can be obtained immediately. 
\end{Rem}

\section{H\"older continuity of Horizontal gradient}\label{sec:LocalH}
In this section, we consider a homogenous quasilinear equation where the operator does not depend on $x$ and $u$. Estimates for this equation shall be 
necessary in Section \ref{sec:c1alpha}. However, 
all results in this section are obtained independently, without any reference to 
the rest of this paper, apart from the usage of the structure function $g$ in \eqref{eq:g prop}. 

We warn the reader that in this section $z$ is used as a variable in $\R^{2n}$, unlike the other sections. This is done to maintain continuity with \cite{Muk0}.

In a domain $\Om \subset \hh^n$, we consider   
\begin{equation}\label{eq:maineq}
\dvh (\A(\X u)) =  0 \ \ \text{in}\ \Om,
\end{equation}
where 
$\A : \R^{2n} \to \R^{2n}$ is a given $C^1$ function. We denote $ \A(z) = (\A_1(z),\A_2(z),\ldots,\A_{2n}(z))$ 
for all $z \in \R^{2n}$  
and $D\A(z)$ as the $2n\times 2n$ Jacobian matrix 
$ (\del \A_i(z)/\del z_j)_{ij}$. 
We assume that $D\A(z)$ is symmetric and 
satisfies
\begin{equation}\label{eq:str cond diff}
 \begin{aligned}
 \F(|z|)|\xi|^2 \leq \,&\inp{D\A(z)\,\xi}{\xi}\leq L\,\F(|z|)|\xi|^2;\\
 &|\A(z)|\leq L\,|z|\F(|z|). 
\end{aligned} 
\end{equation}
for every $ z,\xi \in \R^{2n}$ and $L\geq 1$, where we denote $ \F(t) = g(t)/t$  
maintaining the notation \eqref{eq:GandF}. Here $ g: [0,\infty) \to [0,\infty) $ is 
a given $C^1$ function satisfying \eqref{eq:g prop} and $g(0) =0$. 

The above equation has been considered 
previously in \cite{Muk0} where local boundedness of 
$\X u$ for a weak solution $u$ of \eqref{eq:maineq}, has been established. The goal of this section is to prove the local H\"older continuity of $\X u$. 
We restate Theorem \ref{thm:holder1} here, which is the main result of this section.
\begin{Thm}\label{thm:holder}
Let $u\in HW^{1,G}(\Om)$ be a weak solution of the equation \eqref{eq:maineq} with structure condition \eqref{eq:str cond diff} and $g$ satisfies \eqref{eq:g prop} with $\delta>0$. Then 
$\X u$ is locally H\"older continuous and there exists 
$\sigma=\sigma(n,g_0,L)\in (0,1)$ such that for any $B_{r_0}\subset\Om$ and 
$0<r<r_0/2$, we have 
\begin{equation}\label{eq:holder}
\max_{1\leq l\leq 2n}\intav_{B_r}G(|X_l u - \{X_l u\}_{B_r}|)\dx \leq c\Big(\frac{r}{r_0}\Big)^\sigma
\intav_{B_{r_0}} G(|\X u|)\dx 
\end{equation}
where $c>0$ depends on $n,\delta,g_0,L$. 
\end{Thm}

\subsection{Previous Results}\label{subsec:Previous Results}\noindent
\\
Here we provide some results that are known and previously obtained, which would be essential for our purpose. For more details, 
we refer to \cite{Muk0} and references therein. 

The following monotonicity and ellipticity inequalities follow 
easily from \eqref{eq:str cond diff}. 
\begin{align}
\label{eq:monotone}
&(1)\ \inp{\A(z)-\A(w)}{z-w}\geq c(g_0) 
\begin{cases}
  |z-w|^2 \,\F(|z|)\ &\text{if}\ |z-w|\leq 2|z| \\
 |z-w|^2 \,\F(|z-w|) \ \ &\text{if}\ |z-w|> 2|z|
\end{cases}\\ 
 \label{eq:elliptic}
&(2)\ \inp{\A(z)}{z}\geq c(g_0)\,|z|^2\F(|z|) \geq c(g_0)\, G(|z|)
\end{align}
for all $z,w\in\R^{2n}$ and some constant $c(g_0)>0$. These are essential to show the existence of a weak solution 
$u\in HW^{1,G}(\Om)$ of the equation \eqref{eq:maineq}. 
We refer to \cite{Muk0} for a brief discussion on existence and uniqueness 
for \eqref{eq:maineq}. The following theorem is Theorem 1.1 of 
\cite{Muk0}, which shows the local Lipschitz continuity of the weak solutions. 
\begin{Thm}\label{thm:Lip}
 Let $ u \in \hw{1}{G}(\Om)$ be a weak solution of equation \eqref{eq:maineq}
 satisfying structure condition \eqref{eq:str cond diff} and $g$ satisfies \eqref{eq:g prop} with $\delta>0$. Then 
 $\X u \in L^\infty_\loc(\Om, \R^{2n})$; moreover for any $B_r \subset \Om $, 
 we have 
 \begin{equation}\label{eq:locbound}
  \sup_{B_{\sigma r}}\ G(|\X u|)\leq \frac{c}{(1-\sigma)^Q}\intav_{B_r}G(|\X u|)\dx 
 \end{equation}
 for any $0<\sigma<1$, where $c = c(n,g_0,\delta,L) > 0 $ is a constant. 
\end{Thm}
Now, we also require the following apriori assumption as considered in 
\cite{Muk0}, in order to temporarily remove possible singularities of the function $\F$. Here onwards, this shall be assumed until the end of this section. 
 \begin{equation}\label{eq:ass}
(A):\ \text{There exists}\ m_1,m_2>0 \ \text{such that}
\ \lim_{t \to 0}\F(t)=m_1\ \text{and}\ \lim_{t \to \infty}\F(t)=m_2.
\end{equation}
This combined with the local boundedness of $\X u$ from Theorem \ref{thm:Lip}, makes the equation \eqref{eq:maineq} to be uniformly 
elliptic and enables us to conlcude 
\begin{equation}\label{eq:Cap reg}
\X u \in HW^{1,2}_\loc(\Om,\R^{2n})\cap C^{\,0,\alpha}_\loc(\Om,\R^{2n}) ,
\ \  Tu  \in HW^{1,2}_\loc(\Om)\cap C^{\,0,\alpha}_\loc(\Om) 
\end{equation}
from Theorem 1.1 and Theorem 3.1 of Capogna \cite{Cap--reg}. 
However, every estimates in this section, are independent of the constants $m_1$ and $m_2$ and \eqref{eq:ass} shall be ultimately removed.

The regularity \eqref{eq:Cap reg} is necessary to differentiate the 
equation \eqref{eq:maineq} and obtain the equations satisfied by 
$X_lu$ and $Tu$, as shown in the following two lemmas. The proofs are 
simple and omitted here, we refer to \cite{Muk0} and \cite{Zhong} for details. 
\begin{Lem}\label{lem: Tu}
If $ u \in HW^{1,G}(\Om) $ is a weak solution of \eqref{eq:maineq}, 
then $Tu$ is a weak solution of 
\begin{equation}\label{eq:Tu}
\sum_{i,j=1}^{2n}X_i(D_j\A_i(\X u)X_j(Tu))=0.
\end{equation}
\end{Lem}
\begin{Lem}\label{lem:Xu}
If $ u \in HW^{1,G}(\Om) $ is a weak solution of \eqref{eq:maineq}, then
for any $ l \in \{1,\ldots, n\}$, we have that $ X_lu$ is weak solution of 
\begin{equation}\label{eq:Xu}
\sum_{i,j=1}^{2n} X_i(D_j\A_i(\X u)X_jX_lu)
+\sum_{i=1}^{2n}X_i(D_i\A_{n+l}(\X u)Tu) +T(\A_{n+l}(\X u))= 0
\end{equation}
and similarly, $X_{n+l}u $ is weak solution of
\begin{equation}\label{eq:Xu other}
\sum_{i,j=1}^{2n} X_i(D_j\A_i(\X u)X_jX_{n+l}u)
-\sum_{i=1}^{2n}X_i(D_i\A_{l}(\X u)Tu) -T(\A_{l}(\X u))= 0.
\end{equation}
\end{Lem}
We enlist some Caccioppoli type inequalitites, that are very similar to those in 
\cite{Zhong} and \cite{Muk-Zhong}. They will be essential for the estimates in 
the next subsection. 

The following lemma is similar to Lemma 3.3 in \cite{Zhong}, the proof is trivial and omitted here. 
\begin{Lem}\label{caccioppoli:T}
For any $\beta\ge 0$ and  all $\eta\in C^\infty_0(\Omega)$, we
have, for some $c=c(n,g_0,L)>0$, that
\begin{equation*}
\int_\Omega \eta^2\,\weight | Tu|^\beta|\X (Tu)|^2\,
dx \le \frac{c}{(\beta+1)^2}\int_\Omega |\X\eta
|^2\weight| Tu|^{\beta+2}\, dx.
\end{equation*}
\end{Lem}
The following lemma is similar to Corollary 3.2 of \cite{Zhong} and 
Lemma 2.5 of \cite{Muk-Zhong}. This is crucial for the proof of the H\"older continuity of the horizontal gradient. The proof of the lemma is similar to that in 
\cite{Zhong} and involves few other Caccioppoli type estimates. An outline is 
provided in Appendix II, for the reader's convenience.  
\begin{Lem}\label{cor:Tu:high}
For any $q\geq 4$ and all non-negative $\eta\in C^\infty_0(\Omega)$, we have that
\begin{equation}\label{eq:Tint}
\int_\Omega\eta^{q}\,\weight| Tu|^{q}\, dx \le c(q)K^{q/2}
\int_{\supp(\eta)}\weight |\X u|^{q}\, dx,
\end{equation}
where $K=\| \X\eta\|_{L^\infty}^2+\|\eta
T\eta\|_{L^\infty}$ and  $c(q)=c(n,g_0,L,q)>0$. 
\end{Lem}
The following corollary follows immediately from 
Lemma \ref{caccioppoli:T} and Lemma
\ref{cor:Tu:high}.
\begin{Cor}\label{cor:TX}
For any $q\ge 4$ and all non-negative $\eta\in C^\infty_0(\Omega)$, we have
\begin{equation*}
\begin{aligned}
\int_\Omega \eta^{q+2}\,\weight |Tu|^{q-2}|\X(Tu)|^2\dx
\leq c(q)K^\frac{q+2}{2}\int_{spt(\eta)}\weight |\X u|^q\dx,
\end{aligned}
\end{equation*}
where $K=\|\X\eta\|_{L^\infty}^2+\|\eta
T\eta\|_{L^\infty}$ and $c(q)=c(n,g_0,L,q)>0$.
\end{Cor}

\subsection{The truncation argument}\label{subsec:The truncation argument}\noindent
\\
In this subsection, we follow the technique of \cite{Muk-Zhong} and 
prove Caccioppoli type inequalities invovling a double truncation of 
horizontal derivatives. In the setting of Euclidean spaces, similar 
ideas have been implemented previously by Tolksdorff \cite{Tolk} and  Lieberman 
\cite{Lieb--bound}.

Here onwards, throughout this section, we shall denote 
$u\in HW^{1,G}(\Om)$ as a weak solution of \eqref{eq:maineq} and 
equipped with local Lipschitz continuity from Theorem \ref{thm:Lip}, we denote 
\begin{equation}\label{def:mu}
\mu_i(r)={\sup}_{{B_r}}\, |X_i u|, \quad \mu(r)=\max_{1\le i\le 2n}\mu_i(r).
\end{equation}
for a fixed ball $B_r\subset\Om$.

We fix any $l\in \{ 1,2,.., 2n\}$ and consider the following double 
truncation 
\begin{equation}\label{def:v}
v:=\min\big(\mu(r)/8\,,\,\max\,(\mu(r)/4-X_lu,0)\big).
\end{equation}
It is important to note that, from the regularity \eqref{eq:Cap reg}, we have 
\begin{equation}\label{reg:v} \X v\in L^2_\loc(\Omega;\R^{2n}), \quad Tv\in L^2_\loc(\Omega)
\end{equation}
and moreover, letting 
\begin{equation}\label{eq:setE}
E=\{ x\in \Omega: \mu(r)/8< X_lu<\mu(r)/4\},
\end{equation}
we have that 
\begin{equation}\label{vis0}
\X v=\begin{cases}
-\X X_lu \ & \text{a.e. in } E;\\
0 &\text{a.e. in } \Omega\setminus E,
\end{cases}
\quad\text{and}\quad 
T v=\begin{cases}
-T X_l u \ & \text{a.e. in } E;\\
0 &\text{a.e. in } \Omega\setminus E.
\end{cases}
\end{equation}
 The properties of this 
truncation shall be exploited for proving all the following Caccioppoli 
type estimates. In particular, notice that 
\begin{equation}\label{comparable1fst}
\mu(r)/8\le |\Xu|\le (2n)^{1/2}\mu(r) \quad\text{in } E\cap B_r;
\end{equation}
since $\F(t) = g(t)/t$, \eqref{comparable1fst} combined with 
\eqref{eq:gdoub} implies 
\begin{equation}\label{comparable1}
\frac{1}{8^{g_0}(2n)^{1/2}}\F(\mu(r))\le \F(|\Xu|)\le 8(2n)^{g_0/2}\F(\mu(r)) \quad\text{in } E\cap B_r,
\end{equation}
which shall be used several times during the estimates that follow in this 
subsection. 
The main lemma required to prove Theorem \ref{thm:holder}, is the following.
\begin{Lem}\label{lem:main}
Let $v$ be the truncation \eqref{def:v} and $\eta\in C^\infty_0(B_r)$ be a 
non-negative cut-off function such that $0\le \eta\le 1$ in $B_r$, 
$\eta=1$ in $B_{r/2}$ and that $|\X \eta|\le 4/r,\ |\X \X \eta|\le 16n/r^2$. 
Then we have the following Caccioppoli type inequality 
 \begin{equation}\label{eq:mainest}
  \int_{B_r} \eta^{\beta+4}v^{\beta+2}|\X v|^2\, dx \leq c(\beta+2)^2
\frac{| B_r|^{1-1/\gamma}}{r^2}
  \mu(r)^4\Big(\int_{B_r}\eta^{\gamma\beta}v^{\gamma\beta}\, dx\Big)^{1/\gamma}
\end{equation}
for all $\beta\ge 0$ and $\gamma>1$, where $c=c(n,g_0, L,\gamma)>0$ 
is a constant.
\end{Lem}
In the setting of equations with $p$-laplace type growth, the above lemma has been shown previously in \cite{Muk-Zhong}(see Lemma 1.1). The proof is going to be similar. Hence, we would require two auxillary lemmas, similarly as in 
\cite{Muk-Zhong}. 

We also remark that the inequality \eqref{eq:mainest} also holds corresponding to 
the truncation 
\[ v^\prime=\min\big(\mu(r)/8, \max(\mu(r)/4+X_lu,0)\big),\] 
and the proof can be carried out 
in the same way as that of Lemma \ref{lem:main}.

The following lemma is the analogue of Lemma 3.1 of \cite{Muk-Zhong}. The proof is similar and lengthy, which we provide in the 
Appendix I. 
\begin{Lem}\label{lem:start}
For any $\beta\ge 0$ and all non-negative $\eta\in C^\infty_0(\Omega)$, we have that
\begin{equation}\label{ineq1}
\begin{aligned}
\int_\Omega \eta^{\beta+2}v^{\beta+2}&\weight |\X u|^2|\X\X u|^2\, dx\\
\le & \ c (\beta+2)^2\int_\Omega \eta^\beta \big(|\X \eta|^2+\eta| T\eta|\big)v^{\beta+2}
\weight |\X u|^4 \, dx\\
& +  c(\beta+2)^2\int_\Omega \eta^{\beta+2} v^ \beta\weight |\X u|^4| \X v|^2\, dx\\
& +
c\int_\Omega\eta^{\beta+2} v^{\beta+2}\weight |\X u|^2| Tu|^2\, dx,
\end{aligned}
\end{equation}
where $v$ is as in \eqref{def:v} and $c=c(n,g_0,L)>0$.
\end{Lem}

Throughout the rest of this 
subsection, we fix a ball $B_r\subset\Omega$ and a cut-off function 
$\eta\in C^\infty_0(B_r)$ that satisfies 
\begin{equation}\label{etaBr1}
0\le \eta\le 1 \quad \text{in }B_r, \quad
\eta=1 \quad \text{in }B_{r/2}
\end{equation} 
and  
\begin{equation}\label{etaBr2} 
|\X \eta|\le 4/r,\quad |\X \X \eta|\le 16n/r^2, \quad | T\eta|\le 32n/r^2\quad \text{in }B_r.
\end{equation}

The following technical lemma, that is 
required for the proof of Lemma \ref{lem:main}, 
is a weighted Caccioppoli inequality for $Tu$ involving $v$ similar 
to that in Lemma 3.2 of \cite{Muk-Zhong}. We provide the proof here for sake 
of completeness.
\begin{Lem}\label{lem:tech}
Let $B_r \subset \Om$ be a ball and $\eta\in C^\infty_0(B_r)$ be a cut-off function satisfying \eqref{etaBr1} and \eqref{etaBr2}. 
Let $\tau \in (1/2,1)$ and $\gamma\in(1,2)$ be two fixed numbers. 
Then, for any $\beta \geq0$, 
we have the following estimate,   
\begin{equation}\label{eq:tech}
 \begin{aligned}
  \int_\Omega \eta^{\tau(\beta+2)+4}\,v^{\tau(\beta+4)} 
\weight |\X u|^4|\X (Tu)|^2\, dx
 \,\leq c(\beta+2)^{2\tau}\frac{|B_r|^{1-\tau}}{r^{2(2-\tau)}}\F(\mu(r))
 \mu(r)^6
 \,J^{\tau},
 \end{aligned}
\end{equation}
where $c = c(n,g_0,L,\tau,\gamma)>0$ and 
\begin{equation}\label{def:J}
J= \int_{B_r} \eta^{\beta+4}v^{\beta+2} |\X v|^2\dx 
 \ +\ \mu(r)^4\frac{|B_r|^{1-\frac{1}{\gamma}}}{r^2}
\Big(\int_{B_r} \eta^{\gamma\beta}v^{\gamma\beta}\dx\Big)^\frac{1}{\gamma}.
\end{equation}
\end{Lem}

\begin{proof}
We denote the left hand side of \eqref{eq:tech} by $M$, 
\begin{equation}\label{def:K}
 M=\ \int_\Omega \eta^{\tau(\beta+2)+4}\,v^{\tau(\beta+4)} 
\weight |\X u|^4|\X (Tu)|^2\, dx,
\end{equation}
where $1/2<\tau<1$. 
Now we use
$\varphi = \eta^{\tau(\beta+2)+4}\,v^{\tau(\beta+4)}|\X u|^4\, Tu$ 
as a test function for the equation \eqref{eq:Tu}. We obtain that 
\begin{equation}\label{est:Ki}
 \begin{aligned}
 &\int_\Omega \sum_{i,j=1}^{2n}\eta^{\tau(\beta+2)+4}\,v^{\tau(\beta+4)} 
|\X u|^4 D_j\A_i(\X u)X_jTu\,X_iTu\, dx\\
=&-(\tau(\beta+2)+4)\int_\Omega \sum_{i,j=1}^{2n} \eta^{\tau(\beta+2)+3}\,v^{\tau(\beta+4)} 
|\X u|^4 Tu D_j\A_i(\X u)X_jTu\,X_i\eta\dx\\
&-\tau(\beta+4)\int_\Omega \sum_{i,j=1}^{2n}\eta^{\tau(\beta+2)+4}
\,v^{\tau(\beta+4)-1} 
|\X u|^4 Tu D_j\A_i(\X u)X_jTu\,X_i v\dx\\
&-4\int_\Omega \sum_{i,j,k=1}^{2n} \eta^{\tau(\beta+2)+4}\,v^{\tau(\beta+4)} 
|\X u|^2 X_ku Tu D_j\A_i(\X u)X_jTu\,X_iX_ku\dx\\
=&\, K_1 +K_2+ K_3,
\end{aligned}
\end{equation}
where the integrals in the right hand side of \eqref{est:Ki} are denoted by
$K_1,K_2,K_3$ in order. To prove the lemma, 
we estimate both sides of (\ref{est:Ki}) as follows. 

For the left hand side, 
we have by the structure condition \eqref{eq:str cond diff} that
\begin{equation}\label{est:K}
 \text{left of \eqref{est:Ki}} \geq \int_\Omega \eta^{\tau(\beta+2)+4}\,v^{\tau(\beta+4)} 
\weight |\X u|^4|\X (Tu)|^2\, dx = M, 
\end{equation}
and for the right hand side of (\ref{est:Ki}), we estimate each item $K_i, i=1,2,3$, one by one. 

To this end, 
we denote  
\begin{equation}\label{def:tilde K}
\tilde K = \int_\Om \eta^{(2\tau-1)(\beta+2)+6} 
\,v^{(2\tau-1)(\beta+4)} \weight |\X u|^4 |Tu|^2 |\X(Tu)|^2\dx.
\end{equation} 
First, we estimate $K_1$ by the structure condition \eqref{eq:str cond diff} and 
H\"older's inequality, to get
\begin{equation}\label{est:K1}
\begin{aligned}
\vert K_1\vert \leq & c(\beta+2)\int_\Omega \eta^{\tau(\beta+2)+3}\,v^{\tau(\beta+4)} 
\weight |\X u|^4|Tu||\X (Tu)||\X\eta|\, dx\\
\leq & c(\beta+2) \tilde K^\frac{1}{2} 
\Big(\int_\Om \eta^{\beta+2} v^{\beta+4} 
\weight |\X u|^4 |\X\eta|^2\dx\Big)^\frac{1}{2},
\end{aligned}
\end{equation}
where $c = c(n,g_0,L,\tau)>0$. 

Second, we estimate $K_2$ also by the structure condition \eqref{eq:str cond diff} and 
H\"older's inequality,  
\begin{equation}\label{est:K2}
\begin{aligned}
 \vert K_2\vert \leq & c(\beta+2)\int_\Omega \eta^{\tau(\beta+2)+4}\,v^{\tau(\beta+4)-1} 
\weight |\X u|^4|Tu||\X (Tu)||\X v|dx\\
 \leq & c(\beta+2) \tilde K^\frac{1}{2} 
\Big(\int_\Om \eta^{\beta+4} v^{\beta+2} 
\weight |\X u|^4 |\X v|^2\dx\Big)^\frac{1}{2}.
\end{aligned}
\end{equation}

Finally, we estimate $K_3$. In the following, the first inequality
follows from the
structure condition \eqref{eq:str cond diff}, the second from 
H\"older's inequality and the third from Lemma \ref{lem:start}. We have 
\begin{equation}\label{est:K3}
\begin{aligned}
\vert K_3\vert &\leq c\int_\Omega \eta^{\tau(\beta+2)+4}\,v^{\tau(\beta+4)} 
\weight |\X u|^3|Tu||\X (Tu)||\X\X u|\, dx\\
&\leq c \tilde K^\frac{1}{2} 
\Big(\int_\Om \eta^{\beta+4} v^{\beta+4} 
\weight |\X u|^2 |\X\X u|^2\dx\Big)^\frac{1}{2}\\
&\leq c\, \tilde K^\frac{1}{2}\, I^\frac{1}{2}, 
\end{aligned}
\end{equation}
where $I$ is the right hand side of \eqref{ineq1} in Lemma \ref{lem:start}
\begin{equation}\label{def:I}
\begin{aligned}
I=&  \,c (\beta+2)^2\int_\Omega \eta^{\beta+2}
v^{\beta+4}
\weight |\X u|^4\big(|\X \eta|^2+\eta| T\eta|\big)\, dx\\
& +  c(\beta+2)^2\int_\Omega \eta^{\beta+4} v^{\beta+2}\weight |\X u|^4
| \X v|^2\, dx\\
& +
c\int_\Omega\eta^{\beta+4} v^{\beta+4}\weight |\X u|^2 | Tu|^2\, dx. 
\end{aligned}
\end{equation}
where $c= c(n,g_0,L)>0$.  
Notice that the integrals on the right hand side of \eqref{est:K1} and \eqref{est:K2} are both controlled from above by $I$. 
Hence, we can combine \eqref{est:K1}, \eqref{est:K2} and \eqref{est:K3} to obtain 
\[\vert K_1\vert+\vert K_2\vert+\vert K_3\vert \leq\ c \tilde K^\frac{1}{2} I^\frac{1}{2}, 
\]
from which,  together with the estimate \eqref{est:K} for the left hand side of (\ref{est:Ki}), it follows that
\begin{equation}\label{est:K'}
M\le c {\tilde K}^\frac{1}{2} I^\frac{1}{2},
\end{equation}
where $c = c(n,g_0,L,\tau)>0$. 
Now, we estimate $\tilde K$ by H\"older's inequality as follows. 
\begin{equation}\label{est:tilde K}
 \begin{aligned}
  \tilde K\ \leq & \Big(\int_\Omega\eta^{\tau(\beta+2)+4}\,v^{\tau(\beta+4)} 
\weight |\X u|^4|\X (Tu)|^2\, dx\Big)^{\frac{2\tau-1}{\tau}}\\
& \times\Big(\int_\Omega\eta^{\frac{2\tau}{1-\tau}+4}
\weight |\X u|^4 |Tu|^{\frac{2\tau}{1-\tau}}
|\X(Tu)|^2\dx\Big)^\frac{1-\tau}{\tau}\\
=&\, M^{\frac{2\tau-1}{\tau}} H^\frac{1-\tau}{\tau}, 
 \end{aligned}
\end{equation}
where $M$ is as in \eqref{def:K} and we denote by $H$ the second integral on the right hand side of (\ref{est:tilde K})
\begin{equation}\label{def:M}
 H=\int_\Omega\eta^{\frac{2\tau}{1-\tau}+4}
\weight |\X u|^4 |Tu|^{\frac{2\tau}{1-\tau}}
|\X(Tu)|^2\dx.
\end{equation}
Combining \eqref{est:tilde K} and \eqref{est:K'}, we get
\begin{equation}\label{est:K''}
 M \leq c H^{1-\tau} I^{\tau},
\end{equation}
for some $c = c(n,g_0,L,\tau)>0$. 
To estimate $M$, we estimate $H$ and $I$ from above. 
We estimate $H$ by Corollary \ref{cor:TX} with $q = 2/(1-\tau)$ and monotonicity 
of $g$, to obtain
\begin{equation}\label{est:M}
\begin{aligned}
H&\leq  c\mu(r)^4\int_\Omega \eta^{q+2} \weight | Tu|^{q-2}|\X(Tu)|^2\, dx\\
&\le  \frac{c}{r^{q+2}}\mu(r)^4\int_{B_r}\weight |\X u|^q \, dx 
\le  \frac{c}{r^{q+2}}| B_r|\F(\mu(r))\mu(r)^{q+4},
\end{aligned}
\end{equation}
where $c = c(n,g_0,L,\tau)>0$. 

Now, we fix 
$1<\gamma<2$ and estimate each term of $I$ in \eqref{def:I} as follows. 
For the first term of $I$, we have by H\"older's inequality and monotonicity of 
$g$ that 
\begin{equation}\label{I1}
\begin{aligned}
\int_\Omega \eta^{\beta+2}
v^{\beta+4}
&\weight |\X u|^4\big(|\X \eta|^2+\eta| T\eta|\big)\, dx\\
& \leq \frac{c}{r^2}\F(\mu(r))\mu(r)^8
|B_r|^{1-\frac{1}{\gamma}}
\Big(\int_{B_r} \eta^{\gamma\beta}v^{\gamma\beta}\dx\Big)^\frac{1}{\gamma}.
\end{aligned}
\end{equation}
For the second term of $I$, we similarly have 
\begin{equation}\label{I2}
\begin{aligned}
\int_\Omega \eta^{\beta+4} v^{\beta+2}\weight |\X u|^4
| \X v|^2\, dx\leq  c\F(\mu(r))\mu(r)^4
\int_{B_r} \eta^{\beta+4} v^{\beta+2}| \X v|^2\, dx.
\end{aligned}
\end{equation}
For the third term of $I$, we have that
\begin{equation}\label{I3}
\begin{aligned}
\int_\Omega 
\eta^{\beta+4} v^{\beta+4}
&\weight |\X u|^2|Tu|^2\dx\\
\ \leq\ & \Big(\int_\Om \eta^\frac{2\gamma}{\gamma-1}\weight |\X u|^2
|Tu|^\frac{2\gamma}{\gamma-1}\dx\Big)^{1-\frac{1}{\gamma}}\\
&\times\Big( \int_\Om \eta^{\gamma(\beta+2)}v^{\gamma(\beta+4)}
\weight |\X u|^2 \dx\Big)^\frac{1}{\gamma}\\
\ \leq\ & \frac{c}{r^2}\F(\mu(r))\mu(r)^8
|B_r|^{1-\frac{1}{\gamma}}
\Big(\int_{B_r} \eta^{\gamma\beta}v^{\gamma\beta}\dx\Big)^\frac{1}{\gamma}
\end{aligned}
\end{equation}
where $c = c(n,g_0,L,\gamma)>0$. Here in the above inequalities, the first one follows from H\"older's inequality and the second from Lemma \ref{cor:Tu:high} 
and monotonicity of $g$.  
Combining the estimates for three items of $I$ above \eqref{I1}, \eqref{I2} and \eqref{I3}, we get the following estimate for $I$,
 \begin{equation}\label{est:I}
\begin{aligned}
I \leq c(\beta+2)^2 \F(\mu(r))\mu(r)^4 J,
\end{aligned}
\end{equation}
where $J$ is defined as in \eqref{def:J}   
\begin{equation*}
 J = \int_{B_r} \eta^{\beta+4} v^{\beta+2} |\X v|^2\dx 
\ +\ \mu(r)^4\frac{|B_r|^{1-\frac{1}{\gamma}}}{r^2} 
\Big(\int_{B_r} \eta^{\gamma\beta}v^{\gamma\beta}\dx\Big)^\frac{1}{\gamma}.
\end{equation*}
Now from  the estimates \eqref{est:M} for $G$ and  
\eqref{est:I} for $I$, we obtain the desired estimate for $M$ by \eqref{est:K''}. Combing \eqref{est:M}, \eqref{est:I} and \eqref{est:K''}, we end up with 
\begin{equation}\label{est:K f}
M \leq c (\beta+2)^{2\tau}\frac{|B_r|^{1-\tau}}{r^{2(2-\tau)}}\F(\mu(r))
 \mu(r)^6 J^{\tau},
\end{equation}
where $c = c(n,g_0,L,\tau,\gamma)>0$. This completes the proof. 
\end{proof}
Now we provide the proof of Lemma \ref{lem:main}, for completeness. 
\begin{proof}[Proof of Lemma \ref{lem:main}]
 First, notice that we may assume $\gamma<3/2$, since 
 otherwise we can apply H\"older's inequality to the integral 
 in the right hand side of the claimed inequality \eqref{eq:mainest}. 
 Also, we recall from \eqref{def:v}, that for some $l\in\{1,\ldots,2n\}$, 
 $$ v=\min\big(\mu(r)/8\,,\,\max\,(\mu(r)/4-X_lu,0)\big).$$ 
 We prove the lemma assuming $l\in\{1,\ldots,n\}$; the case for 
 $l\in\{n+1,\ldots,2n\}$ can be proven similarly. 
 Henceforth, we fix $1<\gamma<3/2$ and $l\in\{1,\ldots,n\}$ throughout 
 the rest of the proof. 
 Let $\beta\ge 0$ and $\eta\in C^\infty_0(B_r)$ be a cut-off 
 function satisfying (\ref{etaBr1}) and (\ref{etaBr2}). Using test function 
 $\varphi=\eta^{\beta+4}v^{\beta+3}$ for the equation \eqref{eq:Xu}, 
 we obtain 
 \begin{equation}\label{lemest2}
\begin{aligned}
-(\beta+3)\int_\Omega &\sum_{i,j=1}^{2n}\eta^{\beta+4}v^{\beta+2} 
D_j\A_i(\X u)X_jX_luX_i v\, dx\\
\ =\ & (\beta+4)\int_\Omega \sum_{i,j=1}^{2n} \eta^{\beta+3}v^{\beta+3} 
D_j\A_i(\X u) X_jX_luX_i\eta\, dx\\
& + (\beta+4)\int_\Omega\sum_{i=1}^{2n} \eta^{\beta+3}v^{\beta+3} 
D_i\A_{n+l}(\X u)Tu X_i\eta\, dx\\
&+(\beta+3)\int_\Omega \sum_{i=1}^{2n} \eta^{\beta+4} v^{\beta+2} 
D_i\A_{n+l}(\X u)X_i v \,Tu\, dx\\
&-\int_\Omega \eta^{\beta+4}v^{\beta+3}\,T\big( \A_{n+l}(\X u)\big)\, dx.
\end{aligned}
\end{equation}
Now notice that from \eqref{eq:comm}, we have 
\begin{equation*}
\begin{aligned}
\sum_{i,j=1}^{2n}&D_j\A_i(\X u) X_jX_luX_i\eta+ 
\sum_{i=1}^{2n} D_i\A_{n+l}(\X u)Tu X_i\eta\\ 
=&\sum_{i,j=1}^{2n}D_j\A_i(\X u) X_lX_juX_i\eta
\,=\sum_{i=1}^{2n} X_l\big(\A_i(\X u)\big)X_i\eta.
\end{aligned}
\end{equation*}
Thus, we can combine the first two integrals in the right hand 
side of (\ref{lemest2}) by the above equality. 
Then (\ref{lemest2}) becomes
\begin{equation}\label{lemest3}
\begin{aligned}
-(\beta+3)&\int_\Omega \sum_{i,j=1}^{2n}\eta^{\beta+4}v^{\beta+2} 
D_j\A_i(\X u)X_jX_luX_i v\, dx\\
= \ 
& (\beta+4)\int_\Omega\sum_{i=1}^{2n} \eta^{\beta+3}v^{\beta+3} 
X_l\big(\A_i(\X u)\big)X_i\eta \, dx\\
&+(\beta+3)\int_\Omega \sum_{i=1}^{2n} \eta^{\beta+4} v^{\beta+2} 
D_i\A_{n+l}(\X u)X_ivTu\, dx\\
& -\int_\Omega \eta^{\beta+4}v^{\beta+3}T\big(\A_{n+l}(\X u)\big)\, dx\\
=\  &\, I_1+I_2+I_3,
\end{aligned}
\end{equation}
where we denote the terms in the right hand side of 
(\ref{lemest3}) by $I_1,I_2,I_3$, respectively. 

We will estimate both sides 
of (\ref{lemest3}) as follows. 
For the left hand side, denoting $E$ as in \eqref{eq:setE} and using structure 
condition \eqref{eq:str cond diff}, we have 
\begin{equation}\label{est:lemleft}
\begin{aligned}
\text{left of } (\ref{lemest3}) & \ge (\beta+3)\int_E \eta^{\beta+4}
v^{\beta+2}\weight |\X v|^2\, dx\\
& \ge c_0(\beta+2)\F(\mu(r))\int_{B_r}\eta^{\beta+4}v^{\beta+2}|\X v|^2\, dx,
\end{aligned}
\end{equation}
for a constant $c_0=c_0(n,g_0, L)>0$. Here we have used \eqref{vis0} and 
\eqref{comparable1}. 

For the right hand side of (\ref{lemest3}), we claim that
each item $I_1,I_2,I_3$ satisfies  
\begin{equation}\label{lem:claim}
\begin{aligned}
| I_m|\ \le &\ \frac{c_0}{6}(\beta+2)\F(\mu(r))\int_{B_r}\eta^{\beta+4}v^{\beta+2}|\X v|^2\, dx\\
&\ +c(\beta+2)^3
\frac{| B_r|^{1-1/\gamma}}{r^2}
  \F(\mu(r))\mu(r)^4\Big(\int_{B_r}\eta^{\gamma\beta}v^{\gamma\beta}\, dx\Big)^{1/\gamma},
\end{aligned}
\end{equation}
where $m=1,2,3$, $1<\gamma<3/2$ and  $c$  is a constant depending only on $n,g_0,L$ and $\gamma$.
Then the lemma follows from the estimate (\ref{est:lemleft}) for the
left hand side of (\ref{lemest3}) and the above claim
(\ref{lem:claim}) for each item in the right. Thus, we are only left with 
proving the claim \eqref{lem:claim}. 

In the rest of the proof, we estimate
$I_1, I_2,I_3$ one by one. First for $I_1$, using integration by parts, 
we have that
\[ I_1=-(\beta+4)\int_\Omega\sum_{i=1}^{2n} \A_i(\X u)X_l\big( \eta^{\beta+3}
v^{\beta+3}X_i\eta\big)\, dx,\]
from which it follows by the structure condition \eqref{eq:str cond diff}, that
\begin{equation}\label{lemest4}
\begin{aligned}
| I_1|\le &\, c(\beta+2)^2\int_\Omega \eta^{\beta+2} v^{\beta+3}\weight |\X u|\big(|\X\eta|^2+\eta|\X\X\eta|\big)\, dx\\
&+ c(\beta+2)^2\int_\Omega \eta^{\beta+3}v^{\beta+2}\weight |\X u||\X v\|\X\eta|\, dx\\
\le & \, \frac{c}{r^2}(\beta+2)^2\F(\mu(r))\mu(r)^4\int_{B_r} \eta^\beta
v^\beta\, dx\\
&\, +\frac{c}{r}(\beta+2)^2\F(\mu(r))
\mu(r)^2\int_{B_r}\eta^{\beta+2}v^{\beta+1}|\X v|\, dx,
\end{aligned}
\end{equation}
where $c=c(n,g_0,L)>0$. For the latter inequality of \eqref{lemest4}, we have used the fact that $g(t)=t\F(t)$ is monotonically increasing. 
Now we apply Young's inequality to the last term of  
(\ref{lemest4}) to end up with 
\begin{equation}\label{lemest:I1}
\begin{aligned}
| I_1|\le &\, \frac{c_0}{6} (\beta+2)\F(\mu(r))
\int_{B_r}\eta^{\beta+4}v^{\beta+2}|\X v|^2\, dx\\
&\, +\frac{c}{r^2}(\beta+2)^3\F(\mu(r))\mu(r)^4\int_{B_r}\eta^\beta v^\beta\, dx,
\end{aligned}
\end{equation}
where $c=c(n,g_0,L)>0$ and $c_0$ is the same constant as in (\ref{est:lemleft}).
The claimed estimate (\ref{lem:claim}) for $I_1$, follows from the 
above estimate (\ref{lemest:I1}) and H\"older's inequality.

To estimate $I_2$, we have by the structure condition \eqref{eq:str cond diff} that
\[ | I_2| \le c(\beta+2) \int_\Omega \eta^{\beta+4} v^{\beta+2} \weight 
| \X v| | Tu|\, dx,
\]
from which it follows by H\"older's inequality that
\begin{equation}\label{lemest5}
\begin{aligned}
| I_2| \le  &\, c(\beta+2) \Big(\int_E \eta^{\beta+4}v^{\beta+2}\weight|\X v|^2\, dx\Big)^{\frac{1}{2}}\\
&\qquad\qquad \times \Big(\int_E \eta^{\gamma(\beta+2)}v^{\gamma(\beta+2)}\weight \, dx\Big)^{\frac{1}{2\gamma}}\\
& \qquad \qquad \times\Big(\int_\Omega \eta^q \,\weight| Tu|^q\, dx\Big)^{\frac{1}{q}},
\end{aligned}
\end{equation}
where $q=2\gamma/(\gamma-1)$. 
The fact that the integrals are on the set 
$E$, is crucial since we can use \eqref{comparable1} and the following estimates can not be carried out unless the function $\F$ is increasing. We have the following estimates for the first two integrals of the above, using \eqref{comparable1}. 
\begin{align}
\label{lemest6}
\int_E \eta^{\beta+4}v^{\beta+2}\weight|\X v|^2\, dx
 \le c \F(\mu(r)) \int_{B_r} \eta^{\beta+4}
v^{\beta+2}|\X v|^2\, dx, 
\end{align}
and 
\begin{align}
\label{lemest7}
\int_E \eta^{\gamma(\beta+2)}v^{\gamma(\beta+2)}\weight \, dx\le c \F(\mu(r))\mu(r)^{2\gamma}\int_{B_r} \eta^{\gamma\beta} v^{\gamma\beta}\, dx,
\end{align}
where $c=c(n,g_0,L)>0$. We estimate the last integral in the right hand side of 
(\ref{lemest5}) by \eqref{eq:Tint} of Lemma \ref{cor:Tu:high} and monotonicity of $g$, to obtain
\begin{equation}\label{lemest8}
\begin{aligned}
\int_\Omega \eta^q \,\weight| Tu|^q\, dx& \le \frac{c}{r^q}\int_{B_r}\weight |\X u|^q\, dx\le \frac{c| B_r|}{r^q}\F(\mu(r))\mu(r)^q,
\end{aligned}
\end{equation}
where $c=c(n,g_0,L, \gamma)>0$. Now combining the above three estimates (\ref{lemest6}), (\ref{lemest7}) and (\ref{lemest8}) for the three integrals in (\ref{lemest5}) respectively, we end up with the following estimate for $I_2$ 
\begin{equation*}
| I_2| \le c(\beta+2)\frac{| B_r|^{\frac{\gamma-1}{2\gamma}}}{r}\F(\mu(r))
\mu(r)^2\Big(\int_{B_r}\eta^{\beta+4}v^{\beta+2}|\X v|^2\, dx\Big)^{\frac{1}{2}}
\Big(\int_{B_r}\eta^{\gamma\beta}v^{\gamma\beta}\, dx\Big)^{\frac{1}{2\gamma}},
\end{equation*}
from which, together with Young's inequality, the claim 
(\ref{lem:claim}) for $I_2$ follows.

Finally, we prove the claim (\ref{lem:claim}) for $I_3$.  Recall that
\[ I_3=-\int_\Omega\eta^{\beta+4}v^{\beta+3}T\big(\A_{n+l}(\X u)\big)\, dx.\]
By virtue of the regularity (\ref{reg:v}) for $v$,  integration by parts yields
\begin{equation}\label{lemest10}
\begin{aligned}
I_3 =&\, \int_\Omega \A_{n+l}(\X u) T\big(\eta^{\beta+4}v^{\beta+3}\big)\, dx\\
=&\, (\beta+4)\int_\Omega \eta^{\beta+3}v^{\beta+3} \A_{n+l}(\X u)T\eta\, dx\\
&\, +(\beta+3)\int_\Omega \eta^{\beta+4}v^{\beta+2} \A_{n+l}(\X u)Tv\, dx
= I_3^1+I_3^2,
\end{aligned}
\end{equation}
where we denote the last two integrals in the above equality by $I_3^1$ and $I_3^2$, respectively. The estimate for $I_3^1$ easily follows from the 
structure condition \eqref{eq:str cond diff} and monotonicity of $g$, as 
\begin{equation}\label{est:I31}
\begin{aligned}
| I_3^1|\le & \, c(\beta+2) \int_\Omega \eta^{\beta+3}v^{\beta+3}
\weight |\X u|| T\eta|\, dx\\
\le & \,\frac{c}{r^2}(\beta+2)\F(\mu(r))\mu(r)^4\int_{B_r}\eta^\beta
v^\beta\, dx.
\end{aligned}
\end{equation}
Thus by H\"older's inequality, $I_3^1$ satisfies estimate (\ref{lem:claim}). To 
estimate $I_3^2$,  
note that by (\ref{vis0}) and the 
structure condition \eqref{eq:str cond diff} we have
\begin{equation}\label{lemest11}
| I_3^2|\le c(\beta+2)\int_E \eta^{\beta+4}v^{\beta+2}\weight |\X u|| \X (Tu)|\, dx,
\end{equation}
where the set $E$ is as in \eqref{eq:setE}. For $1<\gamma<3/2$, we 
continue to estimate $I_3^2$ by H\"older's inequality as follows, 
\[\begin{aligned} | I_3^2|\le c(\beta+2)&\Big(\int_E \eta^{(2-\gamma)(\beta+2)+4}v^{(2-\gamma)(\beta+4)} \weight |\X u|^2|\X (Tu)|^2\, dx\Big)^{\frac 1 2}\\
& \times \Big(\int_E\eta^{\gamma(\beta+2)}v^{\gamma\beta+4(\gamma-1)}\weight\, dx\Big)^{\frac 1 2}.\end{aligned}\]
Since, we have (\ref{comparable1}) on the set $E$, hence 
\begin{equation}\label{est:I32}
| I_3^2| \le c(\beta+2)\F(\mu(r))^\frac{1}{2}\mu(r)^{2(\gamma-1)-1} M^{\frac 1 2}\Big(\int_{B_r} \eta^{\gamma\beta}
v^{\gamma\beta}\, dx\Big)^{\frac 1 2},
\end{equation}
where 
\begin{equation}\label{def:K new}
 M= \int_\Omega \eta^{(2-\gamma)(\beta+2)+4}\,v^{(2-\gamma)(\beta+4)} 
\weight |\X u|^4|\X (Tu)|^2\, dx.
\end{equation}
Now we can apply Lemma \ref{lem:tech} to estimate 
$M$ from above. 
Note that Lemma \ref{lem:tech}
with $\tau = 2-\gamma$, gives us that
\begin{equation}\label{est:K final}
M \leq c (\beta+2)^{2(2-\gamma)}\frac{|B_r|^{\gamma-1}}{r^{2\gamma}}\F(\mu(r))
 \mu(r)^6\,J^{2-\gamma}
\end{equation}
where $c = c(n,g_0,L,\gamma)>0$ and
$J$ is defined  as in  \eqref{def:J}     
\begin{equation}\label{def:J new}
 J = \int_{B_r} \eta^{\beta+4}v^{\beta+2} |\X v|^2\dx 
+ \mu(r)^4\frac{|B_r|^{1-\frac{1}{\gamma}}}{r^2} 
\Big(\int_{B_r} \eta^{\gamma\beta}v^{\gamma\beta}\dx\Big)^\frac{1}{\gamma}.
\end{equation}
Now, it follows from \eqref{est:K final} and  \eqref{est:I32} that 
$$ |I^2_3|\leq c(\beta+2)^{3-\gamma} 
\F(\mu(r))\mu(r)^{2\gamma}\,
\frac{|B_r|^\frac{\gamma-1}{2}}{r^\gamma}\,
J^\frac{2-\gamma}{2}
\Big(\int_{B_r} 
\eta^{\gamma\beta}v^{\gamma\beta}\dx\Big)^\frac{1}{2}.$$
By Young's inequality, we end up with
\begin{equation*}
\begin{aligned}
|I^2_3|\leq & \,\frac{c_0}{12}(\beta+2)\F(\mu(r)) J\\
 & + c(\beta+2)^{\frac{4}{\gamma}-1}
 \F(\mu(r))\mu(r)^4
\frac{|B_r|^{1-\frac{1}{\gamma}}}{r^2}
\Big(\int_{B_r} \eta^{\gamma\beta}v^{\gamma\beta}\dx\Big)^\frac{1}{\gamma},
\end{aligned}
\end{equation*}
where $c_0>0$ is the same constant as in \eqref{lem:claim}. 
Note that, with $ J$ as in \eqref{def:J new}, $I_3^2$ satisfies 
an estimate similar to \eqref{lem:claim}. Now the desired claim \eqref{lem:claim}
for $I_3$ follows, since both $I_3^1$ and $I_3^2$ satisfy similar estimates.  
This concludes the proof of the claim \eqref{lem:claim}, and hence the proof of the lemma.
\end{proof}
The following corollary follows from Lemma \ref{lem:main} by Moser's iteration. We  refer to \cite{Muk-Zhong} for the proof. 
\begin{Cor}\label{prop:case1}
There exists a constant $\theta = \theta(n,g_0,L)>0$ such that the following statements hold. If we have 
\begin{equation}\label{condition1}
 \vert\{x\in B_r : X_lu<\mu(r)/4\} \vert\leq \theta |B_r|
\end{equation}
for an index $l\in\{1,\ldots,2n\}$ and for a ball $B_r\subset\Omega$,  then 
 \[\inf_{B_{r/2}}X_lu\ge 3\mu(r)/16;\]
Analogously, if we have 
\begin{equation}\label{condition2}
\vert\{x\in B_r : X_lu>-\mu(r)/4\}\vert\leq \theta |B_r|,
\end{equation}
for an index $l\in\{1,\ldots,2n\}$ and for a ball $B_r\subset\Omega$,  then  
\[\sup_{B_{r/2}}X_lu \le\, -3\mu(r)/16.\]
\end{Cor}

\subsection{Proof of Theorem \ref{thm:holder}}\noindent
\\
At the end of this subsection, we provide the proof of Theorem \ref{thm:holder}. 
As before, we denote $u\in HW^{1,G}(\Om)$ as a weak solution of equation 
\eqref{eq:maineq} We fix a ball $B_{r_0}\subset\Omega$. For all balls
 $B_r, 0<r<r_0$, concentric to $B_{r_0}$, we denote for $l=1,2,...,2n,$
\[ \mu_l(r)=\sup_{B_r} \vert X_lu\vert, \quad \mu(r)=\max_{1\le l\le 2n}\mu_l(r),\]
and 
\[ \omega_l(r)=\osc_{B_r} X_lu, \quad \omega(r)=\max_{1\le l\le 2n}\omega_l(r).\]
We clearly have $\omega(r)\le 2\mu(r)$. For any function $w$, we define
$$ A_{k,\rho}^+(w) =\{ x\in B_\rho: (w(x)-k)^+=\max (w(x)-k,0)>0\};$$ 
and $A_{k,\rho}^-(w)$ is similarly defined. 

The following lemma is similar to Lemma 4.1 of \cite{Muk-Zhong} and Lemma 4.3 of \cite{Zhong}. For sake of completeness, we provide a proof in Appendix I. 
\begin{Lem}\label{lemma:cacci:k}
Let $B_{r_0}\subset\Omega$ be a  ball and $0<r<r_0/2$. Suppose that there is $\tau>0$ such that 
 \begin{equation}\label{comparable'}
 \vert \X u\vert\ge \tau \mu(r) \quad \text{in }\, A_{k,r}^+(X_l u)
 \end{equation}
 for an index $l\in\{1,2,...,2n\}$ and for a constant $k\in \R$. Then for any $q\ge 4$ and any $0<r^{\prime\prime}<r^\prime\le r$, we have 
\begin{equation}\label{HDG}
\begin{aligned}
 \int_{B_{r^{\prime\prime}}} &\weight | \X(X_lu-k)^{+}|^2\, dx\\  
 &\le  
 \frac{c}{(r^\prime-r^{\prime\prime})^2}\int_{B_{r^{\prime}}} \weight
 |(X_lu-k)^{+}|^2\, dx\,+\, cK| A^+_{k,r^\prime}(X_lu)|^{1-\frac{2} {q}}
\end{aligned}
\end{equation}
where $K = r_0^{-2}|B_{r_0}|^{2/q}\mu(r_0)^2\F(\mu(r_0))$ 
and $c\,=\,c(n,p,L,q, \tau)>0$. 
\end{Lem}
\begin{Rem}\label{rem:-version}
 Similarly, we can obtain an inequality, corresponding to (\ref{HDG}), 
 with 
 $(X_lu-k)^+$ replaced by $(X_lu-k)^-$ and
 $A^+_{k,r}(X_lu)$ replaced by $A^-_{k,r}(X_lu)$.
\end{Rem}
\begin{Lem}\label{lem:hold}
There exists a constant $s=s(n,g_0,L)\ge 0$ such that
for every $0<r\leq r_0/16$, we have the following, 
\begin{equation}\label{mur4r}
\omega(r) \le (1-2^{-s})\omega(8r) + 2^s\mu(r_0)\left(\frac{r}{r_0}\right)^\alpha,
\end{equation}
where $\alpha = 1/2$ when $0<g_0<1$ and $\alpha = 1/(1+g_0)$ when $g_0\geq 1$. 
\end{Lem}
\begin{proof}
To prove the lemma, 
we fix a ball $B_r$ concentric to $B_{r_0}$, such that $0<r<r_0/16$. 

Letting $\alpha = 1/2$ when $0<g_0<1$ and $\alpha = 1/(1+g_0)$ 
when $g_0\geq 1$, we may assume that
\begin{equation}\label{assum:mu}
\omega(r) \ge \mu(r_0)\left(\frac{r}{r_0}\right)^{\alpha},
\end{equation}
since, otherwise, \eqref{mur4r} is true with $s=0$. 
In the following, we assume that (\ref{assum:mu}) is true and we divide the 
proof into two cases.

{\it Case 1}. For at least one index $l\in\{1,\ldots,2n\}$, we have either
\begin{equation}\label{small1}
  |\{x\in B_{4r} : X_lu<\mu(4r)/4\}|
 \leq \theta |B_{4r}|
 \end{equation}
or
 \begin{equation}\label{small2}
  |\{x\in B_{4r}: X_lu>-\mu(4r)/4\}|
  \leq \theta |B_{4r}|, 
\end{equation}
where $\theta = \theta(n,g_0,L)>0$ is the constant in Corollary \ref{prop:case1}. 
Assume that (\ref{small1}) is true; the case (\ref{small2}) can be treated in the same way.
We apply Corollary \ref{prop:case1} to obtain that 
\[\vert X_l u\vert \ge 3\mu(4r)/16\quad \text{in }\ B_{2r}.\]
Thus we have
 \begin{equation}\label{comparable2}
  \vert\X u\vert\ge  3\mu(2r)/16\quad \text{in }\ B_{2r}.
 \end{equation}
Due to (\ref{comparable2}),  we can apply Lemma \ref{lemma:cacci:k} with $q = 2Q$ to obtain 
\begin{equation}\label{DG estimate}
\begin{aligned}
\int_{B_{r''}} | \X (X_iu-k)^+|^2\, dx
\,\le\, & \frac{c}{(r'-r'')^2}\int_{B_{r'}}|(X_iu-k)^+|^2\,dx \\
&+  cK \F(\mu(2r))^{-1}| A^+_{k,r'}(X_iu)|^{1-\frac{1} {Q}}
\end{aligned}
\end{equation}
where $K = r_0^{-2}
|B_{r_0}|^{1/Q}\mu(r_0)^2\F(\mu(r_0))$. 
The above inequality holds for all
$0<r''<r'\leq 2r,\ i\in\{1,\ldots,2n\}$ and all $k\in \R$, 
which means that for each $i$, $X_iu$ belongs to the De Giorgi class $DG^+(B_{2r})$, see \cite{Zhong} for details. The corresponding version of 
Lemma \ref{lemma:cacci:k} for $(X_i u-k)^-$, see Remark \ref{rem:-version}, 
shows that $X_iu$ also belong to $DG^-(B_{2r})$ and hence  
$X_iu$ belongs to $DG(B_{2r})$. Now we can apply Theorem 4.1 of \cite{Zhong} 
to conclude that there is $s_0=s_0(n,p,L)>0$ such that for each 
$i\in\{1,2,...,2n\}$
\begin{equation}\label{osc}
 \osc_{B_{r}} X_iu\le (1-2^{-s_0})\osc_{B_{2r}}
X_iu + c K^\frac{1}{2}\F(\mu(2r))^{-\frac{1}{2}} r^\frac{1}{2}.
\end{equation}
Now, from doubling property of $g$, see \eqref{eq:gG3} of Lemma \ref{lem:gandG prop}, we have 
$g(\mu(r_0))\leq \big(\frac{\mu(r_0)}{\mu(2r)}\big)^{g_0} g(\mu(2r))$ 
whenever $2r\leq r_0$ and hence 
$$ \F(\mu(r_0))/\F(\mu(2r)) \leq \big(\mu(r_0)/\mu(2r)\big)^{g_0-1}.$$
Thus, notice that when $0<g_0<1$, we have 
$$ \F(\mu(2r))^{-1} \leq  \F(\mu(r_0))^{-1}$$
and when $g_0\geq 1$, our assumption \eqref{assum:mu} with $\alpha =1/(1+g_0)$ gives 
\begin{align*}
\F(\mu(2r))^{-1} \leq \bigg(\frac{\mu(r_0)}{\mu(2r)}\bigg)^{g_0-1} \F(\mu(r_0))^{-1} &\leq 
2^{g_0-1}\F(\mu(r_0))^{-1}\bigg(\frac{\mu(r_0)}{\om(r)}\bigg)^{g_0-1}\\
&\leq 2^{g_0-1} \F(\mu(r_0))^{-1}\bigg(\frac{r}{r_0}\bigg)^\frac{1-g_0}{1+g_0}
\end{align*} 
where in the second inequality we used that $\mu(2r)\ge \omega(2r)/2\ge \omega(r)/2$.  
In both cases, we find that \eqref{osc} becomes
\begin{equation}\label{osc1}
 \osc_{B_{r}} X_iu \le (1-2^{-s_0})\osc_{B_{2r}}
X_iu+c\mu(r_0)\left(\frac{r}{r_0}\right)^\alpha,
\end{equation}
where $c=c(n,g_0,L)>0$,  $\alpha = 1/2$ when $0<g_0<1$ and $\alpha = 1/(1+g_0)$ when $g_0\geq 1$. This 
shows that the lemma holds in this case. 

{\it Case 2}. If Case 1 does not happen, then for every $i\in\{1,\ldots,2n\}$, we have   
\begin{equation}\label{case2 1}
  |\{x\in B_{4r} : X_iu<\mu(4r)/4\}|
  > \theta |B_{4r}|,
 \end{equation} 
 and
 \begin{equation}\label{case2 2}
 |\{x\in B_{4r}: X_iu>-\mu(4r)/4\}|
 >\theta |B_{4r}|,
\end{equation}
where $\theta = \theta(n,g_0,L)>0$ is the constant in Corollary \ref{prop:case1}. 

Note that on the set 
$\{x\in B_{8r}: X_iu >\mu(8r)/4 \}$, we trivially have
 \begin{equation}\label{comparable4}
 \vert \X u\vert\ge \mu(8r)/4\quad \text{in } A^+_{k,8r}(X_iu)
 \end{equation}
 for all $k\geq \mu(8r)/4$. 
Thus, we can apply Lemma \ref{lemma:cacci:k} with $q = 2Q$ to conclude that 
\begin{equation}\label{DG estimate 2}
\begin{aligned}
\int_{B_{r''}} | \X (X_iu-k)^+|^2\, dx
\,\le\, &\frac{c}{(r'-r'')^2}\int_{B_{r'}}|(X_iu-k)^+|^2\,dx \\
& + c K\,\F(\mu(8r))^{-1}| A^+_{k,r'}(X_iu)|^{1-\frac{1} {Q}}
\end{aligned}
\end{equation}
where $K = r_0^{-2}
|B_{r_0}|^{1/Q}\mu(r_0)^2\F(\mu(r_0))$, 
whenever $k\geq k_0= \mu(8r)/4 $ and $0<r^{\prime\prime}<r^\prime\le 8r$. 
The above inequality is true all $i\in\{ 1,2,...,2n\}$. We note that
\eqref{case2 1} trivially implies
\[|\{x\in B_{4r} : X_iu<\mu(8r)/4\}|
  > \theta |B_{4r}|.\] 
Now we can apply Lemma 4.2 of  \cite{Zhong} 
to conclude that there exists $s_1 = s_1(n,p,L)>0$ such that  
the following holds, 
\begin{equation}\label{sup est}
 \sup_{B_{2r}} X_iu \le\sup_{B_{8r}} X_i u -2^{-s_1}\big(\sup_{B_{8r}}
X_iu-\mu(8r)/4\big)+ cK^\frac{1}{2}\F(\mu(8r))^{-1/2}r^\frac{1}{2}.
\end{equation}
From (\ref{case2 2}), we can derive similarly, see Remark \ref{rem:-version}, that
\begin{equation}\label{inf est}
 \inf_{B_{2r}} X_iu \ge \inf_{B_{8r}} X_iu +2^{-s_1}\big(-\inf_{B_{8r}}
X_iu-\mu(8r)/4\big) - c K^\frac{1}{2}\F(\mu(8r))^{-1/2}r^\frac{1}{2}.
\end{equation}
The above two inequalities \eqref{sup est} and \eqref{inf est} yield
\[\osc_{B_{2r}} X_iu \le\ (1-2^{-s_1})\osc_{B_{8r}}
X_iu+ 2^{-s_1-1}\mu(8r)+c K^\frac{1}{2}\F(\mu(8r))^{-1/2}r^\frac{1}{2},
\]
and hence 
\begin{equation}\label{osc2}
\omega(2r) \leq \big(1-2^{-s_1}\big)\omega(8r)+2^{-s_1-1}\mu(8r)+c K^\frac{1}{2}\F(\mu(8r))^{-1/2} r^\frac{1}{2}.
\end{equation}
By using doubling 
condition of $g$ and the inequality $\mu(8r)\ge \omega(8r)/2\ge \omega(r)/2$ 
along with the assumption \eqref{assum:mu}, we proceed by the same argument as in the preceeding case, to conclude 
\[\omega(2r) \leq \big(1-2^{-s_1}\big)\omega(8r)+2^{-s_1-1}\mu(8r) +c\mu(r_0)\left(\frac{r}{r_0}\right)^\alpha\]
for $\alpha=1/2$ when $0<g_0<1$ and $\alpha=1/(1+g_0)$ when $g_0\geq 1$.
 
Now we notice that (\ref{case2 1}) implies that $\inf_{B_{4r}}X_iu\leq \mu(4r)/4$ 
and (\ref{case2 2}) implies that $\sup_{B_{4r}}X_iu\geq -\mu(4r)/4$ for every 
$i\in\{1,\ldots,2n\}$. Hence 
\[ \omega(8r)\ge \mu(8r)-\mu(4r)/4\ge 3\mu(8r)/4. \]
Then from the above two inequalities we arrive at
\[\omega(2r) \leq \big(1-2^{-s_1-2}\big)\omega(8r) +c\mu(r_0)\left(\frac{r}{r_0}\right)^\alpha,\]
where $c=c(n,g_0,L)>0$, $\alpha = 1/2$ when $0<g_0<1$ and $\alpha = 1/(1+g_0)$ when $g_0\geq 1$.
This shows that also in this case the lemma is true. 
Thus, the proof of the lemma follows from  
choice of $s = \max(0, s_0, s_1+2, \log_2 c)$. 
\end{proof}

\begin{proof}[Proof of Theorem \ref{thm:holder}]\noindent
\\
We first consider the apriori assumption \eqref{eq:ass} so that, equipped with this 
assumption, we have the above lemma, Lemma \ref{lem:hold}. Now, by an iteration on \eqref{mur4r}, it is easy to see that 
\begin{equation}\label{eq:muholder}
\om(r) \leq c\Big(\frac{r}{r_0}\Big)^\sigma \big[ \om(r_0/2) +\mu(r_0/2)\big]
\end{equation}
for some $\sigma=\sigma(n,g_0,L)\in(0,1),\ r\leq r_0/2$ and $c=c(n,g_0,L)>0$. Using  
\eqref{eq:muholder}, observe that 
\begin{equation}\label{hold1}
\begin{aligned}
\intav_{B_r}G(|X_l u - \{X_l u\}_{B_r}|)\dx \leq c\,G(\om_l(r))
&\leq c \,G\bigg(\Big(\frac{r}{r_0}\Big)^\sigma \big[ \om(r_0/2) +\mu(r_0/2)\big]
\bigg)\\
&\leq c\Big(\frac{r}{r_0}\Big)^\sigma \sup_{B_{r_0/2}}G(|\X u|) 
\end{aligned}
\end{equation}
where we have used \eqref{eq:jenap} for the first inequality and \eqref{eq:gG2} for the last inequality of the above. Hence 
from \eqref{eq:locbound}, we end up with 
\begin{equation}\label{hold2}
\intav_{B_r}G(|X_l u - \{X_l u\}_{B_r}|)\dx 
\leq c\Big(\frac{r}{r_0}\Big)^\sigma \intav_{B_{r_0}} G(|\X u|)\dx
\end{equation}
which gives us the estimate \eqref{eq:holder}. 

Now, to complete the proof, 
first we need to show that the estimate \eqref{hold2} is uniform, without the assumption 
\eqref{eq:ass}. This involves a standard approximation argument, using the 
following regularization, as constructed \cite{Muk0}; 
\begin{equation}\label{eq:regularization}
\F_\eps(t)=  \F\Big(\min\{\,t+\eps\,,\,1/\eps\,\}\Big) 
\quad\text{and}\quad 
\Aeps(z) =  \eta_\eps(|z|)\F_\eps(|z|)\,z + \Big(1-\eta_\eps(|z|)\Big)\A(z)
\end{equation}
where $0<\eps<1,\, \eta_\eps\in C^{\,0,1}([0,\infty))$ as in \cite{Muk0} and $\F(t)=g(t)/t$ for $g$ satisfying \eqref{eq:g prop} with $\delta>0$. 
Then, given $u\in HW^{1,G}(B_r)$
we consider $u_\eps$ that solves $\dvh (\A_\eps(\X u_\eps))=0$ and  
$u_\eps - u \in HW^{1,G}_0(B_r)$. We have $\A_\eps\to \A$ and $\F_\eps\to \F$ 
uniformly on compact subsets and $\F_\eps$ satisfies the assumption 
\eqref{eq:ass} with $m_1=\F(\eps)$ and $m_2=\F(1/\eps)$. Since the estimate 
\eqref{hold2} are independent of $m_1$ and $m_2$, hence the limit $\eps\to 0$ can be taken to obtain the uniform estimate, where the constant depends on 
$n,\delta,g_0,L$. 

Now, we show that the uniform estimate \eqref{hold2} implies that 
$X_lu$ is H\"older continuous for every $l\in\{1,\ldots,2n\}$. Using  
\eqref{eq:gG2} and Jensen's inequality on \eqref{hold2}, notice that 
\begin{equation}\label{hold3}
\begin{aligned}
\Big(\intav_{B_r}|X_l u &- \{X_l u\}_{B_r}|\dx\Big)
g\Big(\intav_{B_r}|X_l u - \{X_l u\}_{B_r}|\dx\Big)\\
&\leq (1+g_0)G \Big(\intav_{B_r}|X_l u - \{X_l u\}_{B_r}|\dx\Big)
\leq c\Big(\frac{r}{r_0}\Big)^\sigma \intav_{B_{r_0}} G(|\X u|)\dx
\end{aligned}
\end{equation}
for some $c=c(n,\delta,g_0,L)>0$. Now, observe that if 
$\intav_{B_r}|X_l u - \{X_l u\}_{B_r}|\dx\geq 1$ then, 
$$ \Big(\intav_{B_r}|X_l u- \{X_l u\}_{B_r}|\dx\Big)
g\Big(\intav_{B_r}|X_l u - \{X_l u\}_{B_r}|\dx\Big)
\geq g(1)\intav_{B_r}|X_l u - \{X_l u\}_{B_r}|\dx; $$
otherwise if $\intav_{B_r}|X_l u - \{X_l u\}_{B_r}|\dx\leq 1$, then 
from doubling condition
$$ \Big(\intav_{B_r}|X_l u- \{X_l u\}_{B_r}|\dx\Big)
g\Big(\intav_{B_r}|X_l u - \{X_l u\}_{B_r}|\dx\Big)
\geq g(1)\Big(\intav_{B_r}|X_l u - \{X_l u\}_{B_r}|\dx\Big)^{1+g_0}. $$
Notice that, both cases of the above when combined with \eqref{hold3}, 
yield
\begin{equation}\label{hold4}
\intav_{B_r}|X_l u- \{X_l u\}_{B_r}|\dx \leq C\Big(n,\delta,g_0,L,g(1), 
\|u\|_{HW^{1,G}(\Om)}\Big) 
\Big(\frac{r}{r_0}\Big)^\frac{\sigma}{1+g_0} 
\end{equation}
which implies that $X_lu\in \mathcal L^{1,Q+\sigma'}(B_r)$ and 
hence, recalling \eqref{eq:holdcamp}, $X_lu\in C^{\,0,\sigma'}(B_r)$ with $\sigma'=\sigma/(1+g_0)$ for some 
$\sigma=\sigma(n,g_0,L)\in(0,1)$. 
This completes the proof.
\end{proof}

\begin{Rem}\label{rem:camp}
Let $B_R\subset B_{R_0}\subset\subset\Om$ be concentric balls for 
$0<R<R_0$. 
As illustrated in the above proof, if $w\in HW^{1,G}(\Om)$ 
with $\|u\|_{HW^{1,G}(\Om)} \leq M$, satisfies the inequality 
$$ \int_{B_R}G(|\X w - \{\X w\}_{B_R}|)\dx 
\leq C (R/R_0)^\lambda $$ 
for some positive constants $C= C(n,\delta,g_0,R_0,M)>0$ and $\lambda\in(0,Q+1)$ with  $Q=2n+2$, then we have 
$\X w \in \mathcal L^{1,\lambda'}(B_R,\R^{2n})$; where if $\lambda \in (0,Q)$ 
then $\lambda' =\lambda$ and if $\lambda\in (Q,Q+1)$ then 
$\lambda' = Q+(\lambda-Q)/(1+g_0)$. This shall be used in the next 
section. 
\end{Rem}

\section{$C^{1,\alpha}$-regularity of weak solutions}\label{sec:c1alpha}
In this section, we prove Theorem \ref{thm:c1alpha}. In a fixed 
subdomain $\Om'$ compactly contained in $\Om$, we show 
that the weak solutions are locally $C^{1,\beta}$ in $\Om'$. 
The proof is 
standard, based on the results of the preceeding section and 
a Campanato type perturbation technique. Similar arguments in the 
Euclidean setting, can be found in \cite{Dib, Gia-Giu--div, Lieb--gen}, etc.
\subsection{The perturbation argument}\noindent
\\
Given $\Om'\subset\subset\Om$, we fix $x_0\in\Om'$ and a ball 
$B_R=B_R(x_0)\subset \Om'$ for $R\leq R_0= \frac{1}{2}\dist(\Om',\del\Om)$ and consider 
$u\in HW^{1,G}(B_R)\cap L^\infty(B_R)$ as weak solution of $\Q u=0$ in 
$B_R$, where $\Q$ is 
defined as in \eqref{eq:qop}. We recall the structure conditions for Theorem \ref{thm:c1alpha}, as follows;
\begin{align}
 \label{DA} &\frac{g(|p|)}{|p|}\,|\xi|^2
 \leq \,\inp{D_p\,A(x,z,p)\,\xi}{\xi}\,\leq L\,\frac{g(|p|)}{|p|}\,|\xi|^2;\\
 \label{A0} &|A(x,z,p)-A(y,w,p)|\,\leq
 \,L'\big(1+g(|p|)\big)\Big( |x-y|^\alpha + |z-w|^\alpha\Big);\\
 \label{B0} &|B(x,z,p)|\,\leq\, L'\big(1+g(|p|)\big)|p|
\end{align}
for all $(x,z,p)\in\Om\times \R\times \R^{2n}$ and the matrix 
$D_p\,A(x,z,p)$ is symmetric. 
In addition, we recall the hypothesis of 
Theorem \ref{thm:c1alpha} that, there exists $M_0>0$ such that $\|u\|\leq M_0$ in 
$\Om'$. 

From structure condition \eqref{DA}, it is not difficult to check that $A(x,z,p)$ satisfies conditions 
reminiscent of \eqref{A1} and \eqref{A2}; the condition on variable $z$ for \eqref{A1} and \eqref{A2} are absolved in the constants $L$ and $L'$, since 
the solution $u$ is bounded.  However, the condition \eqref{B0} on $B$ is more 
relaxed than \eqref{B} and \eqref{B'}, which is necessary for $C^{1,\beta}$-regularity. 

Thus, this allows us to apply Theorem \ref{thm:c0alpha} 
and conclude $u$ is H\"older continuous with 
\begin{equation}\label{oscu}
\osc_{B_R} u \leq \theta(R) = \gamma R^\tau
\end{equation}
for some constant $\gamma=\gamma(M_0,\dist(\Om',\del\Om))>0$ and 
$\tau\in(0,1)$ can be chosen to be as small as required. Here onwards, we   
suppress the dependence of the data $n,\delta,g_0,\alpha, L,L',\dist(\Om',\del\Om)$; all positive constants depending on these shall be denoted as $c$ and the constants dependent further on $g(1), M_0$ in addition,  
shall be denoted as $C$ throughout this subsection.

The proof of Theorem \ref{thm:c1alpha} involves a standard freezing technique as followed previously in the Euclidean setting in \cite{Lieb--gen}. However, the integral oscillation estimate \eqref{eq:holder} of the previous section is weaker than that of the Euclidean setting (see \cite[(5.3b) of p. 339]{Lieb--gen}, which is, in fact, false in the setting of $\hh^n$ due to the non-zero terms from the commutators). Therefore, \cite{Lieb--gen} can not be followed entirely. For the present case, the proof is carried out with an extra step in the argument, using Lemma \ref{lem:intosc} and Proposition \ref{prop:Aml} below. 

Let us denote $\A:\R^{2n}\to \R^{2n}$ as
\begin{equation}\label{eq:Apert}
 \A(p) = A(x_0,u(x_0),p),
\end{equation}
so that from \eqref{DA}, $\A$ satisfies the structure condition 
\eqref{eq:str cond diff} and hence also the monotonicity and ellipticity conditions \eqref{eq:monotone} and \eqref{eq:elliptic} (with possible dependence on $g_0$ and $\delta$). Hence, for the problem
\begin{equation}\label{eq:probpert}
 \begin{cases}
  \dvh (\A(\X \tu))=  0\ \ \text{in}\ B_R;\\
 \ \tu - u\in HW^{1,G}_0(B_R). 
 \end{cases}
\end{equation}
we can use the monotonicity inequalities and uniform estimates from Section \ref{sec:LocalH}. 
\begin{Lem}\label{lem:pert}
 If $u\in HW^{1,G}(B_R)\cap C(\bar B_R)$ is given, then there exists a unique weak 
 solution $\tu\in HW^{1,G}(B_R)\cap C(\bar B_R)$ for the problem 
 \eqref{eq:probpert}, which satisfies the following:
 \begin{align}
  \label{pert1}(i)&\ \sup_{B_R}|u-\tu|\leq \osc_{B_R} u\ ;\\
  \label{pert2}(ii)&\ \int_{B_R} G(|\X\tu|)\dx 
  \leq c\int_{B_R} G(|\X u|)\dx.
 \end{align}
\end{Lem}
\begin{proof}
 Existence and uniqueness is standard from monotonicity of $\A$, we refer to   \cite{Muk0} for more details. Also, \eqref{pert1} follows easily 
 from Comparison principle and the fact that 
 $$\inf_{\del B_R}u\leq \tu\leq \sup_{\del B_R}u\quad \text{in}\ B_R, $$ 
which is easy to show by considering $\varphi= (\tu-\sup_{\del B_R}u)^+$ (and 
similarly the other case) as a test function for \eqref{eq:probpert}, see Lemma 5.1 in \cite{Dib}. 

The proof of \eqref{pert2} is also standard. Using test function
$\varphi= \tu -u$ on \eqref{eq:probpert}, we get 
\begin{equation}\label{eq:el1}
 \int_{B_R}\inp{\A(\X \tu)}{\X \tu}\dx 
=  \int_{B_R}\inp{\A(\X \tu)}{\X u}\dx.
\end{equation} 
Now we choose $k=k(\delta, g_0,L)>0$ such that combining ellipticity \eqref{eq:elliptic} and boundedness of $\A$, we have 
$\inp{\A(p)}{p}\geq (2/k) |p||\A(p)|$. Hence, we obtain 
\begin{equation*}
  \begin{aligned}
   \int_{B_R}\inp{\A(\X \tu)}{\X u}\dx 
   &\leq \frac{1}{k}\int_{|\X \tu|\geq k|\X u|} |\A(\X \tu)||\X \tu|\dx 
   \,+\, \int_{|\X \tu|< k|\X u|} |\A(\X \tu)||\X u|\dx\\
   &\leq \frac{1}{2}\int_{B_R}\inp{\A(\X \tu)}{\X \tu}\dx 
   \,+\,  k^{g_0}c\int_{B_R}g(|\X u|)\,|\X u|\dx.
  \end{aligned}
 \end{equation*}
which combined with \eqref{eq:el1} and the ellipticity \eqref{eq:elliptic}, concludes 
the proof. 
\end{proof}

We require the following comparison lemma. 
\begin{Lem}\label{lem:intosc}
There exists 
$\sigma=\sigma(n,g_0,L)\in(0,1)$ and $c=c(n,g_0,\delta, L)>0$ such that, for every 
$0<\varrho < R/2 $, the following estimate holds:
\begin{equation*}
\begin{aligned}
\intav_{B_\varrho}G(|\X u - \{\X u\}_{B_\varrho}|)\dx &\,\leq\, c\Big(\frac{\varrho}{R}\Big)^\sigma  \intav_{B_{ R}}G(|\X u|)\dx + c\Big(\frac{R}{\varrho}\Big)^Q \intav_{B_R} G(|\X u-\X\tu|)\dx   .
\end{aligned}
\end{equation*}
\end{Lem}
\begin{proof}
From \eqref{eq:jenap} and triangle inequality, we have 
\begin{equation}\label{eq:io1}
\begin{aligned}
\intav_{B_\varrho} G(|\X u - \{\X u\}_{B_\varrho}|)\dx
\leq c\intav_{B_\varrho}G(|\X \tilde u - \{\X \tilde u\}_{B_\varrho}|)\dx 
+c\intav_{B_\varrho}G(|\X u - \X \tilde u|)\dx . 
\end{aligned}
\end{equation}
Now, we shall estimate both terms of the right hand side of \eqref{eq:io1} 
seperately. 

The Theorem \ref{thm:holder} being proved in the previous section, we estimate 
the first term of \eqref{eq:io1} using the estimate \eqref{eq:holder} for $\X\tilde u$ as 
\begin{equation*}
\begin{aligned}
\intav_{B_\varrho}G(|\X \tilde u - \{\X \tilde u\}_{B_\varrho}|)\dx &\leq c\Big(\frac{\varrho}{R}\Big)^\sigma 
\intav_{B_{R}}G(|\X \tilde u|)\dx \\
&\leq  c\Big(\frac{\varrho}{R}\Big)^\sigma \intav_{B_{R}}G(|\X u|)\dx  
+  c\Big(\frac{\varrho}{R}\Big)^\sigma \intav_{B_{R}}G(|\X \tilde u-\X u|)\dx 
\end{aligned}
\end{equation*}
 
The second term of \eqref{eq:io1} is estimated simply as 
\begin{equation*}
\intav_{B_\varrho}G(|\X u - \X \tilde u|)\dx \leq 
c\Big(\frac{R}{\varrho}\Big)^Q\intav_{B_R}G(|\X u - \X \tilde u|)\dx,  
\end{equation*}
and combining estimates of both terms of \eqref{eq:io1}, 
the proof is finished.
\end{proof}

To proceed with the proof of Theorem \ref{thm:c1alpha}, we shall need the 
following technical lemma which is a variant of a lemma of 
Campanato \cite{Camp}. This is a fundamental lemma which has been extensively used in the literature. 
We refer to \cite{Gia-Giu--sharp} or 
{\cite[Lemma 2.1]{Gia}} for a proof.
 \begin{Lem}\label{lem:camp}
 Let $\phi :(0,\infty)\to [0,\infty)$ be a non-decreasing function and $A,B > 1, \alpha >0 $ be fixed constants.  
 Suppose that for any $ \rho < r\leq R_0$ and $\epsilon>0$, we have 
 $$ \phi(\rho)  \leq A\left[\left(\frac{\rho}{r}\right)^\alpha + \kappa\right]\phi(r)  + Br^{\alpha-\epsilon};$$ 
 then there exists a constant $\kappa_0=\kappa_0(\alpha,A,B)>0$ such that if $ \kappa < \kappa_0$, we have
 $$ \phi(\rho)  \leq c\left(\frac{\rho}{r}\right)^{\alpha-\epsilon}\big[\phi(r) + 
  Br^{\alpha-\epsilon}\big] $$
  for all $\rho<r\leq R_0$, where $c=c(\alpha,\epsilon,A)>0$ is a constant.
\end{Lem}
The proof of Theorem \ref{thm:c1alpha} requires the following intermediary step, which can be regarded as an almost-Lipschitz estimate. This is a consequence of the uniform Lipschitz estimate of $\tilde u$ from \cite{Muk0} and the above perturbation lemma. 
\begin{Prop}\label{prop:Aml}
Let $u\in HW^{1,G}(\Om)$ be a weak solution of $\Q u=0$, then for any $0<\eps<1$ and  $0<r\leq R\leq R_0/2$, we have 
$\X u\in \mathcal L^{1,Q-\eps}_\loc (\Om)$ and 
\begin{equation}\label{eq:Aml}
\int_{B_r}G(|\X u|)\dx\leq c \left(\frac{r}{R}\right)^{Q-\eps}
\bigg[\int_{B_R}G(|\X u|)\dx + CR^{Q-\eps}\bigg]. 
\end{equation}
\end{Prop}
\begin{proof}
 For $B_R\subset \Om'\subset\subset \Om$, we have $\|u\|\leq M_0$ in $\bar B_R$ and up to a representative we 
can regard that $u\in HW^{1,G}(B_R)\cap C(\bar B_R)$. Let us denote 
\begin{equation}\label{eq:Ipert}
I = \int_{B_R}\inp{\A(\X u)}{(\X u -\X \tu)} \dx,
\end{equation}
where $\A$ is as in \eqref{eq:Apert} and $\tu\in HW^{1,G}(B_R)\cap C(\bar B_R)$ is the weak solution of \eqref{eq:probpert}. Since $u = \tu$ in $\del B_R$, the function $u-\tu$ can be used to test the equations satisfied by $u$ and 
$\tu$, which shall be used to estimate $I$ to obtain both lower and upper bounds. 

First, using $u-\tu$ as test function for $\Q u = 0$, we obtain 
\begin{equation}\label{pI1}
\begin{aligned}
I &=  \int_{B_R}\inp{A(x_0,u(x_0),\Xu)-A(x,u,\Xu)}{(\Xu-\X \tu)}\dx \\
  &\qquad\qquad +\int_{B_R}B(x,u,\Xu)(u-\tu)\dx\\
  &\leq c\big(R^\alpha+\theta(R)^\alpha \big)
  \int_{B_R} g(1+|\X u|)|\X u -\X \tu|\dx\\
  &\qquad\qquad + c\,\theta(R)\int_{B_R} g(1+|\X u|)|\X u|\dx
\end{aligned}
\end{equation}
with $\theta(R)$ as in \eqref{oscu}, 
where we have used structure condition \eqref{A0} and \eqref{B0} for the first term and \eqref{pert1} for the second term of the right hand side of \eqref{pI1}. Now we 
use \eqref{eq:gG5} of Lemma \ref{lem:gandG prop} and \eqref{pert2} of 
Lemma \ref{lem:pert} to estimate the first term of the above and obtain 
that 
\begin{equation}\label{Iless}
I \leq c\,\theta(R)^\alpha \int_{B_R} G(1+|\X u|)\dx.
\end{equation}

Secondly, to obtain the upper bound for $I$, we shall use the monotonicity 
inequality \eqref{eq:monotone}. Let us denote 
$S_1=\set{x\in B_R}{|\X u-\X\tu|\leq 2|\X u|}$ and 
$S_2=\set{x\in B_R}{|\X u-\X\tu|> 2|\X u|}$. Taking 
$u-\tu$ as test function for \eqref{eq:probpert} and using \eqref{eq:monotone}, we obtain 
\begin{equation}\label{Igrt}
\begin{aligned}
I &= \int_{B_R}\inp{\A(\X u)-\A(\X\tu)}{(\X u-\X\tu)}\dx\\
&\geq c\int_{S_1}\F(|\X u|)|\X u-\X\tu|^2\dx +c\int_{S_2}
\F(|\X u - \X\tu|)|\X u-\X\tu|^2\dx
\end{aligned}
\end{equation}
Recalling $G(t)\leq t^2\F(t)$ from \eqref{eq:gG2}, we have from 
\eqref{Iless} and \eqref{Igrt}, 
that 
\begin{equation}\label{s2}
\int_{S_2} G(|\X u-\X\tu|)\dx \leq c\,\theta(R)^\alpha \int_{B_R} G(1+|\X u|)\dx.
\end{equation}
Now since $|\X u-\X\tu|\leq 2|\X u|$ in $S_1$ by definition, we obtain the following from 
\eqref{eq:gG2}, monotonicity of $g$ and H\"older's inequality; 
\begin{equation}\label{s1}
 \begin{aligned}
  \int_{S_1}G(|\Xu-\X \tu|)\dx &\leq c\,\Big(\int_{S_1}\F(|\Xu|)|\Xu-\X \tu|^2\dx\Big)^\frac{1}{2}
  \Big(\int_{S_1}G(|\Xu|)\dx\Big)^\frac{1}{2}\\
  &\leq c\,\theta(R)^{\alpha/2}\int_{B_R}G(1+|\Xu|)\dx
 \end{aligned}
\end{equation}
where the latter inequality of the above follows from \eqref{Iless} and \eqref{Igrt}. Now, we add \eqref{s2} and \eqref{s1} to obtain the estimate 
of the integral over whole of $B_R$,
\begin{equation}\label{Gineq1}
\int_{B_R} G(|\X u-\X\tu|)\dx \leq c\,\theta(R)^{\alpha/2} \int_{B_R} G(1+|\X u|)\dx.
\end{equation}
Recalling \eqref{eq:locbound} and \eqref{pert2}, note that for any 
$0<r\leq R/2$, we have 
$$ \int_{B_r} G(|\X \tu|) \dx 
\leq r^Q \sup_{B_{R/2}} G(|\X \tu|) \leq c\Big(\frac{r}{R}\Big)^Q 
\int_{B_R} G(|\X \tu|)\dx \leq c\Big(\frac{r}{R}\Big)^Q 
\int_{B_R} G(|\X u|)\dx.$$
where $Q=2n+2$. 
Combining the above with \eqref{Gineq1}, we obtain 
\begin{equation}\label{campform}
 \int_{B_r} G(|\X u|) \dx \,\leq\, c\left(\frac{r}{R}\right)^Q \int_{B_R} G(|\X u|) \dx + c\,\theta(R)^{\alpha/2} \int_{B_R} G(1+|\X u|)\dx. 
\end{equation}
Now, we follow the bootstrap technique of Giaquinta-Giusti \cite{Gia-Giu--div}. 

For $0<\rho\leq R_0$, let us denote
$ \Phi(\rho) =  \int_{B_\rho} G(|\X u|) \dx $, so that we rewrite \eqref{campform} as 
\begin{equation}\label{cam}
\Phi(r) \leq c\left(\frac{r}{R}\right)^Q \Phi(R) 
+ cR^\vartheta \int_{B_R} G(1+|\X u|)\dx
\end{equation}
where $\vartheta = \tau\alpha/2$ with $\tau\in(0,1)$ as in \eqref{oscu}. We 
proceed by induction, with the hyposthesis 
\begin{equation}\label{hypk}
\int_{B_R} G(1+|\X u|)\dx \leq CR^{(k-1)\vartheta} 
\quad \text{for some}\ k\in\N,\ k\vartheta <Q.
\end{equation}
The hypothesis clearly holds for $k=1$. Assuming the hypothesis \eqref{hypk} holds for some $k\in\N$, first notice that by virtue of \eqref{eq:jenap}, we have 
$$ \int_{B_R} G(|\X u -\{\X u\}_{B_R}|)\dx \leq CR^{(k-1)\vartheta}  $$
which further implies that $\X u\in \mathcal L^{1,(k-1)\vartheta}(\Om',\R^{2n})$, 
see Remark \ref{rem:camp}. Now using \eqref{hypk} in \eqref{cam}, 
we apply Lemma \ref{lem:camp} to obtain that 
$$ \Phi(R) \leq c\left(\frac{R}{R_0}\right)^{k\vartheta}\big[ \Phi(R_0) + C\big],  $$
which, from definition of $\Phi$, implies the hypothesis \eqref{hypk} for 
$k+1$ and $\X u\in \mathcal L^{1,k\vartheta}(\Om',\R^{2n})$ from Remark \ref{rem:camp}. We choose 
can choose $\vartheta$ small enough and carry on a finite induction 
for $k=0,1,\ldots m$ where $m$ is chosen such that 
$m\vartheta <Q \leq (m+1)\vartheta <Q+1$. Thus,  
after the last induction step, we conclude that
$$ \Phi(r) \leq  c\left(\frac{r}{R}\right)^{m\vartheta}\big[ \Phi(R) + CR^{m\vartheta}\big]  $$
and $\X u\in \mathcal L^{1,m\vartheta}(\Om',\R^{2n})$. Given any $0<\eps<1$, we can choose $\vartheta$ small enough such that 
$Q-\eps\leq m\vartheta$ and the proof is finished.
\end{proof}
Furthermore, using the estimate \eqref{Gineq1} above together with Lemma \ref{lem:intosc}, we can get 
\begin{equation}\label{eq:hold0}
\begin{aligned}
\int_{B_\varrho}G(|\X u - \{\X u\}_{B_\varrho}|)\dx &\leq c\Big(\frac{\varrho}{R}\Big)^{Q+\sigma}\int_{B_{R}}G(|\X u|)\dx 
+ c \int_{B_R} G(|\X u-\X\tu|)\dx\\
&\leq c\Big(\frac{\varrho}{R}\Big)^{Q+\sigma}\int_{B_{R}}G(|\X u|)\dx 
+  cR^\theta \int_{B_R} G(1+|\X u|)\dx. 
\end{aligned}
\end{equation}
where $\theta>0$ is chosen similarly as in the above proof, which can be made small enough. This shall be required to prove Theorem \ref{thm:c1alpha} along with the estimate of Proposition \ref{prop:Aml}. 

\begin{proof}[Proof of Theorem \ref{thm:c1alpha}]\noindent

Let us take $0<\varrho \leq r <R_0/2$ and we rewrite \eqref{eq:hold0} as  
\begin{equation}\label{eq:hold1}
\begin{aligned}
\int_{B_\varrho}G(|\X u - \{\X u\}_{B_\varrho}|)\dx \leq c\Big(\frac{\varrho}{r}\Big)^{Q+\sigma}\int_{B_{r}}G(|\X u|)\dx 
+  c\,r^\theta \int_{B_r} G(1+|\X u|)\dx.
\end{aligned}
\end{equation}
From \eqref{eq:Aml} of Proposition \ref{prop:Aml} we have, 
\begin{equation}\label{eq:Aml2}
\int_{B_r}G(|\X u|)\dx\leq c \left(\frac{r}{R_0}\right)^{Q-\eps}
\bigg[\int_{B_{R_0}}G(|\X u|)\dx + CR_0^{Q-\eps}\bigg]. 
\end{equation}
Using \eqref{eq:Aml2} on \eqref{eq:hold1} to obtain the following estimate, 
\begin{equation}\label{hl1}
\begin{aligned}
\int_{B_\varrho}G(|\X u - \{\X u\}_{B_\varrho}|)\dx &\leq\, c 
\Big(\frac{\varrho^{Q+\sigma}}{r^{\sigma+\eps} R_0^{Q-\eps}}\Big)
\bigg[\int_{B_{R_0}}G(|\X u|)\dx + C R_0^{Q-\eps}\bigg] + Cr^{Q+\theta-\eps} \\
&\leq C\Big(\frac{\varrho^{Q+\sigma}}{r^{\sigma+\eps}} + r^{Q+\theta_0}\Big)
\end{aligned}
\end{equation}
where $\theta_0=\theta-\eps>0$ is chosen with a choice of a small enough $\eps>0$. For some $0<\kappa<1$ we rewrite the above with the choice $r= \varrho^\kappa$ to have 
\begin{equation}\label{hl2}
\begin{aligned}
\int_{B_\varrho}G(|\X u - \{\X u\}_{B_\varrho}|)\dx \leq C\big(\varrho^{Q+(1-\kappa)\sigma-\kappa\eps} + \varrho^{\kappa(Q+\theta_0)}\big)
\leq C\varrho^{Q+\gamma},
\end{aligned}
\end{equation}
where the latter inequality follows when $Q+\gamma \leq \min\{Q+(1-\kappa)\sigma-\kappa\eps\,,\, \kappa(Q+\theta_0)\}$; 
indeed we can make sure that this is true with the choice 
of $\kappa = \kappa(\gamma)$ such that 
$$ \frac{Q+\gamma}{Q+\theta_0} \,\leq\, \kappa \,\leq\, \frac{\sigma-\gamma}{\sigma+\eps},$$
for any $0<\gamma< (\sigma\theta_0-Q\eps)/(Q+\sigma+\theta_0 +\eps)$, where $0<\eps<\sigma\theta_0/Q$. Furthermore, having $\gamma, \eps$ small enough, $\kappa=\kappa(\gamma)$ can be chosen close enough to $1$ and therefore, we can make sure $\varrho^\kappa< R_0/2$, whenever $0<\varrho<R_0/2$. 
Thus, from \eqref{hl2} we have obtained 
$$ \intav_{B_\varrho}G(|\X u - \{\X u\}_{B_\varrho}|)\dx \leq C\varrho^\gamma, $$ 
for any $0<\varrho<R_0/2$, which imlpies $\X u \in C^{0,\beta}(\Omega',\R^{2n})$ with $\beta=\gamma/(1+g_0)$ from Remark \ref{rem:camp} and the proof is finished. 
\end{proof}

\subsection{Concluding Remarks}\label{remarks}\noindent
\\
Here we discuss some possible extensions of the structure conditions that can be included and results similar to the above can be obtained with minor 
modifications of the arguments.
\begin{enumerate}
\item Any dependence of $x$ in structure conditions for $A(x,z,p)$ and $B(x,z,p)$ has been suppressed so far, for sake of simplicity. However, we remark that 
for some given non-negative measurable functions $a_1,a_2,a_4,a_5,b_1,b_2$, 
the structure condition
\begin{equation*}
\begin{aligned}
 &\inp{A(x,z,p)}{p}\geq |p|g(|p|)
 -a_1(x)\,g\bigg(\frac{|z|}{R}\bigg)\frac{|z|}{R}-a_2(x);\\
   &|A(x,z,p)| \leq a_3\, g(|p|) + a_4(x) \,g\bigg(\frac{|z|}{R}\bigg)
   +a_5(x) ;\\
    &|B(x,z,p)|\leq \frac{1}{R}
  \bigg[b_0\,g(|p|) + b_1(x)\,g\bigg(\frac{|z|}{R}\bigg) + b_2(x) \bigg],
\end{aligned}
\end{equation*}
can also be considered for obtaining the Harnack inequalities. In this case, we 
would require $a_1,a_2,a_4,a_5,b_1,b_2\in L^q_\loc(\Om)$ for some $q>Q$. 
Similar arguments can be carried out with a choice of $\chi>0$, such that 
$\|a_5\|_{L^q(B_R)}+\|b_2\|_{L^q(B_R)}\leq g(\chi)$ and 
$\|a_2\|_{L^q(B_R)}\leq g(\chi)\chi$. We refer to \cite{Lieb--gen} and 
\cite{C-D-G} for more details of such cases. \\

\item The function $g(t)/t$ in the growth conditions can be replaced by 
$f(t)$, where $f$ is a continuous doubling positive function on 
$(0,\infty)$ and $t\mapsto f(t)t^{1-\delta}$ 
is non-decreasing. A $C^1$-function $\tilde g$ can be found 
satisfying \eqref{eq:g prop} and $\tilde g(t) \sim tf(t)$(see 
{\cite[Lemma 1.6]{Lieb--gen}}), which is sufficient 
to carry out all of the above arguments.  
\end{enumerate}

\section*{Appendix I}
\begin{proof}[Proof of Lemma \ref{lem:start}]
Fix $l\in\{ 1,2,...,n\}$ and $\beta\ge 0$. Let $\eta\in C^\infty_0(\Omega)$ be a non-negative cut-off function. 
Using 
$
\varphi=\eta^{\beta+2}v^{\beta+2} |\X u|^2 X_lu
$
as a test-function in equation \eqref{eq:Xu}, we get
\begin{equation}\label{eq1}
\begin{aligned}
\int_\Omega &\sum_{i,j=1}^{2n} \eta^{\beta+2}v^{\beta+2}D_j\A_i(\X u)X_jX_iuX_i\big( |\X u|^2 X_lu\big)\, dx\\
= & -(\beta+2)\int_\Omega\sum_{i,j=1}^{2n} \eta^{\beta+1} v^{\beta+2}|\X u|^2 X_lu D_j\A_i(\X u)X_jX_luX_i\eta\, dx\\
& -(\beta+2)\int_\Omega \sum_{i,j=1}^{2n} \eta^{\beta+2} v^{\beta+1}
|\X u|^2 X_l uD_j\A_i(\X u)X_iX_luX_iv\, dx\\
& -\int_\Omega\sum_{i=1}^{2n} D_i\A_{n+l}(\X u) TuX_i\big(\eta^{\beta+2}
v^{\beta+2}|\X u|^2 X_lu\big)\, dx\\
& +\int_\Omega T\big(\A_{n+l}(\X u)\big)\eta^{\beta+2}v^{\beta+2}|\X u|^2 X_lu\, dx\\
= &\, I_1^l+I_2^l+I_3^l+I_4^l.
\end{aligned}
\end{equation}
Here we denote the integrals in the right hand side of \eqref{eq:Xu} by 
$I_1^l, I_2^l, I_3^l$ and $I_4^l$ in order respectively. Similarly 
for all $l\in\{n+1,n+2,...,2n\}$, from equation \eqref{eq:Xu other}, we have 
\begin{equation}\label{eq2}
\begin{aligned}
\int_\Omega &\sum_{i,j=1}^{2n} \eta^{\beta+2}v^{\beta+2}D_j\A_i(\X u)X_jX_iuX_i\big( |\X u|^2 X_lu\big)\, dx\\
= & -(\beta+2)\int_\Omega\sum_{i,j=1}^{2n} \eta^{\beta+1} v^{\beta+2}|\X u|^2 X_lu D_j\A_i(\X u)X_jX_luX_i\eta\, dx\\
& -(\beta+2)\int_\Omega \sum_{i,j=1}^{2n} \eta^{\beta+2} v^{\beta+1}
|\X u|^2 X_l uD_j\A_i(\X u)X_iX_luX_iv\, dx\\
& +\int_\Omega\sum_{i=1}^{2n} D_i\A_{l-n}(\X u) TuX_i\big(\eta^{\beta+2}
v^{\beta+2}|\X u|^2 X_lu\big)\, dx\\
& -\int_\Omega T\big(\A_{l-n}(\X u)\big)\eta^{\beta+2}v^{\beta+2}|\X u|^2 X_lu\, dx\\
= &\,  I_1^l+I_2^l+I_3^l+I_4^l.
\end{aligned}
\end{equation}
Again we denote the integrals in the right hand side of (\ref{eq2}) by 
$I_1^l, I_2^l, I_3^l$ and $I_4^l$ in order respectively. 
Summing up the above equation (\ref{eq1}) and (\ref{eq2}) for all $l$ from
$1$ to $2n$, we end up with
\begin{equation}\label{equ3}
\int_\Omega \sum_{i,j,l} \eta^{\beta+2}v^{\beta+2}D_j\A_i(\X u)X_jX_iuX_i\big( |\X u|^2 X_lu\big)\, dx=  \sum_l \sum_{m=1}^4 I_m^l, 
\end{equation}
where all sums for $i,j,l$ are from $1$ to $2n$. 

In the following, we estimate both sides of (\ref{equ3}). For the left hand of (\ref{equ3}), note that
\[ X_i\big(|\X u|^2 X_lu\big)=|\X u|^2 X_iX_l u+X_i(|\X u|^2)X_lu.\]
Then by the structure condition \eqref{eq:str cond diff}, we have that
\[ \sum_{i,j,l}D_j\A_i(\X u)X_jX_luX_i\big(|\X u|^2X_lu\big)\ge 
\weight |\X u|^2|\X\X u|^2,\]
which gives us the following estimate for the left hand side of \eqref{equ3}
\begin{equation}\label{est:left}
\text{left of } \eqref{equ3} \ge \int_\Omega \eta^{\beta+2}v^{\beta+2}
\weight |\X u|^2|\X\X u|^2\, dx.
\end{equation}
Now we estimate the right hand side of (\ref{equ3}). We will show that
$I_m^l$ satisfies the following estimate for each $l=1,2,...,2n$ and each $m=1,2,3,4$
\begin{equation}\label{est1}
\begin{aligned}
| I_m^l|\le & \, \frac{1}{36n}\int_\Omega \eta^{\beta+2}v^{\beta+2}\weight |\X u|^2|\X\X u|^2\,dx\\
 & + c (\beta+2)^2\int_\Omega \eta^\beta \big(|\X \eta|^2+\eta| T\eta|\big) v^{\beta+2}
\weight |\X u|^4 \, dx\\
& +  c(\beta+2)^2\int_\Omega \eta^{\beta+2} v^ \beta\weight |\X u|^4| \X v|^2\, dx\\
& +
c\int_\Omega\eta^{\beta+2} v^{\beta+2}\weight |\X u|^2| Tu|^2\, dx,
\end{aligned}
\end{equation}
where $c=c(n,g_0,L)>0$. Then the lemma follows from the above estimates 
(\ref{est:left}) and (\ref{est1})
for both sides of (\ref{equ3}). The proof of the lemma is finished, modulo
the proof of (\ref{est1}). In the rest, we prove (\ref{est1}) in the order of $m=1, 2,3,4$.

First, when $m=1$, we have for $I_1^l, l=1,2,...,2n$, by the structure condition \eqref{eq:str cond diff} that
\[
| I_1^l| \le c(\beta+2)\int_\Omega \eta^{\beta+1}| \X \eta| v^{\beta+2}\weight |\X u|^3|\X\X u|\, dx,
\]
from which it follows by Young's inequality that
\begin{equation}\label{est:I1}
\begin{aligned}
| I_1^l|\le &\, \frac{1}{36n}\int_\Omega \eta^{\beta+2}v^{\beta+2}
\weight |\X u|^2|\X\X u|^2\,dx\\
 & + c(\beta+2)^2\int_\Omega \eta^\beta |\X \eta|^2 v^{\beta+2}
\weight |\X u|^4\, dx.
\end{aligned}
\end{equation}
Thus (\ref{est1}) holds for $I_1^l, l=1,2,...,2n$. 

Second, when $m=2$,  we have for $I_1^l, l=1,2,...,2n$, by the structure condition \eqref{eq:str cond diff} that
\[ | I_2^l|\le c(\beta+2)\int_\Omega\eta^{\beta+2}v^{\beta+1}
\weight |\X u|^3|\X\X u\|\X v|\, dx,
\]
from which it follows by Young's inequality that
\begin{equation}\label{est:I2}
\begin{aligned}
| I_2^l|\le &\, \frac{1}{36n}\int_\Omega \eta^{\beta+2}v^{\beta+2}
\weight |\X u|^2|\X\X u|^2\,dx\\
 & + c(\beta+2)^2\int_\Omega \eta^{\beta+2} v^{\beta}
\weight |\X u|^4|\X v|^2\, dx.
\end{aligned}
\end{equation}
This proves (\ref{est1}) for $I_2^l, l=1,2,...,2n$.

Third, when $m=3$, we use 
\begin{equation*}
\begin{aligned}
\big| X_i & \big(  \eta^{\beta+2}v^{\beta+2}|\X u|^2 X_lu\big)\big|\le  3 \eta^{\beta+2} v^{\beta+2} |\X u|^2 |\X\X u|\\
& + (\beta+2)\eta^{\beta+1}v^{\beta+2}
|\X u|^3|\X \eta| +(\beta+2)\eta^{\beta+2}v^{\beta+1}
|\X u|^3| \X v|.
\end{aligned}
\end{equation*}
and the structure condition \eqref{eq:str cond diff}, to obtain
\begin{equation*}
\begin{aligned}
| I_3^l| \le &\,  c \int_\Omega \eta^{\beta+2}v^{\beta+2}\weight |\X u|^2|\X \X u| | Tu|\, dx\\
& + c(\beta+2)\int_\Omega\eta^{\beta+1}|\X \eta| v^{\beta+2}\weight |\X u|^3| Tu|\, dx\\
& + c(\beta+2)\int_\Omega \eta^{\beta+2} v^{\beta+1}\weight |\X u|^3|\X v\| Tu|\, dx,
\end{aligned}
\end{equation*}
from which it follows by Young's inequality that
\begin{equation}\label{est:I3}
\begin{aligned}
| I_3^l| \le &\, \frac{1}{36n}\int_\Omega \eta^{\beta+2}v^{\beta+2}
\weight |\X u|^2|\X\X u|^2\,dx\\
&+c \int_\Omega \eta^{\beta+2}v^{\beta+2}\weight |\X u|^2 | Tu|^2\, dx\\
& + c(\beta+2)^2 \int_\Omega\eta^{\beta}|\X \eta|^2 v^{\beta+2}\weight |\X u|^4\, dx\\
& + c(\beta+2)^2\int_\Omega \eta^{\beta+2} v^{\beta}\weight |\X u|^4|\X v|^2\, dx.
\end{aligned}
\end{equation}
This proves (\ref{est1}) for $I_3^l, l=1,2,...,2n$.

Finally, when $m=4$, we prove (\ref{est1}) for $I_4^l$. We consider
only the case $l=1,2,...,n$. The case $l=n+1,n+2,...,2n$ can be treated similarly. 
Let us denote 
\begin{equation}\label{def:w} 
w=\eta^{\beta+2}|\X u|^2 X_lu.
\end{equation}
so that we can write test-function $\varphi=\eta^{\beta+2}v^{\beta+2} |\X u|^2 X_lu$ as
$\varphi=v^{\beta+2}w$. Then, for $I_4^l$ in \eqref{eq1}, we rewrite  $T=X_1X_{n+1}-X_{n+1}X_1$ and use 
integration by parts to obtain
\begin{equation}\label{est:I41}
\begin{aligned}
I_4^l=\int_\Omega T\big(\A_{n+l}(\X u)\big)\varphi\, dx
=\int_\Omega X_1\big(\A_{n+l}(\X u)\big) X_{n+1}\varphi-X_{n+1}\big(\A_{n+l}(\X u)\big)X_1\varphi\, dx.
\end{aligned}
\end{equation}
Using  
$\X \varphi= (\beta+2)v^{\beta+1}w\X v+ v^{\beta+2}\X w$ 
in (\ref{est:I41}), we get 
\begin{equation}\label{est:I42}
\begin{aligned}
I_4^l=&\, (\beta+2)\int_\Omega v^{\beta+1} w\Big(X_1\big(\A_{n+l}(\X u)\big)X_{n+1} v-X_{n+1}\big( \A_{n+l}(\X u)\big)X_1 v\Big)\, dx\\
&+\int_\Omega v^{\beta+2}\Big( X_1\big(\A_{n+l}(\X u)\big)X_{n+1}w-X_{n+1}\big(\A_{n+l}(\X u)\big)X_1w\Big)\, dx\\
= &\, J^l+K^l.
\end{aligned}
\end{equation}
Here we denote the first and the second integral in the right hand side of 
(\ref{est:I41}) by $J^l$ and $K^l$, respectively. Now we estimate 
$J^l$ as follows. From structure condition \eqref{eq:str cond diff} and  (\ref{def:w}) 
\begin{equation*}
| J^l| \le c(\beta+2)\int_\Omega \eta^{\beta+2} v^{\beta+1}
\weight |\X u|^3| \X\X u\|\X v|\, dx,
\end{equation*}
from which it follows by Young's inequality, that
\begin{equation}\label{est:Jl}
\begin{aligned}
| J^l| \le &\, \frac{1}{72n}\int_\Omega \eta^{\beta+2}v^{\beta+2}\weight |\X u|^2|\X\X u|^2\,dx\\
&+c(\beta+2)^2\int_\Omega \eta^{\beta+2} v^{\beta}
\weight |\X u|^4|\X v|^2\, dx.
\end{aligned}
\end{equation}
The above inequality shows that $J^l$ satisfies similar estimate as (\ref{est1}) for all $l=1,2,...,n$.
Now we estimate $K^l$. Integration by parts again, yields 
\begin{equation}\label{est:Kl}
\begin{aligned}
K^l=&\, (\beta+2) \int_\Omega v^{\beta+1} \A_{n+l}(\X u)\Big( X_{n+1} vX_1 w-X_1vX_{n+1}w\Big)\, dx\\
&\quad-\int_\Omega v^{\beta+2}\A_{n+l}(\X u) Tw\, dx\\
= &\, K_1^l+K_2^l.
\end{aligned}
\end{equation}
For $K_1^l$, we have by the structure condition \eqref{eq:str cond diff} that
\begin{equation*}
\begin{aligned}
| K^l_1| \le\ &  c(\beta+2)\int_\Omega \eta^{\beta+2}v^{\beta+1}\weight |\X u|^3|\X\X u\|\X v|\, dx\\
& +c(\beta+2)^2\int_\Omega\eta^{\beta+1}v^{\beta+1}\weight |\X u|^4|\X v\|\X \eta|\, dx
\end{aligned}
\end{equation*}
from which it follows by Young's inequality that
\begin{equation}\label{est:Kl1}
\begin{aligned}
| K^l_1| \le &\, \frac{1}{144n}\int_\Omega \eta^{\beta+2}v^{\beta+2}\weight |\X u|^2|\X\X u|^2\,dx\\
&+c(\beta+2)^2\int_\Omega \eta^{\beta+2} v^{\beta}
\weight |\X u|^4|\X v|^2\, dx\\
& + c(\beta+2)^2\int_\Omega \eta^{\beta}|\X \eta|^2 v^{\beta+2}\weight |\X u|^4\, dx.
\end{aligned}
\end{equation}
The above inequality shows that $K_1^l$ also satisfies similar estimate as (\ref{est1}) for all $l=1,2,...,n$. We continue to estimate $K_2^l$ in (\ref{est:Kl}). Note that
\begin{equation*}
\begin{aligned}
Tw=\, (\beta+2)\eta^{\beta+1}|\X u|^2 X_luT\eta+\eta^{\beta+2}|\X u|^2 X_lTu + \sum_{i=1}^{2n}2\eta^{\beta+2}X_luX_iuX_iTu.
\end{aligned}
\end{equation*}
Therefore we write $K^l_2$ as
\begin{equation*}
\begin{aligned}
K_2^l = &\, -(\beta+2)\int_\Omega\eta^{\beta+1}v^{\beta+2} \A_{n+l}(\X u)|\X u|^2 X_luT\eta\, dx\\
&\, -\int_\Omega \eta^{\beta+2}v^{\beta+2} \A_{n+l}(\X u)|\X u|^2 X_lTu\, dx\\
&-2\sum_{i=1}^{2n}\int_\Omega\eta^{\beta+2}v^{\beta+2} \A_{n+l}(\X u)X_luX_iu X_iTu\, dx.
\end{aligned}
\end{equation*}
For the last two integrals in the above equality, we apply integration by parts 
to get 
\begin{equation*}
\begin{aligned}
K_2^l = &\, -(\beta+2)\int_\Omega\eta^{\beta+1}v^{\beta+2} \A_{n+l}(\X u)|\X u|^2 X_luT\eta\, dx\\
&\, +\int_\Omega X_l\Big(\eta^{\beta+2}v^{\beta+2} \A_{n+l}(\X u)
\vert\X u\vert^2\Big)Tu\, dx\\
&+2\sum_{i=1}^{2n}\int_\Omega X_i\Big(\eta^{\beta+2}v^{\beta+2} \A_{n+l}(\X u)X_luX_iu\Big)Tu\, dx.
\end{aligned}
\end{equation*}
Now we may estimate the integrals in the above equality by the structure condition \eqref{eq:str cond diff}, to obtain the following estimate for $K^l_2$.
\begin{equation*}
\begin{aligned}
| K_2^l|\le &\, c(\beta+2)\int_\Omega \eta^{\beta+1}v^{\beta+2}\weight |\X u|^4| T\eta|\, dx\\
&\, +c\int_\Omega \eta^{\beta+2}v^{\beta+2}\weight |\X u|^2|\X\X u\| Tu|\, dx\\
&\, +c(\beta+2)\int_\Omega \eta^{\beta+2} v^{\beta+1}\weight |\X u|^3|\X v\| Tu|\, dx\\
&\, +c(\beta+2)\int_\Omega \eta^{\beta+1}v^{\beta+2}\weight |\X u|^3|\X\eta\| Tu|\, dx.
\end{aligned}
\end{equation*}
By Young's inequality, we end up with the following estimate for $K_2^l$
\begin{equation}\label{est:K2l}
\begin{aligned}
| K_2^l|\le & \, \frac{1}{144n}\int_\Omega \eta^{\beta+2}v^{\beta+2}\weight |\X u|^2|\X\X u|^2\,dx\\
 & + c (\beta+2)^2\int_\Omega \eta^\beta \big(|\X \eta|^2+\eta| T\eta|\big)v^{\beta+2}
\weight |\X u|^4\, dx\\
& +  c(\beta+2)^2\int_\Omega \eta^{\beta+2} v^ \beta\weight |\X u|^4| \X v|^2\, dx\\
& +
c\int_\Omega\eta^{\beta+2} v^{\beta+2}\weight |\X u|^2| Tu|^2\, dx.
\end{aligned}
\end{equation}
This shows that $K_2^l$ also satisfies similar estimate as (\ref{est1}). Now we combine the estimates (\ref{est:Kl1}) for $K^l_1$ and (\ref{est:K2l}) for $K_2^l$.
Recall that $K^l=K_1^l+K_2^l$ as denoted in (\ref{est:Kl}). We obtain that the following estimate for $K^l$.
\begin{equation}\label{est:Kll}
\begin{aligned}
| K^l|\le & \, \frac{1}{72n}\int_\Omega \eta^{\beta+2}v^{\beta+2}\weight |\X u|^2|\X\X u|^2\,dx\\
 & + c (\beta+2)^2\int_\Omega \eta^\beta \big(|\X \eta|^2+\eta| T\eta|\big)v^{\beta+2}
\weight |\X u|^4\, dx\\
& +  c(\beta+2)^2\int_\Omega \eta^{\beta+2} v^ \beta\weight |\X u|^4| \X v|^2\, dx\\
& +
c\int_\Omega\eta^{\beta+2} v^{\beta+2}\weight |\X u|^2| Tu|^2\, dx.
\end{aligned}
\end{equation}
Recall that $I_4^l=J^l+K^l$. We combine the estimates (\ref{est:Jl}) for $J^l$ and (\ref{est:Kll}) for $K^l$, and we can see that the claimed estimate (\ref{est1}) holds for $I_4^l$ for all $l=1,2,...,n$. We can prove 
(\ref{est1}) similarly for $I_4^l$ for all $l=n+1,n+2,...,2n$. This finishes the proof of the claim (\ref{est1}) for $I_m^l$ for all $l=1,2,...,2n$ and all $m=1,2,3,4$, and hence also the proof of the lemma.
\end{proof}

\begin{proof}[Proof of Lemma \ref{lemma:cacci:k}]\noindent
\\
Recalling \eqref{comparable'}, notice that 
$\tau\mu(r)\leq |\X u|\leq (2n)^\frac{1}{2}\mu(r)$ in $A^+_{k,r}(X_lu)$. 
Then this combined with doubling condition of $g$, implies that  
\begin{equation}\label{comparable}
 \frac{\tau^{g_0}}{(2n)^{1/2}}\F(\mu(r))\le \F(|\Xu|)\le \frac{(2n)^{g_0/2}}{\tau}\F(\mu(r)) \quad \text{in }\, A_{k,r}^+(X_l u).
 \end{equation}
 In the proof, we only consider $l\in\{1,\ldots,n\}$; the proof is similar for $l\in\{n,\ldots,2n\}$. In addition, note that we can also assume $|k|\leq \mu(r_0)$ without loss of generality, to prove \eqref{HDG}. This proof is very similar to that of Lemma 4.3 in \cite{Zhong}. 
 
 Let $\eta\in C^\infty_0(B_{r'})$ is a standard cutoff function such that $\eta = 1$ in $B_{r''}$ and $|\X\eta|\leq 2/(r'-r'')$, we choose 
 $\varphi = \eta^2 (X_lu -k)^+$ as a test function in 
 equation \eqref{eq:Xu} to get 
\begin{equation*}
\begin{aligned}
\int_{B_r} \sum_{i,j}&\eta^2D_j\A_i(\Xu)X_jX_l uX_i((X_lu -k)^+)\, dx\\
=&\ -2\int_{B_r}\sum_{i,j}\eta (X_lu -k)^+ D_j\A_i(\Xu)X_jX_l uX_i \eta\, dx\\
&\ -\int_{B_r} \sum_i D_i\A_{n+l}(\Xu)TuX_i(\eta^2 (X_lu -k)^+)\, dx\\
&\ +\int_{B_r} \eta^2 (X_lu -k)^+ T(\A_{n+l}(\Xu)) \, dx
\end{aligned}
\end{equation*}
Using structure condition \eqref{eq:str cond diff} and Young's inequality, 
we obtain 
\begin{equation}\label{J}
  \begin{aligned}
   \int_{B_r}\eta^2\, \weight &| \X(X_lu-k)^{+}|^2\, dx \\ 
   &\le 
 c\int_{B_r}|
\X\eta|^2\weight |(X_lu-k)^{+}|^2\, dx\\
&\quad +\ c\int_{ A^+_{k,r}}\eta^2\, \weight| Tu|^2\, dx\\
&\quad +\ c\int_{B_r}\eta^2(X_lu-k)^{+}\weight|\X(Tu)| \, dx\\
 &=\, J_1+ J_2 +J_3.
  \end{aligned}
 \end{equation}
Notice that to show \eqref{HDG} from \eqref{J}, 
we need to estimate $J_2$ and $J_3$. First, we estimate $J_2$ 
using H\"older's inequality, \eqref{eq:Tint} and \eqref{comparable} 
as follows.  
\begin{equation}\label{J2}
 \begin{aligned}
  J_2 \,&\le \Big(\int_{B_{r_0/2}} \weight| Tu|^q\,
dx\Big)^{\frac 2 q}\Big(\int_{A_{k,r}^+}\weight\,
dx\Big)^{1-\frac{2}{q}}\\
&\leq c\,r_0^{-2}\mu(r_0)^2\F(\mu(r_0))
|B_{r_0}|^\frac{2}{q}| A^+_{k,r}(X_lu)|^{1-\frac{2} {q}}
 \end{aligned}
\end{equation}
for $c=c(n,g_0,L,q,\tau)>0$.

The estimate of $J_3$ is more involved. We wish to show the following, which 
combined with \eqref{J} and \eqref{J2}, completes the proof of this lemma.
\begin{equation}\label{claim J3}
 J_3 \leq \mathcal M/2 + c\,r_0^{-2}
\mu(r_0)^2\F(\mu(r_0))
|B_{r_0}|^\frac{2}{q}| A^+_{k,r}(X_lu)|^{1-\frac{2} {q}}
\end{equation}
for some $c = c(n,p,L,q,\tau)>0$, where 
\begin{equation}\label{M}
 \begin{aligned}
  \mathcal{M}\, := \int_{B_r}\eta^2 \weight | \X(X_lu-k)^{+}|^2\, dx
  + \int_{B_r}|\X\eta|^2\weight |(X_lu-k)^{+}|^2\, dx. 
 \end{aligned} 
\end{equation}
In order to prove 
the claim \eqref{claim J3}, 
we follow the iteration argument of Zhong \cite{Zhong}. 

For any $\kappa \geq 0$,  
we take $ \eta^2|(X_lu-k)^+|^2|Tu|^\kappa Tu $ as a test function in 
\eqref{eq:Tu} and use structure condition \eqref{eq:str cond diff}, 
to obtain 
\begin{equation*}
 \begin{aligned}
  (\kappa+1)\int_{B_r} \eta^2|(X_iu-& k)^+|^2 \weight|Tu|^\kappa|\X(Tu)|^2 dx \\
   \leq\ c \int_{B_r} &\eta |(X_iu-k)^+|^2\weight
  |Tu|^{\kappa+1}|\X(Tu)||\X\eta|\dx\\
 + c \int_{B_r} &\eta^2|(X_lu-k)^+|\weight
  |Tu|^{\kappa+1}|\X(Tu)||\X (X_lu-k)^+|\dx
 \end{aligned}
\end{equation*}
Using Cauchy-Schwartz inequality on the above, we obtain 
\begin{equation}\label{it step}
 \begin{aligned}
  \int_{B_r}\eta^2|(X_iu-& k)^+|^2 \weight|Tu|^\kappa|\X(Tu)|^2 dx\\ 
  \leq c\, \mathcal{M}^\frac{1}{2}&\Big(\int_{B_r}\eta^2
|(X_iu- k)^+|^2 \weight|Tu|^{2\kappa+2}|\X(Tu)|^2 dx\Big)^\frac{1}{2}
 \end{aligned}
\end{equation}
for $c = c(n,g_0,L)>0$ and $\mathcal{M}$ as defined in \eqref{M}. 
Now we iterate \eqref{it step}, choosing the sequence 
$ \kappa_m = 2^m-2$ for $m\in \N$. 
For any $m\geq 1$, we set 
$$ a_m  = \int_{B_r}\eta^2 \weight|(X_lu-k)^+|^2|Tu|^{\kappa_m}|\X(Tu)|^2 dx $$
and obtain 
$ a_1 \,\leq\, (c\,\mathcal{M})^\frac{1}{2} a_2^\frac{1}{2}\,\leq\, \ldots
\,\leq (c\,\mathcal{M})^{(1-\frac{1}{2^m})}\,a_{m+1}^\frac{1}{2^m}$, for every $m\in \N$. Now, for 
some large enough $m$ to be chosen later, we estimate $a_{m+1}$.
Recalling, 
$|k|\leq \mu(r_0)$ and using Corrolary \ref{cor:TX}, we obtain  
\begin{equation}\label{2 a}
\begin{aligned}
a_{m+1}  &\leq c\,\mu(r_0)^2\int_{B_{r_0/2}} 
\weight | Tu|^{\kappa_{m+1}}|\X(Tu)|^2\, dx\\
  &\le c\,r_0^{-(\kappa_{m+1}+4)}\int_{B_{r_0}}
  \weight |\X u|^{\kappa_{m+1}+2} \, dx
\end{aligned}
\end{equation}
for some $c=c(n,g_0,L, m)>0$. Hence, we get 
\begin{equation}\label{est a}
 a_{m+1}  \leq c\,r_0^{-(\kappa_{m+1}+4)}
 \F(\mu(r_0))
 \mu(r_0)^{\kappa_{m+1}+4}\, |B_{r_0}|.
\end{equation}
Now we go back to the estimate of $J_3$.
From 
H\"{o}lder's inequality and \eqref{comparable}, 
\begin{equation*}
 \begin{aligned}
  J_3  &\leq c \Big(\int_{B_r}\eta^2  \weight |(X_lu-k)^+|^2 |\X(Tu)|^2\dx\Big)^{\frac{1}{2}} 
  \Big(\int_{A_{k,r}^+} \weight dx\Big)^\frac{1}{2}\\
  &\leq c\,a_1^{1/2}\F(\mu(r_0))^{1/2}
  |A_{k,r}^+(X_lu)|^{1/2}.
 \end{aligned}
\end{equation*}
for $c=c(n,g_0,L,\tau)>0$. We continue further, 
using the iteration to estimate $a_1^{1/2}$ in terms of $a_{m+1}$ 
and $\mathcal{M}$. Then we use \eqref{est a} and obtain 
\begin{equation*}
  \begin{aligned}
   J_3 &\leq c\,\mathcal{M}^{\frac{1}{2}(1-\frac{1}{2^m})}
   \,a_{m+1}^\frac{1}{2^{m+1}}\F(\mu(r_0))^{1/2}
  |A_{k,r}^+(X_lu)|^\frac{1}{2}\\
  &\leq \frac{c}{r_0^{(1+\frac{1}{2^m})}}\mathcal{M}^{\frac{1}{2}(1-\frac{1}{2^m})}
  \F(\mu(r_0))^{\frac{1}{2}(1+\frac{1}{2^m})}
  \mu(r_0)^{(1+\frac{1}{2^m})}|B_{r_0}|^\frac{1}{2^{m+1}}|A_{k,r}^+(X_lu)|^\frac{1}{2}
  \end{aligned}
 \end{equation*}
Using Young's inequality on the above, we finally obtain
\begin{equation}\label{J3}
 J_3 \leq \mathcal M/2 + c\,r_0^{-2}
\F(\mu(r_0))\mu(r_0)^2
|B_{r_0}|^\frac{1}{2^m+1}| A^+_{k,r}(X_lu)|^\frac{2^m}{2^m+1}
\end{equation}
for some $c = c(n,g_0,L,\tau,m)>0$. The claim \eqref{claim J3} follows 
immediately from \eqref{J3}, with the choice of 
$m = m(q)\in\N$ such that $2^m/(2^m+1)\,\geq\, 1-2/q$. 
This completes the proof.
\end{proof}

\section*{Appendix II}
Here we provide an outline of the proof of Lemma \ref{cor:Tu:high} for the reader's convenience. It requires 
some Caccioppoli type estimates of horizontal and vertical derivatives, similar to those in \cite{Zhong}. The proof of Lemma \ref{cor:Tu:high} shall follow in the end. 

The following Lemma is similar to Lemma 3.4 in \cite{Zhong} and Lemma 2.6 in \cite{Muk-Zhong}. The proof is similar 
and easier than the proof of Lemma \ref{lem:start} in Appendix I, so we omit it. 
We refer the reader to \cite{Muk-Zhong} for some remarks on the proof of 
Lemma 2.6 in it. 
\begin{Lem}\label{caccioppoli:horizontal:sigma}
For any $\beta\ge 0$ and $\eta\in C^\infty_0(\Omega)$, there exists 
$c= c(n,g_0,L)>0$ such that 
\begin{equation*}
\begin{aligned}
\int_{\Omega}\eta^2\, \weight |\X u|^\beta|\X\X u|^2\, dx\le&
\ c\int_\Omega (| \X \eta|^2+\eta| T\eta|)\weight |\X u|^{\beta+2}\, dx\\
&+c(\beta+1)^4\int_\Omega\eta^2\, \weight |\X u|^\beta| Tu|^2\, dx.
\end{aligned}
\end{equation*}
\end{Lem}

The following lemma is similar to Lemma 3.5 of \cite{Zhong}.

\begin{Lem}\label{caccioppoli:horizontal:T}
For any $\beta\ge 2$ and all non-negative $\eta\in C^\infty_0(\Omega)$, we have
\begin{equation*}
\begin{aligned}
\int_\Omega \eta^{\beta+2}&\weight| Tu|^{\beta}|\X\X u|^2\, dx 
\le
c(\beta+1)^2\|  \X\eta\| _{L^\infty}^2\int_\Omega\eta^\beta
\weight |\X u|^2  | Tu|^{\beta-2}| \X\Xu|^2\, dx,
\end{aligned}
\end{equation*}
for some constant $c=c(n,g_0,L)>0$.
\end{Lem}

\begin{proof}
Note that have the following identity for any $\varphi\in C^\infty_0(\Omega)$, which can be easily obtained using 
$X_l\varphi$ as a test function in equation \eqref{eq:maineq} (see the proof of Lemma 3.5 in \cite{Muk0}).
\begin{equation}\label{weak1}
\int_\Omega\sum_{i=1}^{2n} X_l(\A_i(\Xu)X_i\varphi)\, dx =\int_\Omega T(\A_{n+l}(\Xu))\varphi\, dx
\end{equation}
Let $\eta\in C^\infty_0(\Omega)$ be a non-negative cut-off function. Fix any 
$l\in \{ 1,2,\ldots,n\}$ and $\beta\ge 2$, let $\varphi=\eta^{\beta+2}| Tu|^\beta X_lu$.
We use $\varphi$ as a test function in (\ref{weak1}).
Note that 
\[
X_i\varphi=\eta^{\beta+2} | Tu|^\beta X_iX_l u+\beta \eta^{\beta+2}| Tu|^{\beta-2}
TuX_l uX_i(Tu)+(\beta+2)\eta^{\beta+1} X_i\eta| Tu|^\beta X_l u
\]
and that $X_{n+l} X_l=X_lX_{n+l}-T$. Using these, we
obtain 
\begin{equation}\label{weak2}
\begin{aligned}
\int_{\Omega}\sum_i\eta^{\beta+2}| Tu|^\beta &X_l(\A_i(\Xu))X_lX_iu\,
dx=\int_\Omega \eta^{\beta+2} X_l(\A_{n+l}(\X u)) | Tu|^\beta Tu\, dx\\
&-(\beta+2)\int_{\Omega}\sum_i\eta^{\beta+1} | Tu|^\beta X_l (\A_i(\Xu))X_luX_i\eta\, dx\\
&+\int_{\Omega} \eta^{\beta+2}T\left(\A_{n+l}(\Xu)\right)| Tu|^\beta X_lu\, dx.\\
&-\beta\int_\Omega \sum_i \eta^{\beta +2} | Tu|^{\beta-2} Tu X_lu X_l(\A_i(\X u))X_i Tu\, dx\\
=&\ I_1+I_2+I_3+I_4.
\end{aligned}
\end{equation}
We will estimate both sides of (\ref{weak2}) as follows. For the
left hand side, the structure condition \eqref{eq:str cond diff} implies that
\[
\int_{\Omega}\sum_i\eta^{\beta+2}| Tu|^\beta X_l(\A_i(\Xu))X_lX_iu\,
dx\ge \int_{\Omega} \eta^{\beta+2}\weight| Tu|^\beta | X_l\Xu|^2\, dx.
\]

For the right hand side, we will show that for each item,
the following estimate is true.
\begin{equation}\label{claim1}
\begin{aligned}
| I_k|\ \le &\ c\tau\int_{\Omega}\eta^{\beta+2}\weight| Tu|^{\beta}|\X\X u|^2\, dx\\
&+\ \frac{c(\beta+1)^2\|  \X\eta\| _{L^\infty}^2}{\tau}
\int_{\Omega}\eta^{\beta}\weight |\X u|^2 | Tu|^{\beta-2}| \X\X u|^2\, dx,
\end{aligned}
\end{equation}
for $k=1,2,3, 4$, where $c= c(n,p,L)>0$ and $\tau>0$ is a constant. 
By the above estimates for both sides of (\ref{weak2}), we end up with
\begin{equation*}
\begin{aligned}
\int_{\Omega} \eta^{\beta+2}\weight &| Tu|^\beta | X_l\Xu|^2\, dx\ \le\  c\tau\int_{\Omega}\eta^{\beta+2}\weight| Tu|^{\beta}|\X\X u|^2\, dx\\
&+\ \frac{c(\beta+1)^2\|  \X\eta\| _{L^\infty}^2}{\tau}
\int_{\Omega}\eta^{\beta}\weight |\X u|^2 | Tu|^{\beta-2}| \X \X u|^2\, dx.
\end{aligned}
\end{equation*}
The above inequality is true for all $l=1,2,\ldots, n$. Similarly,
we can prove that it is true also for all $l=n+1,\ldots, 2n$. 
Now, by choosing $\tau>0$ small enough,
we complete the proof of the lemma, assuming the proof of (\ref{claim1}).

To prove (\ref{claim1}), we start with $I_4$. By structure condition \eqref{eq:str cond diff} and Young's inequality
\begin{equation*}
\begin{aligned}
| I_4|\le&\, c\beta \int_\Omega
\eta^{\beta+2}\weight |\X u|| Tu|^{\beta-1}| X_l \X u\| \X (Tu)|\, dx\\
\le &\ \frac{\tau}{\|  \X\eta\| _{L^\infty}^2}\int_\Omega \eta^{\beta+4}\weight
| Tu|^\beta|\X (Tu)|^2\,
dx\\
&+\ \frac{c\beta^2\|  \X\eta\| _{L^\infty}^2}{\tau}
\int_{\Omega}\eta^{\beta}\weight |\X u|^2 | Tu|^{\beta-2}| X_l \X u|^2\, dx.
\end{aligned}
\end{equation*}
We then apply  Lemma \ref{caccioppoli:T} to estimate the first integral in the right hand side.
\begin{equation}\label{weak3}
\begin{aligned}
\int_{\Omega}\eta^{\beta+4}\weight 
| Tu|^\beta|\X (Tu)|^2\,
dx\le\ c\int_{\Omega} \eta^{\beta+2}|\X \eta|^2\weight|
Tu|^{\beta+2}\, dx.
\end{aligned}
\end{equation}
Using this, we obtain
\begin{equation}\label{weak3prime}
\begin{aligned}
| I_4|
 &\le  c\tau\int_\Omega \eta^{\beta+2}\weight
| Tu|^{\beta+2}\,
dx\\
&+\frac{c\beta^2\|  \X\eta\| _{L^\infty}^2}{\tau}
\int_{\Omega}\eta^{\beta}\weight |\X u|^2 | Tu|^{\beta-2}| X_l \X u|^2\, dx.
\end{aligned}
\end{equation}
Since $| Tu|\le 2| \X\X u|$, (\ref{weak3prime}) implies that $I_4$ satisfies (\ref{claim1}).

To prove that (\ref{claim1}) holds for $I_1$,
integration by parts yields
\begin{equation*}
\begin{aligned}
I_1=&-\int_{\Omega} \A_{n+l}(\Xu)X_l(\eta^{\beta+2}| Tu|^\beta
Tu)\, dx\\
=&-(\beta+1)\int_{\Omega} \eta^{\beta+2} | Tu|^\beta \A_{n+l}(\Xu)X_l(Tu)\, dx\\
&-(\beta+2)\int_{\Omega}\eta^{\beta+1}
\A_{n+l}(\Xu) X_l\eta | Tu|^\beta Tu\, dx
=I_{11}+I_{12}.
\end{aligned}
\end{equation*}
We will show that (\ref{claim1}) holds for both $I_{11}$ and $I_{12}$.
For $I_{11}$, by structure condition \eqref{eq:str cond diff} and Young's inequality,
\begin{equation*}
\begin{aligned}
| I_{11}|\ \le &\ c(\beta+1)
\int_{\Omega}\eta^{\beta+2}\weight |\X u| | Tu|^\beta| \X (Tu)|\,
dx\\
\le &\ \frac{\tau}{\|  \X\eta\| _{L^\infty}^2}\int_{\Omega}\eta^{\beta+4}\weight
| Tu|^\beta|\X (Tu)|^2\,
dx\\
&+\ \frac{c(\beta+1)^2 \|  \X\eta\| _{L^\infty}^2}{\tau}
\int_{\Omega}\eta^{\beta}\weight |\X u|^2 | Tu|^\beta\, dx,
\end{aligned}
\end{equation*}
which, together with (\ref{weak3}) and the fact  $| Tu|\le 2| \X\X u|$,  implies that (\ref{claim1}) holds for $I_{11}$.
For $I_{12}$, (\ref{claim1}) follows from
\[
| I_{12}| \le  c(\beta+2)\int_{\Omega}\eta^{\beta+1}|\X\eta|\weight |\X u|| Tu|^{\beta+1}\, dx,
\]
and Young's inequality. This proves that $I_1$ satisfies (\ref{claim1}). 

For $I_2$, we have by structure condition \eqref{eq:str cond diff}, that 
\[
|I_2 |\le c(\beta+2)\int_{\Omega}\eta^{\beta+1}|\X\eta|\weight |\X u|| Tu|^\beta| X_l\X u|\, dx,
\]
from which, together with Young's inequality and $| Tu|\le 2| \X\X u|$, 
(\ref{claim1}) for $I_2$ follows. 

Finally,  $I_3$ has the same bound as that of $I_{11}$. We have 
\[
| I_3|\le c \int_\Omega \eta^{\beta+2}\weight |\X u|| Tu  |^\beta| \X (Tu)|\, dx, 
\]
thus $I_3$ satisfies (\ref{claim1}), too. This completes the proof of (\ref{claim1}), and hence that of the lemma.
\end{proof}

The following corollary is easy to prove, by using H\"older's inequality on Lemma \ref{caccioppoli:horizontal:T}.
\begin{Cor}\label{cor1}
For any $\beta\ge 2$ and all non-negative $\eta\in C^\infty_0(\Omega)$, we have
\begin{equation*}
\begin{aligned}
\int_\Omega \eta^{\beta+2}\weight| Tu|^\beta| \X\X u|^2\, dx\le c^{\frac{\beta}{2}}(\beta+1)^{\beta}\| \X\eta\| _{L^\infty}^\beta
\int_\Omega \eta^2\, \weight |\X u|^{\beta}  |\X\X u|^2\, dx,
\end{aligned}
\end{equation*}
where $c=c(n,g_0,L)>0$.
\end{Cor}
Now, we resate Lemma \ref{cor:Tu:high} as follows. 
\begin{Lem}
For any $\beta\geq 2$ and all non-negative $\eta\in C^\infty_0(\Omega)$, we have that
\begin{equation}\label{eq:Tint diff}
\int_\Omega\eta^{\beta+2}\,\weight| Tu|^{\beta+2}\, dx \le c(\beta)K^{\frac{\beta+2}{2}}
\int_{\supp(\eta)}\weight |\X u|^{\beta+2}\, dx,
\end{equation}
where $K=\| \X\eta\|_{L^\infty}^2+\|\eta
T\eta\|_{L^\infty}$ and  $c(\beta)=c(n,g_0,L,\beta)>0$. 
\end{Lem}
\begin{proof}
First, we show the following claim. For 
all non-negative $\eta\in C^\infty_0(\Omega)$, we show that 
\begin{equation}\label{eq:he}
\int_\Omega \eta^2\, \weight |\X u|^\beta |\X \X u|^2\, dx\le c(\beta+1)^{10}K\int_{\supp(\eta)} \weight |\X u|^{\beta+2}\, dx,
\end{equation}
where $K=\|  \X\eta\| _{L^\infty}^2+\| \eta
T\eta\| _{L^\infty}$ and $c= c(n,g_0,L)>0$.  
Then, \eqref{eq:Tint diff} follows easily from 
Corollary \ref{cor1}, the estimate \eqref{eq:he} and the fact that
$|Tu|\leq 2|\X\X u|$. Thus, we are only left with the 
proof of the claimed estimate \eqref{eq:he}. 

To prove \eqref{eq:he}, notice that 
by Lemma \ref{caccioppoli:horizontal:sigma}, we only need to estimate
the integral $$\int_\Omega \eta^2\, \weight |\X u|^\beta | Tu|^2\, dx.$$ 
From H\"older's inequality, we have 
\begin{equation*}
\begin{aligned}
\int_\Omega\eta^2\weight |\X u|^\beta| Tu|^2\, dx
 \le \Big(\int_\Omega \eta^{\beta+2}\weight 
|Tu|^{\beta+2}\, dx\Big)^{\frac{2}{\beta+2}}\Big(\int_{\supp(\eta)}\weight |\X u|^{\beta+2}\, dx\Big)^{\frac{\beta}{\beta+2}}.
\end{aligned}
\end{equation*}
Then, using $|Tu|\leq |\XX u|$ on the above, we obtain the following from 
Lemma \ref{caccioppoli:horizontal:sigma}, 
Corollary \ref{cor1} and Young's inequality, 
\[
\int_\Omega \eta^2\, \weight |\Xu|^\beta |\XX u|^2\, dx
\le c(\beta+1)^{\frac{4(\beta+2)}{\beta}+2} K\int_{\supp(\eta)} \weight |\X u|^{\beta+2}\, dx,
\]
which proves the claim \eqref{eq:he} and hence completes the proof. 
\end{proof}



\bibliographystyle{plain}
\bibliography{LipH}

\end{document}